\documentclass[10pt]{amsart}
\usepackage{amsfonts}
\usepackage{ifthen}
\usepackage{amsthm}
\usepackage{amsmath}
\usepackage{graphicx}
\usepackage{amscd,amssymb,amsthm}
\usepackage{hyperref}
\usepackage{epstopdf}

\newcounter{minutes}
\divide\time by 60
\newcounter{hours}
\multiply\time by 60 \addtocounter{minutes}{-\time}

\newcommand{\real}{\operatorname{Re}}

\newtheorem{theorem}{Theorem}
\newtheorem{lemma}{Lemma}

\newtheorem{problem}{Open Problem}

\keywords{Bessel functions of the first kind, modified Bessel functions of the first kind; product of Bessel functions; cross-product of Bessel
functions; univalent, starlike, convex functions; radius of starlikeness;
radius of convexity; zeros of a cross product of Bessel functions; Laguerre-P%
\'{o}lya class of entire functions; Laguerre inequality; interlacing
property of zeros; distribution of zeros of entire functions; zeros of hypergeometric polynomials; Fourier critical points.}

\subjclass[2010]{30D15, 30C15, 30C45, 33C10}

\title{Products of Bessel and modified Bessel functions}

\author[\'A. Baricz]{\'{A}rp\'{a}d Baricz$^{\bigstar}$}
\address{Department of Economics, Babe\c{s}-Bolyai University,
Cluj-Napoca, Romania}
\address{Institute of Applied Mathematics, \'{O}buda University, Budapest, Hungary}
\email{bariczocsi@yahoo.com}

\author[R. Sz\'asz]{R\'obert Sz\'asz}
\address{Department of Mathematics and Informatics, Sapientia Hungarian University of Transylvania, T\^argu-Mure\c{s}, Romania}
\email{rszasz@ms.sapientia.ro}

\author[N. Ya\u{g}mur]{Nihat Ya\u{g}mur}
\address{Lalapasa mah, Ali Riza bey sokak, no. 24, Yakutiye Erzurum, Turkey}
\email{nhtyagmur@gmail.com}

\thanks{$^{\bigstar}$The research of \'A. Baricz was supported by a research grant of the Babe\c{s}-Bolyai University for young researchers with
project number GTC-31777.}

\begin{document}

\def\thefootnote{}
\footnotetext{ \texttt{File:~\jobname .tex,
          printed: \number\year-\number\month-\number\day,
          \thehours.\ifnum\theminutes<10{0}\fi\theminutes}
} \makeatletter\def\thefootnote{\@arabic\c@footnote}\makeatother

\maketitle

\begin{center}
Dedicated to Beril K\"ubra, Bor\'oka and Kopp\'any
\end{center}

\begin{abstract}
The reality of the zeros of the product and cross-product of Bessel and modified Bessel functions of the
first kind is studied. As a consequence the reality of the zeros of two hypergeometric polynomials is obtained together with the number of the Fourier critical points of the normalized forms of the product and cross-product of Bessel functions. Moreover, the interlacing properties of the real zeros of these products of Bessel functions and their derivatives are also obtained. As an application some geometric properties of the normalized forms of the cross-product and product of Bessel and modified Bessel functions of the first kind are studied. For the
cross-product and the product three different kind of normalization are investigated and for each of
the six functions the radii of starlikeness and convexity are precisely determined by using their
Hadamard factorization. For these radii of starlikeness and convexity tight lower and upper bounds are given via Euler-Rayleigh inequalities. Necessary and sufficient conditions are also given for the parameters such
that the six normalized functions are starlike and convex in the open unit disk. The properties and the characterization of
real entire functions from the Laguerre-P\'{o}lya class via hyperbolic
polynomials play an important role in this paper. Some open problems are also stated, which may be of interest for further research.
\end{abstract}

\tableofcontents

\section{Introduction}

Bessel and modified Bessel functions are important functions of mathematical physics. They appear frequently in problems of applied mathematics and their properties were studied extensively by many researchers. In \cite{ben} Ashbaugh and Benguria presented an alternative proof to Rayleigh's conjecture that among all clamped plates of a given area, the circular one gives the lowest principal frequency. Moreover, based on explicit properties of Bessel
functions they generalized the result to the analogous case in three dimensions. Motivated by their appearance as eigenvalues in the clamped plate problem for the ball, Ashbaugh and Benguria \cite{ben} have conjectured that the positive zeros of the
function $\mathbf{\Phi }_{{\nu }},$ defined by
$$\mathbf{\Phi }_{{\nu }}(z)=J_{{\nu }}(z)I_{{\nu }}^{\prime }(z)-J_{{\nu }}^{\prime }(z)I_{{\nu }}(z),$$
increase with ${\nu }$ on $\left[-\frac{1}{2},\infty\right),$ where $J_{{\nu }}$ and $I_{{\nu }}$ stand for the Bessel and modified Bessel functions of the first kind (see \cite{OL}). Lorch \cite{Lorch} verified the above conjecture of Ashbaugh and Benguria and presented some
other properties of the zeros of\ the function $\mathbf{\Phi }_{{\nu }}.$
In \cite{HAT} the proof made by Lorch was modified to ${\nu }\in
(-1, $ $\infty )$, and necessary and sufficient
conditions were deduced for the close-to-convexity of a normalized form of the functions $\mathbf{\Phi }_{{\nu
}}$ and $\mathbf{\Pi }_{{\nu }},$ and their derivatives. Moreover, very recently Baricz et al. \cite{cross}
were interested on the monotonicity patterns for the cross-product of Bessel and
modified Bessel functions, by showing for example that the positive zeros of the cross-product and of the Dini function $z\mapsto (1-\nu)J_{\nu}(z)+zJ_{\nu}'(z)$ are interlacing. In \cite{cross} one of the key tools in the proofs of the main results it was the fact that the zeros of the cross-product are increasing with $\nu.$ Motivated by the above results and by the importance of the cross-product $\mathbf{\Phi }_{{\nu }},$ in this paper our aim is to present an exhaustive study on the real zeros and geometric properties of the product $$\mathbf{\Pi }_{{\nu }}(z)=J_{{\nu }}(z)I_{{\nu }}(z)$$ and the cross-product $\mathbf{\Phi }_{{\nu }}.$ By using among others the above monotonicity property of the zeros of the cross-product, our aim is to determine the radii of starlikeness and convexity of the normalized forms of the functions $\mathbf{\Phi }_{{\nu }}$
and $\mathbf{\Pi }_{{\nu }}$ together with the complete characterization of the starlikeness and convexity of these normalized forms. It is worth to mention that some similar results were obtained recently in
the papers \cite{bdm} and \cite{BY}, in which the radii of starlikeness and convexity of
$q$-Bessel functions as well as for Lommel and Struve functions of the first kind were determined. We also mention
that by using the corresponding recurrence relations for the Bessel and modified Bessel functions of the first kind,
then the cross-product can be rewritten as $$\mathbf{\Phi}_{\nu}(z)=J_{{\nu }+1}(z)I_{{\nu }}(z)+J_{{\nu }}(z)I_{{\nu }+1}(z).$$ Moreover, we know that
if $z\in {\mathbb{C}}$ and ${\nu }\in {\mathbb{C}}$ such
that ${\nu \neq -1,-2,{\dots}}$ then the functions $\mathbf{\Phi }_{{\nu }}$ and $%
\mathbf{\Pi }_{{\nu }}$ can be written as follows (see \cite[p. 148]{Wat} and \cite%
{HAT}):
\begin{equation}
\mathbf{\Phi }_{{\nu }}(z)=2\sum_{n\geq0}\frac{\left( -1\right) ^{n}\left( \frac{z}{2}\right)
^{2{\nu }+4n+1}}{n!\Gamma \left( {\nu }+n+1\right) \Gamma \left( {\nu }%
+2n+2\right) },  \label{fi}
\end{equation}
\begin{equation}
\mathbf{\Pi }_{{\nu }}(z)=\sum_{n\geq0}\frac{\left( -1\right)
^{n}\left( \frac{z}{2}\right) ^{2{\nu }+4n}}{n!\Gamma \left( {\nu }%
+n+1\right) \Gamma \left( {\nu }+2n+1\right) }.  \label{pi}
\end{equation}

It is important to mention here that the problem to consider the radius of starlikeness of the normalized Bessel functions was first studied by Brown \cite{brown}, who used the methods of Nehari \cite{nehari} and Robertson \cite{robertson}. The key tool in Brown's proofs was the fact that the Bessel function of the first kind is a particular solution of the homogeneous Bessel differential equation. For related interesting results the reader can refer to \cite{brown2,brown3,todd,merkes,robertson,wilf} and to the references therein, and for more details can refer to \cite{sz}. Our approach follows the method from \cite{BAS} and it is completely different than of Brown \cite{brown,brown2,brown3}, Nehari \cite{nehari} and Robertson \cite{robertson}. In order to achieve our goal we also study the distribution of the real zeros of the product and cross-product of Bessel and modified Bessel functions of the first kind. Although these results were used to obtain some sharp geometric properties for the products, the results on the zeros can be
of independent interest and may be useful for people who are using special functions in their research.

The paper is organized as follows: the next section contains the main results of the paper. Section 3 is devoted for
preliminary results and their proofs, while section 4 contains the proofs of the main results. We note that the preliminary
results deduced in section 3 are also of independent interest and can be useful in problems related to the product and
cross-product of Bessel functions.

\section{Properties of the product and cross-product of Bessel functions}
\setcounter{equation}{0}

In this section our aim is to study the distribution of the real zeros of the product and cross-product of Bessel and modified Bessel functions of the first kind. This is analogous to the distribution of the real zeros of Bessel functions of the first kind, deduced by Lommel, and later by Hurwitz. As a consequence some geometric properties of the normalized forms of the cross-product and product of Bessel and modified Bessel functions of the first kind are studied. For the cross-product and the product three different kind of normalization are investigated and for each of
the six functions the radii of starlikeness and convexity are precisely determined by using their
Hadamard factorization. For these radii of starlikeness and convexity tight lower and upper bounds are given via Euler-Rayleigh inequalities. Necessary and sufficient conditions are also given for the parameters such
that the six normalized functions are starlike and convex in the open unit disk.

\subsection{Reality of the zeros of $\mathbf{\Phi}_{\nu}$ and $\mathbf{\Pi}_{\nu}$} The following result is very important in the study of the geometric properties of the product and cross-product of Bessel and modified Bessel functions of the first kind, however it is also of independent interest. This result is based on general results on zeros of entire functions by Runckel \cite{runckel} and it is analogous to the celebrated result of von Lommel, which states that when $\nu>-1$ all the zeros of $J_{\nu}$ are real. The reality of the zeros of $\mathbf{\Pi}_{\nu}$ is equivalent to this result, however the proof stated in this paper is completely different from Lommel's approach.

\begin{theorem}\label{thzeros}
 If $\nu>-1,$ then the functions $\mathbf{\Phi}_{\nu}$ and $\mathbf{\Pi}_{\nu}$ have infinitely many zeros, which are all real.
\end{theorem}

\subsection{The reality of the zeros of two hypergeometric polynomials} The above result can be written equivalently as follows. The proof of this result was motivated by the paper of Dimitrov and Ben Cheikh \cite{dimitrov} in which an alternative proof of von Lommel's result on the reality of the zeros of Bessel functions was given.

\begin{theorem}\label{thhyper}
If $\nu>-1$ and $n\in\mathbb{N},$ then all zeros of the hypergeometric polynomials
$${}_1F_3\left(-n;\nu+1,\frac{\nu}{2}+1,\frac{\nu}{2}+\frac{3}{2};z\right)\ \ \ \mbox{and}\ \ \ {}_1F_3\left(-n;\nu+1,\frac{\nu}{2}+\frac{1}{2},\frac{\nu}{2}+1;z\right)$$
are real.
\end{theorem}

\subsection{The Fourier critical points of the normalized forms of $\mathbf{\Phi}_{\nu}$ and $\mathbf{\Pi}_{\nu}$} To define the notion of the Fourier critical point let $f$ be a real entire function defined in an open interval $(a,b)\subset\mathbb{R}.$ Let $l\in\mathbb{N}$ and suppose that $c\in(a,b)$ is a zero of $f^{(l)}(x)$ of multiplicity $m\in\mathbb{N},$ that is, $f^{(l)}(c)=\dots=f^{(l+m-1)}(c)=0$ and $f^{(l+m)}(c)\neq0.$ Now, let $k=0$ if $f^{(l-1)}(c)=0,$ otherwise let
$$k=\left\{\begin{array}{ll}{m}/{2},& \mbox{if}\ m\ \mbox{is even},\\
{(m+1)}/{2},& \mbox{if}\ m\ \mbox{is odd and}\ f^{(l-1)}(c)f^{(l+m)}(c)>0,\\
{(m-1)}/{2},& \mbox{if}\ m\ \mbox{is odd and}\ f^{(l-1)}(c)f^{(l+m)}(c)<0. \end{array}\right.$$
Then we say that $f^{(l)}(x)$ has $k$ critical zeros and $m-k$ noncritical zeros at $x = c.$ A point in $(a,b)$ is said to be a Fourier critical point of $f$ if some derivative of $f$ has a critical zero at the point. For more details on Fourier critical points we refer to \cite{kim}. Now, we consider the real entire functions
$$\lambda_{\nu}^{\mathbf{\Phi}}(z)=\sum_{n\geq0}\frac{1}{\Gamma \left( {\nu }+n+1\right) \Gamma \left( {\nu }%
+2n+2\right)}\frac{z^n}{n!}$$
and
$$\lambda_{\nu}^{\mathbf{\Pi}}(z)=\sum_{n\geq0}\frac{1}{\Gamma \left( {\nu }%
+n+1\right) \Gamma \left( {\nu }+2n+1\right) }{\frac{z^n}{n!}},$$
which satisfy
$$2\left(\frac{z}{2}\right)^{2\nu+1}\lambda_{\nu}^{\mathbf{\Phi}}\left(-\frac{z^4}{16}\right)=\mathbf{\Phi}_{\nu}(z)\ \ \ \mbox{and}\ \ \
\left(\frac{z}{2}\right)^{2\nu}\lambda_{\nu}^{\mathbf{\Pi}}\left(-\frac{z^4}{16}\right)=\mathbf{\Pi}_{\nu}(z).$$
Note that the functions $\lambda_{\nu}^{\mathbf{\Phi}}$ and $\lambda_{\nu}^{\mathbf{\Pi}}$ have growth order $\frac{1}{4}$ and consequently they are of genus $0.$ Thus, according to the fact that
every real entire function of genus $0$ has just as many Fourier critical points as couples of nonreal zeros (see \cite[Theorem 4.1]{kim}), Theorem \ref{thzeros} can be written equivalently as follows.

\begin{theorem}\label{thfourier}
If $\nu>-1,$ then $\lambda_{\nu}^{\mathbf{\Phi}}$ and $\lambda_{\nu}^{\mathbf{\Pi}}$ have no Fourier critical points.
\end{theorem}

It is worth to mention that the result applied above on the equality of the number of Fourier critical points and number of couples of nonreal zeros implies actually a longstanding conjecture, the so-called Fourier-P\'olya conjecture, which states that a real integral function of genus $0$ has just as many Fourier critical points as couples of imaginary zeros. For more details we refer to the paper \cite{kim} and to the references therein.

Motivated by Hurwitz theorem on the distribution of the real and nonreal zeros of Bessel functions of the first kind and its alternative proof by Ki and Kim \cite{kim} it is natural to ask about the distribution of the zeros of the products considered in this paper.

\begin{problem}
When $\nu<-1$ find the number of nonreal zeros of the products $\mathbf{\Phi}_{\nu}$ and $\mathbf{\Pi}_{\nu}$ or equivalently find the number of Fourier critical points of the functions $\lambda_{\nu}^{\mathbf{\Phi}}$ and $\lambda_{\nu}^{\mathbf{\Pi}}.$
\end{problem}

\subsection{The radii of starlikeness of the normalized forms of $\mathbf{\Phi}_{\nu}$ and $\mathbf{\Pi}_{\nu}$} Let $\mathbb{D}_{r}$ be the open disk $\left\{ {z\in \mathbb{C}:\left\vert
z\right\vert <r}\right\} ,$ where $r>0\ $and $\mathbb{D}_{1}=\mathbb{D}.$
As usual, with $\mathcal{A}$ we denote the class of analytic functions $f:%
\mathbb{D}_{r}\rightarrow \mathbb{C}$ which satisfy the usual normalization
conditions $f(0)=f^{\prime }(0)-1=0.$ Let us denote by $\mathcal{S}$ the
class of functions belonging to $\mathcal{A}$ which are univalent in $%
\mathbb{D}_{r}$ and let $\mathcal{S}^{\star }(\alpha )$ be the subclass of $%
\mathcal{S}$ consisting of functions which are starlike of order $\alpha $
in $\mathbb{D}_{r},$ where $0\leq \alpha <1.$ The analytic characterization
of this class of functions is%
\begin{equation*}
\mathcal{S}^{\star }(\alpha )=\left\{ f\in \mathcal{S}\ :\ \real\left(
\frac{zf^{\prime }(z)}{f(z)}\right) >\alpha \ \ \mathrm{for\ all}\
\ z\in \mathbb{D}_{r}\right\}
\end{equation*}%
and we adopt the convention $\mathcal{S}^{\star }=\mathcal{S}^{\star }(0)$.
The real number%
\begin{equation*}
r_{\alpha }^{\star }(f)=\sup \left\{ r>0\ :\ \real\left( \frac{zf^{\prime
}(z)}{f(z)}\right) >\alpha \ \ \mathrm{\ for\ all}\ \ z\in \mathbb{%
D}_{r}\right\}
\end{equation*}%
is called the radius of starlikeness of order $\alpha $ of the function $f.$
It is worth to mention that $r^{\star }(f)=r_{0}^{\star }(f)$ is the largest
radius such that the image region $f(\mathbb{D}_{r^{\star }(f)})$ is a
starlike domain with respect to the origin.

Since neither $\mathbf{\Phi }_{{\nu }},$ nor $\mathbf{\Pi }_{{\nu }}$
belongs to $\mathcal{A}$, first we perform some natural normalization. Of course there exist infinitely many such kind of
normalization, however, we restricted ourselves to the following three kind of normalization of which analogue for Bessel
functions of the first kind were extensively studied in the literature. For $%
{\nu }>-1$ we define three functions originating from $\mathbf{\Phi }_{{\nu }%
}:$
\begin{equation*}
f_{{\nu }}(z)=\left( 2^{2{\nu }}\Gamma \left( {\nu }+1\right) \Gamma \left( {%
\nu }+2\right) \mathbf{\Phi }_{{\nu }}(z)\right) ^{\frac{1}{2{\nu +1}}},%
\text{ \ }{\nu \neq -\frac{1}{2}},
\end{equation*}%
\begin{equation*}
g_{{\nu }}(z)=2^{2{\nu }}z^{-2{\nu }}\Gamma \left( {\nu }+1\right) \Gamma
\left( {\nu }+2\right) \mathbf{\Phi }_{{\nu }}(z)
\end{equation*}%
and
\begin{equation*}
h_{{\nu }}(z)=2^{2{\nu }}z^{-\frac{{\nu }}{2}+\frac{3}{4}}\Gamma \left( {\nu
}+1\right) \Gamma \left( {\nu }+2\right) \mathbf{\Phi }_{{\nu }}(\sqrt[4]{z}%
).
\end{equation*}%
Similarly, we associate with $\mathbf{\Pi }_{{\nu }}$ the functions
\begin{equation*}
u_{{\nu }}(z)=\left( 2^{2{\nu }}\Gamma ^{2}\left( {\nu }+1\right) \mathbf{%
\Pi }_{{\nu }}(z))\right) ^{\frac{1}{2{\nu }}},\text{ \ }{\nu \neq 0},
\end{equation*}%
\begin{equation*}
v_{{\nu }}(z)=2^{2{\nu }}z^{-2{\nu +1}}\Gamma ^{2}\left( {\nu }+1\right)
\mathbf{\Pi }_{{\nu }}(z)
\end{equation*}
and
\begin{equation*}
w_{{\nu }}(z)=2^{2{\nu }}z^{-\frac{{\nu }}{2}{+1}}\Gamma ^{2}\left( {\nu }%
+1\right) \mathbf{\Pi }_{{\nu }}(\sqrt[4]{z}).
\end{equation*}
Clearly the functions $f_{{\nu }}$, $g_{{\nu }}$, $h_{{\nu }}$, $u_{{\nu }}$%
, $v_{{\nu }}$ and $w_{{\nu }}$ belong to the class $\mathcal{A}$. The following main results we establish concern the radii of starlikeness of $\mathbf{\Phi}_{\nu}$ and $\mathbf{\Pi}_{\nu}$,
and read as follows. Throughout in the sequel we suppose that $0\leq \alpha <1.$

\begin{theorem}\label{T1}
\begin{enumerate}
\item[\textbf{a)}] If ${\nu }>-\frac{1}{2},$ then $r_{\alpha }^{\star }(f_{{%
\nu }})=x_{{\nu ,\alpha ,1}}$ where $x_{{\nu ,\alpha ,1}}$ is the smallest
positive root of
\begin{equation*}
r\mathbf{\Phi }_{{\nu }}^{\prime }(r)-\alpha (2{\nu +1)}\mathbf{\Phi }_{{\nu
}}(r)=0.
\end{equation*}%
Moreover, if ${\nu \in \left(-1,-\frac{1}{2}\right)}$ then $r_{\alpha }^{\star }(f_{{\nu }})=x_{{%
\nu ,\alpha }}$ where $x_{{\nu ,\alpha }}$ is the unique root of
the equation
\begin{equation*}
\sqrt{i}r\mathbf{\Phi }_{{\nu }}^{\prime }(\sqrt{i}r)-\alpha (2{\nu +1)}%
\mathbf{\Phi }_{{\nu }}(\sqrt{i}r)=0.
\end{equation*}

\item[\textbf{b)}] If ${\nu }>-1,$ then $r_{\alpha }^{\star }(g_{{\nu }})=y_{{%
\nu ,\alpha ,1}}$ where $y_{{\nu ,\alpha ,1}}$ is the smallest positive root
of the equation
\begin{equation*}
r\mathbf{\Phi }_{{\nu }}^{\prime }(r)-(\alpha +2{\nu )}\mathbf{\Phi }_{{\nu }%
}(r)=0.
\end{equation*}

\item[\textbf{c)}] If ${\nu }>-1,$ then $r_{\alpha }^{\star }(h_{{\nu }})=z_{{%
\nu ,\alpha ,1}}$ where $z_{{\nu ,\alpha ,1}}$ is the smallest positive root
of the equation
\begin{equation*}
r^{\frac{1}{4}}\mathbf{\Phi }_{{\nu }}^{\prime }(r^{\frac{1}{4}})-(4\alpha +2{\nu -3)}%
\mathbf{\Phi }_{{\nu }}(r^{\frac{1}{4}})=0.
\end{equation*}
\end{enumerate}
\end{theorem}

\begin{theorem}\label{T2}
\begin{enumerate}
\item[\textbf{a)}] If ${\nu }>0,$ then $r_{\alpha }^{\star }(u_{{\nu }%
})=\delta _{{\nu ,\alpha ,1}}$ where $\delta _{{\nu ,\alpha ,1}}$ is the
smallest positive root of
\begin{equation*}
r\mathbf{\Pi }_{{\nu }}^{\prime }(r)-2\alpha {\nu }\mathbf{\Pi }_{{\nu }%
}(r)=0.
\end{equation*}%
Moreover, if ${\nu \in (-1,0)}$ then $r_{\alpha }^{\star }(u_{{\nu }})=\delta
_{{\nu ,\alpha }}$ where $\delta _{{\nu ,\alpha }}$ is the unique root of the equation
\begin{equation*}
\sqrt{i}r\mathbf{\Pi }_{{\nu }}^{\prime }(\sqrt{i}r)-2\alpha {\nu }\mathbf{\Pi
}_{{\nu }}(\sqrt{i}r)=0.
\end{equation*}

\item[\textbf{b)}] If ${\nu }>-1,$ then $r_{\alpha }^{\star }(v_{{\nu }%
})=\rho _{{\nu ,\alpha ,1}}$ where $\rho _{{\nu ,\alpha ,1}}$ is the
smallest positive root of the equation
\begin{equation*}
r\mathbf{\Pi }_{{\nu }}^{\prime }(r)-(\alpha +2{\nu -1})\mathbf{\Pi }_{{\nu }%
}(r)=0.
\end{equation*}

\item[\textbf{c)}] If ${\nu }>-1,$ then $r_{\alpha }^{\star }(w_{{\nu }%
})=\sigma _{{\nu ,\alpha ,1}}$ where $\sigma _{{\nu ,\alpha ,1}}$ is the
smallest positive root of the equation
\begin{equation*}
r^{\frac{1}{4}}\mathbf{\Pi }_{{\nu }}^{\prime }(r^{\frac{1}{4}})-2(2\alpha +{\nu -2})\mathbf{%
\Pi }_{{\nu }}(r^{\frac{1}{4}})=0.
\end{equation*}
\end{enumerate}
\end{theorem}

\subsection{Bounds for the radii of starlikeness of the normalized forms of $\mathbf{\Phi}_{\nu}$ and $\mathbf{\Pi}_{\nu}$} Now we are going to present some tight lower and upper bounds for the radii of starlikeness of the six normalized functions. Our approach include the Euler-Rayleigh inequalities (for such inequalities involving zeros of Bessel functions of the first kind we refer to \cite[p. 501]{Wat} and to \cite{ismail}) and the fact that the so-called Laguerre-P\'olya class of real entire functions is closed under differentiation. It is worth to mention that the inequalities presented in the following six theorems are only the first three corresponding Euler-Rayleigh inequalities, and of course they can be improved by considering higher order Euler-Rayleigh inequalities. In other words, in each of the following six theorems there exist an increasing sequence of functions in terms of $\nu$ and a decreasing sequence of functions in terms of $\nu$ such that the corresponding radii of starlikeness are between these two sequences. The first set of results in this subsection concerns some tight bounds for the radii of starlikeness of the functions $f_{\nu},$ $g_{\nu}$ and $h_{\nu}.$ In the sequel we use the so-called Pochhammer symbol $(a)_n=a(a+1){\dots}(a+n-1)$ for $a\neq0.$

It is important to mention here that although the following six theorems were deduced to delimitate the radii of starlikeness of the normalized products of Bessel and modified Bessel functions of the first kind, the results presented here are interesting in their own right because provide actually tight lower and upper bounds for the first positive zeros of some special functions. Similar results were used often in the literature of applied mathematics concerning the zeros of Bessel functions itself. Moreover, very recently some lower and upper bounds for $\gamma_{\nu,1}$ (obtained in \cite{cross} and similar to what we have in Theorem \ref{THbound3} for $\gamma_{\nu,1}'$) were used by \"Ozer and \c{S}eng\"ul \cite{ozer} in their study of the linearized stability and transitions for the Poiseuille flow of a second grade fluid (which
is a model for non-Newtonian fluids) in order to compute the corresponding critical Reynolds number.

\begin{theorem}
\label{THbound3} The radius of starlikeness $r^{\star}(f_{\nu
}),$ denoted by $\gamma _{{\nu }%
,1}^{\prime },$ satisfies $$r^{\star}(f_{\nu })=\gamma _{{\nu }%
,1}^{\prime }<\sqrt[4]{4(\nu +1)_{3}(2\nu +1)}$$ for each $\nu >-\frac{1}{2}.$
Moreover, if $\nu >-{\frac{1}{2}},$ then $\gamma _{{\nu },1}^{\prime }$ satisfies
\begin{equation*}
\sqrt[4]{\frac{16\left( 2{\nu +1}\right) (\nu +1)_{3}}{2{\nu +5}}}<\gamma _{{%
\nu },1}^{\prime }<\sqrt[4]{\frac{16\left( 2\nu +1\right) \left( 2\nu
+5\right) \left( \nu +1\right) _{5}\allowbreak }{20\nu ^{3}+184\nu
^{2}+529\nu +473}},
\end{equation*}%
\begin{equation*}
\sqrt[8]{\frac{2^{8}\left( 2\nu +1\right) ^{2}(\nu +1)_{3}(\nu +1)_{5}}{%
20\nu ^{3}+184\nu ^{2}+529\nu +473}}<\gamma _{{\nu },1}^{\prime }<\sqrt[4]{%
\frac{\allowbreak 8\left( \nu +1\right) _{3}\left( \nu +6\right) \left( \nu
+7\right) \left( 2\nu +1\right) \allowbreak \left( 20\nu ^{3}+184\nu
^{2}+529\nu +473\right) }{\nu _{1}}\allowbreak }
\end{equation*}%
and%
\begin{equation*}
\sqrt[12]{\frac{\allowbreak 2^{11}\left( 2\nu +1\right) ^{3}\left( (\nu
+1)_{3}\right) ^{2}(\nu +1)_{7}}{\nu _{1}}}<\gamma _{{\nu }%
,1}^{\prime }<\sqrt[4]{\frac{32\left( \nu +1\right) _{5}\left( \nu +8\right)
\allowbreak \left( \nu +9\right) \left( 2\nu +1\right) \nu _{1}}{\nu
_{2}}},
\end{equation*}%
where%
\begin{equation*}
\nu _{1}=168\nu ^{5}+2876\nu ^{4}+18\,590\nu ^{3}+57\,349\nu
^{2}+84\,874\nu +48\,267
\end{equation*}%
and%
\begin{eqnarray*}
\nu _{2}&=&6864\nu ^{9}+245\,792\nu ^{8}+3802\,808\nu
^{7}+33\,438\,984\nu ^{6}+184\,372\,941\nu ^{5}+661\,304\,856\nu ^{4}\\
&&+1542\,867\,228\nu ^{3}+2256\,870\,262\nu ^{2}+1877\,042\,671\nu
+675\,828\,138.
\end{eqnarray*}
\end{theorem}

The following two theorems are related to some lower and upper bounds for radii of starlikeness of the normalized forms $g_{\nu}$ and $h_{\nu}.$

\begin{theorem}
\label{THbound4} The radius of starlikeness $r^{\star}(g_{\nu
}) $ satisfies $r^{\star}(g_{\nu })<\sqrt[4]{4(\nu +1)_{3}}$ for each $\nu
>-1.$ Moreover, under the same condition it satisfies
\begin{equation*}
\sqrt[4]{\frac{16(\nu +1)_{3}}{5}}<r^{\star }(g_{\nu })<\sqrt[4]{\frac{%
80(\nu +1)_{5}}{16\nu ^{2}+189\nu +473}},
\end{equation*}%
\begin{equation*}
\sqrt[8]{\frac{2^{8}(\nu +1)_{3}(\nu +1)_{5}}{16\nu ^{2}+189\nu +473}}%
<r^{\star }(g_{\nu })<\sqrt[4]{\frac{8(\nu +1)_{3}(\nu +6)(\nu +7)\left(
16\nu ^{2}+189\nu +473\right) }{\nu _{3}}}
\end{equation*}%
and%
\begin{equation*}
\sqrt[12]{\frac{\allowbreak 2^{11}\left( (\nu +1)_{3}\right)^{2}(\nu +1)_{7}%
}{\nu _{3}}}<r^{\star }(g_{\nu })<\sqrt[4]{\frac{\allowbreak
\allowbreak 32(\nu +1)_{5}(\nu +8)(\nu +9)\nu _{3}}{\nu _{4}}%
},
\end{equation*}%
where%
\begin{equation*}
\nu _{3}=32\nu ^{4}+824\nu ^{3}+7969\nu ^{2}+32\,944\nu +48\,267
\end{equation*}%
and%
\begin{eqnarray*}
\nu _{4}&=&256\nu ^{8}+13\,568\nu ^{7}+312\,736\nu
^{6}+4085\,373\nu ^{5}+32\,951\,080\nu ^{4}\\
&&+167\,370\,756\nu ^{3}+521\,177\,838\nu ^{2}+907\,600\,351\nu
+675\,828\,138.
\end{eqnarray*}
\end{theorem}

\begin{theorem}
\label{THbound5}
The radius of starlikeness $r^{\star }(h_{\nu })$ satisfies $r^{\star }(h_{\nu })<16(\nu +1)_{3}$ for each $\nu >-1.$
Moreover, if $\nu >-1,$ then it satisfies%
\begin{equation*}
8(\nu +1)_{3}<r^{\star }(h_{\nu })<\frac{32\left( \nu +1\right) _{5}}{\nu
^{2}+24\nu +71}
\end{equation*}%
\begin{equation*}
16\sqrt{\frac{(\nu +1)_{3}(\nu +1)_{5}}{\nu ^{2}+24\nu +71}}<r^{\star
}(h_{\nu })<\frac{16(\nu +1)_{3}(\nu +6)(\nu +7)\left( \nu ^{2}+24\nu
+71\right) }{\nu _{5}}
\end{equation*}%
and%
\begin{equation*}
16\sqrt[3]{\frac{\left( (\nu +1)_{3}\right)^{2}(\nu +1)_{7}}{\nu _{5}}}<r^{\star }(h_{\nu })<\frac{\allowbreak \allowbreak 16(\nu +1)_{5}(\nu
+8)(\nu +9)\nu _{5}}{\nu _{6}},
\end{equation*}%
where%
\begin{equation*}
\nu _{5}=\nu ^{4}+37\nu ^{3}+593\nu ^{2}+3275\nu +5598
\end{equation*}%
and%
\begin{eqnarray*}
\nu _{6} &=&\nu ^{8}+68\nu ^{7}+2062\nu ^{6}+36\,519\nu
^{5}+388\,627\nu ^{4}\\
&&+2477\,862\nu ^{3}+9218\,508\nu ^{2}+18\,391\,471\nu +15\,167\,442.
\end{eqnarray*}
\end{theorem}

The second set of results in this subsection concerns the bounds for the radii of starlikeness of the functions $u_{\nu},$ $v_{\nu}$ and $w_{\nu}.$ It is important to mention here that the method used in these six theorems concerning bounds for the radii of starlikeness is not new, its origins goes back to Euler and Rayleigh, see \cite{ismail} for more details. As we mentioned above, the bounds deduced for the radii of starlikeness are actually particular cases of the Euler-Rayleigh inequalities and it is possible to show that the deduced lower bounds increase and the upper bounds decrease to the corresponding radii of starlikeness. The fact that the lower bounds increase can be deduced directly from the corresponding Euler-Rayleigh inequalities, while the fact that the upper bounds are decreasing is actually a consequence of the Cauchy-Schwarz inequality. In other words, the inequalities presented in this paper can be improved by using higher order Euler-Rayleigh inequalities. However, we restricted ourselves to the third Euler-Rayleigh inequalities since these
are already quite complicated. Moreover, it is of interest to mention that it is possible to show that the radii of starlikeness of the six normalized functions considered in this paper actually correspond to the radii of univalence. We recall that the radius of univalence of the function $f\in\mathcal{S}$ is the largest radius $r$ for which $f$ maps univalently the open disk $\mathbb{D}_r$ into some domain. Some similar results on the equality of the radii of starlikeness and univalence were proved recently by Akta\c{s} et al. \cite{aktas} for normalized Bessel, Struve and Lommel functions of the first kind by using the ideas of Kreyszig and Todd \cite{todd}, as well as of Wilf \cite{wilf}.

\begin{theorem}
\label{THSbound1} The radius of starlikeness $r^{\star }(u_{\nu
}),$ denoted by $t_{{\nu },1},$ satisfies $t_{{\nu },1}<%
\sqrt[4]{8\nu (\nu +1)^{2}(\nu +2)}$ for each $\nu >0.$ Moreover, if $\nu >0,
$ then $t_{{\nu },1}$ satisfies%
\begin{equation*}
2\sqrt[4]{{\nu }(\nu +1)^{2}}<t_{{\nu },1}<2\sqrt[4]{\frac{(\nu )_{4}\left(
\nu +1\right) _{2}}{5\nu ^{2}+15\nu +12}},
\end{equation*}%
\begin{equation*}
2\sqrt[8]{\frac{\left( \nu \right) _{3}(\nu )_{4}\left( \nu +1\right) ^{2}}{%
5\nu ^{2}+15\nu +12}}<t_{{\nu },1}<\sqrt[4]{\frac{8\nu \left( \nu +1\right)
^{2}\left( \nu +3\right) _{3}\left( 5\nu ^{2}+15\nu +12\right) \allowbreak }{%
\nu _{7}}\allowbreak }
\end{equation*}%
and%
\begin{equation*}
\sqrt[12]{\frac{2^{11}\nu (\nu )_{4}\left( \nu \right) _{6}\left( \nu
+1\right) ^{4}\allowbreak }{\nu _{7}}}<t_{{\nu },1}<2\sqrt[4]{\frac{%
2\nu \left( \nu +1\right) ^{2}\left( \nu +2\right) ^{2}\left( \nu +4\right)
\left( \nu +6\right) \left( \nu +7\right) \nu _{7}}{\nu _{8}}%
},
\end{equation*}%
where%
\begin{equation*}
\nu _{7}=21\nu ^{4}+173\nu ^{3}+533\nu ^{2}+717\nu +360
\end{equation*}%
and%
\begin{eqnarray*}
\nu _{8}&=&429\nu ^{8}+8688\nu ^{7}+76\,280\nu ^{6}+377\,494\nu
^{5}+1148\,139\nu ^{4} \\
&&+2194\,202\nu ^{3}+2574\,064\nu ^{2}+1698\,048\nu +483\,840.
\end{eqnarray*}
\end{theorem}

\begin{theorem}
\label{THSbound2} The radius of starlikeness $r^{\star }(v_{\nu
})$ satisfies $r^{\star
}(v_{\nu })<\sqrt[4]{4(\nu +1)^{2}(\nu +2)}$ for each $\nu >-1.$ Moreover,
if $\nu >-1,$ then it satisfies
\begin{equation*}
2\sqrt[4]{\frac{(\nu +1)^{2}(\nu +2)}{5}}<r^{\star }(v_{\nu })<2\sqrt[4]{%
\frac{5\left( \nu +1\right) \left( \nu +1\right) _{4}}{16\nu ^{2}+157\nu +291%
}},
\end{equation*}%
\begin{equation*}
2\sqrt[8]{\frac{(\nu +1)_{4}\left( \nu +1\right) ^{3}\left( \nu +2\right) }{%
16\nu ^{2}+157\nu +291}}<r^{\star }(v_{\nu})<\sqrt[4]{\frac{8\left( \nu
+1\right) \left( \nu +1\right)_{3}\left( \nu +5\right)_{2}\left(16\nu
^{2}+157\nu +291\right)}{\nu_{9}}}
\end{equation*}%
and%
\begin{equation*}
\sqrt[12]{\frac{\allowbreak 2^{11}(\nu +1)_{3}(\nu +1)_{6}\left( \nu
+1\right) ^{4}\left( \nu +2\right) }{\nu_{9}}}<r^{\star }(v_{\nu
})<2\sqrt[4]{\frac{2\left( \nu +1\right) ^{2}\left( \nu +2\right) \left( \nu
+4\right) \left( \nu +7\right) \allowbreak _{2}\nu_{9}}{\nu
_{10}}},
\end{equation*}%
where%
\begin{equation*}
\nu _{9}=32\nu ^{5}+792\nu ^{4}+7753\nu ^{3}+35\,977\nu
^{2}+78\,453\nu +64\,469
\end{equation*}%
and%
\begin{eqnarray*}
\nu _{10}&=&256\nu ^{8}+11\,520\nu ^{7}+224\,672\nu
^{6}+2469\,757\nu ^{5}+16\,606\,040\nu ^{4} \\
&&+69\,429\,816\nu ^{3}+175\,324\,950\nu ^{2}+243\,560\,267\nu +142\,215\,442,
\end{eqnarray*}
\end{theorem}

\begin{theorem}
\label{THSbound3}
The radius of starlikeness $r^{\star }(w_{\nu })
$ satisfies $r^{\star }(w_{\nu })<16(\nu +1)^{2}(\nu +2)$ for each $\nu >-1.$
Moreover, under the same condition it satisfies
\begin{equation*}
8(\nu +1)^{2}(\nu +2)<r^{\star }(w_{\nu })<\frac{32\left( \nu +1\right)
\left( \nu +1\right) _{4}}{\nu ^{2}+22\nu +45},
\end{equation*}%
\begin{equation*}
16\left( \nu +1\right) ^{2}\left( \nu +2\right) \sqrt{\frac{(\nu +3)(\nu +4)%
}{\nu ^{2}+22\nu +45}}<r^{\star }(w_{\nu })<\frac{16\left( \nu +1\right)
\left( \nu +1\right) _{3}\left( \nu +5\right) _{2}\left( \nu ^{2}+22\nu
+45\right) }{\nu_{11}}
\end{equation*}%
and%
\begin{equation*}
16\left( \nu +1\right) ^{2}\left( \nu +2\right) \sqrt[3]{\frac{(\nu +3)(\nu
+3)_{4}}{\nu _{11}}}<r^{\star }(w_{\nu })<\frac{16\left( \nu
+1\right) ^{2}\left( \nu +2\right) \left( \nu +4\right) \left( \nu +7\right)
\allowbreak \left( \nu +8\right) \nu _{11}}{\nu _{12}},
\end{equation*}%
where%
\begin{equation*}
\nu _{11}=\nu ^{5}+36\nu ^{4}+584\nu ^{3}+3554\nu ^{2}+8919\nu +7834
\end{equation*}%
and%
\begin{eqnarray*}
\nu _{12}&=&\nu ^{8}+60\nu ^{7}+1610\nu ^{6}+25\,303\nu
^{5}+229\,535\nu ^{4}\\
&&+1199\,202\nu ^{3}+3542\,412\nu ^{2}+5461\,979\nu +3396\,826.
\end{eqnarray*}
\end{theorem}

\subsection{The radii of convexity of the normalized forms of $\mathbf{\Phi}_{\nu}$ and $\mathbf{\Pi}_{\nu}$}
Let $\mathcal{K}(\beta)$ be the subclass of $\mathcal{S}$ consisting of functions which are convex of
order $\beta$ in $\mathbb{D}_{r},$ where $0\leq \beta<1.$ The analytic characterization of this class of functions is
\begin{equation*}
\mathcal{K}(\beta)=\left\{ f\in \mathcal{S}\ :\ \real\left( 1+\frac{%
zf^{\prime \prime }(z)}{f^{\prime }(z)}\right) >\beta \ \ \mathrm{for\ all}%
\ \ z\in \mathbb{D}_{r}\right\} ,
\end{equation*}%
and for $\beta=0$ it reduces to the class $\mathcal{K}$ of convex
functions. The real number
\begin{equation*}
r_{\beta}^{c}(f)=\sup \left\{ r>0\ :\ \real\left( 1+\frac{zf^{\prime
\prime }(z)}{f^{\prime }(z)}\right) >\beta \ \ \mathrm{\ for\ all}\ \ z\in
\mathbb{D}_{r}\right\} ,
\end{equation*}%
is called the radius of convexity of order $\beta $ of the function $f.$
Note that $r^{c}(f)=r_{0}^{c}(f)$ is the largest radius such that the image
region $f(\mathbb{D}_{r^{c}(f)})$ is a convex domain in ${\mathbb{C}}$.

The following set of main results concerns the radii of convexity.

\begin{theorem}\label{T3}
\begin{enumerate}
\item[\textbf{a)}] If ${\nu }>-\frac{1}{2},$ then $r_{\alpha}^{c}(f_{\nu})$ is the smallest positive root of
\begin{equation*}
1+\frac{r\,\mathbf{\Phi }_{{\nu }}^{\prime \prime }(r)}{\mathbf{\Phi }_{{\nu
}}^{\prime }(r)}+\left( \frac{1}{2{\nu }+1}-1\right) \frac{r\,\mathbf{\Phi }%
_{{\nu }}^{\prime }(r)}{\mathbf{\Phi }_{{\nu }}(r)}=\alpha .
\end{equation*}

\item[\textbf{b)}] If ${\nu }>-1,$ then $r_{\alpha}^{c}(g_{\nu})$ is the smallest positive root of
\begin{equation*}
-2{\nu }+r\frac{(1-2{\nu )}\mathbf{\Phi }_{{\nu }}^{\prime }(r)+r\,\mathbf{%
\Phi }_{{\nu }}^{\prime \prime }(r)}{-2{\nu }\mathbf{\Phi }_{{\nu }}(r)+r%
\mathbf{\Phi }_{{\nu }}^{\prime }(r)}=\alpha .
\end{equation*}

\item[\textbf{c)}] If ${\nu }>-1,$ then $r_{\alpha}^{c}(h_{\nu})$ is the smallest positive root of
\begin{equation*}
3-2{\nu }+r^{\frac{1}{4}}\frac{(4-2{\nu )}\mathbf{\Phi }_{{\nu }}^{\prime
}(r^{\frac{1}{4}})+r^{\frac{1}{4}}\mathbf{\Phi }_{{\nu }}^{\prime \prime
}(r^{\frac{1}{4}})}{(3-2{\nu )}\mathbf{\Phi }_{{\nu }}(r^{\frac{1}{4}})+r^{%
\frac{1}{4}}\mathbf{\Phi }_{{\nu }}^{\prime }(r^{\frac{1}{4}})}=4\alpha .
\end{equation*}
\end{enumerate}

Moreover, we have the inequalities $r_{\alpha }^{c}(f_{{\nu }})<\gamma _{{%
\nu },1}^{\prime }<\gamma _{{\nu },1},$ $r_{\alpha }^{c}(g_{{\nu }})<\zeta _{%
{\nu },1}<\gamma _{{\nu },1},$ and $r_{\alpha }^{c}(h_{{\nu }})<\xi _{{\nu }%
,1}<\gamma _{{\nu },1}$ where $\zeta _{{\nu },1}$ and $\xi _{{\nu },1}$ are
the first positive zeros of $z\mapsto z\mathbf{\Phi }_{{%
\nu }}^{\prime }(z)-2{\nu }\mathbf{\Phi }_{{\nu }}(z)$ and $z\mapsto z%
\mathbf{\Phi }_{{\nu }}^{\prime }(z)-(2{\nu -3)}\mathbf{\Phi }_{{\nu }}(z),$
while $\gamma _{{\nu },1}$ and $%
\gamma _{{\nu },1}^{\prime }$ denote the first positive zeros of $\mathbf{%
\Phi }_{{\nu }}$ and $\mathbf{\Phi }_{{\nu }}^{\prime }$, respectively.
\end{theorem}

\begin{theorem}
\label{T4}
\begin{enumerate}
\item[\textbf{a)}] If ${\nu }>0,$ then $r_{\alpha}^{c}(u_{\nu})$ is the smallest positive root of
\begin{equation*}
1+\frac{r\,\mathbf{\Pi }_{{\nu }}^{\prime \prime }(r)}{\mathbf{\Pi }_{{\nu }%
}^{\prime }(r)}+\left( \frac{1}{2{\nu }}-1\right) \frac{r\,\mathbf{\Pi }_{{%
\nu }}^{\prime }(r)}{\mathbf{\Pi }_{{\nu }}(r)}=\alpha .
\end{equation*}

\item[\textbf{b)}] If ${\nu }>-1,$ then $r_{\alpha}^{c}(v_{\nu})$ is the smallest positive root of
\begin{equation*}
1-2{\nu }+r\frac{(2-2{\nu )}\mathbf{\Pi }_{{\nu }}^{\prime }(r)+r\,\mathbf{%
\Pi }_{{\nu }}^{\prime \prime }(r)}{(1-2{\nu )}\mathbf{\Pi }_{{\nu }}(r)+r%
\mathbf{\Pi }_{{\nu }}^{\prime }(r)}=\alpha .
\end{equation*}

\item[\textbf{c)}] If ${\nu }>-1,$ then $r_{\alpha}^{c}(w_{\nu})$ is the smallest positive root of
\begin{equation*}
4-2{\nu }+r^{\frac{1}{4}}\frac{(5-2{\nu )}\mathbf{\Pi }_{{\nu }}^{\prime
}(r^{\frac{1}{4}})+r^{\frac{1}{4}}\mathbf{\Pi }_{{\nu }}^{\prime \prime }(r^{%
\frac{1}{4}})}{(4-2{\nu )}\mathbf{\Pi }_{{\nu }}(r^{\frac{1}{4}})+r^{\frac{1%
}{4}}\mathbf{\Pi }_{{\nu }}^{\prime }(r^{\frac{1}{4}})}=4\alpha .
\end{equation*}
\end{enumerate}

Moreover, we have $r_{\alpha }^{c}(u_{{\nu }})<j_{{\nu }%
,1}^{\prime }<j_{{\nu },1},$ $r_{\alpha }^{c}(v_{{\nu }})<\vartheta _{{\nu }%
,1}<j_{{\nu },1},$ and $r_{\alpha }^{c}(w_{{\nu }})<\omega _{{\nu },1}<j_{{%
\nu },1},$ where $\vartheta _{{\nu },1}$ and $\omega _{{\nu },1}$ are the
first positive zeros of $z\mapsto z\mathbf{\Pi }_{{\nu }%
}^{\prime }(z)-(2{\nu -1})\mathbf{\Pi }_{{\nu }}(z)$ and $z\mapsto z%
\mathbf{\Pi }_{{\nu }}^{\prime }(z)-(2{\nu -4})\mathbf{\Pi }_{{\nu }}(z),$ while
$j_{{\nu },1}$ and $j_{{\nu }%
,1}^{\prime }$ denote the first positive zeros of $J_{{\nu }}$ and $J_{{\nu }%
}^{\prime }$, respectively.
\end{theorem}

\subsection{Bounds for the radii of convexity of the normalized forms of $\mathbf{\Phi}_{\nu}$ and $\mathbf{\Pi}_{\nu}$}
Proceeding similarly as in the proof of the bounds for the radii of starlikeness of the normalized products, in this subsection our aim is to
present some lower and upper bounds for the radii of convexity of the functions $g_{\nu},$ $h_{\nu},$ $v_{\nu}$ and $w_{\nu}.$ These bounds are also
particular cases of Euler-Rayleigh inequalities and consequently they can be improved by using higher order Euler-Rayleigh inequalities. We restricted here ourselves only to the first two Euler-Rayleigh inequalities.

\begin{theorem}
\label{THC1} The radius of convexity $r^{c}(g_{\nu })$ satisfies $%
r^{c}(g_{\nu })<\sqrt[4]{\frac{4(\nu +1)_{3}}{5}}$ for each $\nu >-1.$
Moreover, if $\nu >-1,$ then the radius of convexity of $g_{{\nu }}$
satisfies
\begin{equation*}
2\sqrt[4]{\frac{(\nu +1)_{3}}{25}}<r^{c}(g_{\nu })<2\sqrt[4]{\frac{25(\nu
+1)_{5}}{\nu_{13}}}
\end{equation*}%
and%
\begin{equation*}
2\sqrt[8]{\frac{(\nu +1)_{3}(\nu +1)_{5}}{\nu _{13}}}<r^{c}(g_{\nu
})<\sqrt[4]{\frac{8(\nu +1)_{3}(\nu +6)(\nu +7)\nu _{13}}{3\nu
_{14}}},
\end{equation*}%
where%
\begin{equation*}
\nu_{13}=544\nu ^{2}+5301\nu +12\,257
\end{equation*}%
and%
\begin{equation*}
\nu_{14}=2112\nu ^{4}+48\,784\nu ^{3}+417\,279\nu ^{2}+1556\,904\nu
+2123\,797.
\end{equation*}
\end{theorem}

\begin{theorem}
\label{THC2}The radius of convexity $r^{c}(h_{\nu })$ satisfies $%
r^{c}(h_{\nu })<8(\nu +1)_{3}$ for each $\nu >-1.$ Moreover, if $\nu >-1,$
then the radius of convexity of $h_{{\nu }}$ satisfies%
\begin{equation*}
4\left( \nu +1\right) _{3}<r^{c}(h_{\nu })<\frac{64\left( \nu +1\right) _{5}%
}{7\nu ^{2}+108\nu +293}
\end{equation*}%
and%
\begin{equation*}
16(\nu +1)_{3}\sqrt{\frac{(\nu +4)_{2}}{\nu_{15}}}<r^{c}(h_{\nu })<%
\frac{8\left( \nu +1\right) _{3}\left( \nu +6\right) _{2}\nu_{15}}{%
3\nu_{16}},
\end{equation*}%
where%
\begin{equation*}
\nu_{15}=7\nu ^{2}+108\nu +293
\end{equation*}%
and%
\begin{equation*}
\nu_{16}=3\nu ^{4}+91\nu ^{3}+1059\nu ^{2}+4965\nu +7834.
\end{equation*}
\end{theorem}

\begin{theorem}
\label{THC3} The radius of convexity $r^{c}(v_{\nu })$ satisfies $r^{c}(v_{\nu })<\sqrt[4]{\frac{4(\nu
+1)^{2}(\nu +2)}{5}}$ for each $\nu >-1.$ Moreover, if $\nu >-1,$ then the
radius of convexity of $v_{{\nu }}$ satisfies
\begin{equation*}
2\sqrt[4]{\frac{(\nu +1)^{2}(\nu +2)}{25}}<r^{c}(v_{\nu })<2\sqrt[4]{\frac{%
25\left(\nu +1\right) \left(\nu +1\right)_{4}}{\nu_{17}}}
\end{equation*}%
and%
\begin{equation*}
2\sqrt[8]{\frac{(\nu +1)_{4}\left( \nu +1\right) ^{3}\left( \nu +2\right) }{%
\nu_{17}}}<r^{c}(v_{\nu })<\sqrt[4]{\frac{8\left( \nu +1\right)
\left( \nu +1\right) _{3}\left(\nu +5\right)_{2}\nu_{17}}{\nu
_{18}}},
\end{equation*}%
where%
\begin{equation*}
\nu _{17}=544\nu ^{2}+4213\nu +7419
\end{equation*}%
and%
\begin{equation*}
\nu _{18}=6336\nu ^{5}+140\,016\nu ^{4}+1212\,429\nu
^{3}+5112\,301\nu ^{2}+10\,459\,449\nu +8300\,897.
\end{equation*}
\end{theorem}

\begin{theorem}
\label{THC4} The radius of convexity $r^{c}(w_{\nu })$ satisfies $r^{c}(w_{\nu
})<8(\nu +1)^{2}(\nu +2)$ for each $\nu >-1.$ Moreover, if $\nu >-1,$ then
the radius of convexity of $w_{{\nu }}$ satisfies%
\begin{equation*}
4(\nu +1)^{2}(\nu +2)<r^{c}(w_{\nu })<\frac{64\left( \nu +1\right) \left(
\nu +1\right) _{4}}{7\nu ^{2}+94\nu +183}
\end{equation*}%
and%
\begin{equation*}
16\left( \nu +1\right) ^{2}\left( \nu +2\right) \sqrt{\frac{(\nu +3)\left(
\nu +4\right) }{\nu_{19}}}<r^{c}(w_{\nu })<\frac{8\left(\nu
+1\right) \left(\nu +1\right)_{3}\left( \nu +5\right)_{2}\nu_{19}}{\nu_{20}},
\end{equation*}%
where%
\begin{equation*}
\nu _{19}=7\nu ^{2}+94\nu +183
\end{equation*}%
and%
\begin{equation*}
\nu _{20}=9\nu ^{5}+264\nu ^{4}+3108\nu ^{3}+16\,222\nu
^{2}+37\,827\nu +32\,138.
\end{equation*}
\end{theorem}

\subsection{Starlikeness of the normalized forms of $\mathbf{\Phi}_{\nu}$ and $\mathbf{\Pi}_{\nu}$}
The following set of results concerns the necessary and sufficient conditions on
starlikeness with respect to the origin for the six normalized functions of the cross-product and
product of Bessel and modified Bessel functions of the first kind. Fig. \ref{fig} contains the graph of the left-hand sides of the equations \eqref{starf}, \eqref{starg}, \eqref{starh}, \eqref{staru}, \eqref{starv} and \eqref{starw} when $\alpha=0.$ As we can see in the proof
of these results the key tools are the monotonicity properties of the zeros of the cross-product
and product with respect to the order. Motivated by the results from \cite{BAS}, it would be interesting
to see the counterparts of the following results for the convex case. For the functions $g_{\nu},$ $h_{\nu},$
$v_{\nu}$ and $w_{\nu}$ we were able to deduce the counterparts of the following theorems in the next subsection, however,
for the functions $f_{\nu}$ and $u_{\nu}$ the determination of the order of convexity remains an open problem, which can be of interest
for further research. We mention that the particular cases of part {\bf c} of the
following theorems when $\alpha=0$ have been deduced also in \cite{HAT}, however, the general case of the starlikeness of order $\alpha$ has not been
considered there. It would be of interest to see if the following results imply some close-to-convexity results on the derivatives
of the normalized functions $f_{\nu},$ $g_{\nu},$ $h_{\nu},$ $u_{\nu},$ $v_{\nu}$ and $w_{\nu},$ similarly as it was proved in \cite{HAT} for $h_{\nu}$ and $w_{\nu}.$

\begin{theorem}
\label{T5}
\begin{enumerate}
\item[\textbf{a)}] The function $f_{{\nu }}$ is starlike of order $\alpha $
in $\mathbb{D}$ if and only if ${\nu \geq \nu }_{\alpha }^{\star }(f_{{\nu }%
}) $, where ${\nu }_{\alpha }^{\star }(f_{{\nu }})$ is the unique root of the
equation%
\begin{equation}\label{starf}
2J_{\nu}(1)I_{\nu}(1)-(\alpha (2{\nu +1)+1)}\left(J_{\nu+1}(1)I_{\nu}(1)+J_{\nu}(1)I_{\nu+1}(1)\right)=0.
\end{equation}
In particular, $f_{\nu}$ is starlike in $\mathbb{D}$ if and only if ${\nu \geq \nu }_{0}^{\star }(f_{\nu })\simeq-0.44{\dots}$,
where ${\nu}_{0}^{\star }(f_{{\nu }})$ is the unique root of the equation \eqref{starf} when $\alpha=0.$

\item[\textbf{b)}] The function $g_{{\nu }}$ is starlike of order $\alpha $
in $\mathbb{D}$ if and only if ${\nu \geq \nu }_{\alpha }^{\star }(g_{{\nu }%
}) $, where ${\nu }_{\alpha }^{\star }(g_{{\nu }})$ is the unique root of the
equation%
\begin{equation}\label{starg}
2J_{\nu}(1)I_{\nu}(1)-(\alpha +2\nu +1)\left(J_{\nu+1}(1)I_{\nu}(1)+J_{\nu}(1)I_{\nu+1}(1)\right)=0.
\end{equation}
In particular, $g_{\nu}$ is starlike in $\mathbb{D}$ if and only if ${\nu \geq \nu }_{0}^{\star }(g_{\nu })\simeq-0.87{\dots}$,
where ${\nu}_{0}^{\star }(g_{{\nu }})$ is the unique root of the equation \eqref{starg} when $\alpha=0.$

\item[\textbf{c)}] The function $h_{{\nu }}$ is starlike of order $\alpha $
in $\mathbb{D}$ if and only if ${\nu \geq \nu }_{\alpha }^{\star }(h_{{\nu }%
}) $, where ${\nu }_{\alpha }^{\star }(h_{{\nu }})$ is the unique root of the
equation%
\begin{equation}\label{starh}
J_{\nu}(1)I_{\nu}(1)-(2\alpha +\nu -1)\left(J_{\nu+1}(1)I_{\nu}(1)+J_{\nu}(1)I_{\nu+1}(1)\right)=0.
\end{equation}
In particular, $h_{\nu}$ is starlike in $\mathbb{D}$ if and only if ${\nu \geq \nu }_{0}^{\star }(h_{\nu })\simeq-0.94{\dots}$,
where ${\nu}_{0}^{\star }(h_{{\nu }})$ is the unique root of the equation \eqref{starh} when $\alpha=0.$
\end{enumerate}
\end{theorem}

\begin{theorem}
\label{T6}
\begin{enumerate}
\item[\textbf{a)}] The function $u_{{\nu }}$ is starlike of order $\alpha $
in $\mathbb{D}$ if and only if ${\nu \geq \nu }_{\alpha }^{\star }(u_{{\nu }%
}) $, where ${\nu }_{\alpha }^{\star }(u_{{\nu }})$ is the unique root of the
equation%
\begin{equation}\label{staru}
J_{{\nu }}(1)I_{{\nu +1}}(1)-J_{{\nu +1}}(1)I_{{\nu }}(1)+2{\nu (1-\alpha )}%
J_{{\nu }}(1)I_{{\nu }}(1)=0.
\end{equation}
In particular, $u_{\nu}$ is starlike in $\mathbb{D}$ if and only if ${\nu \geq \nu }_{0}^{\star }(u_{\nu })\simeq0.05{\dots}$,
where ${\nu}_{0}^{\star }(u_{{\nu }})$ is the unique root of the equation \eqref{staru} when $\alpha=0.$

\item[\textbf{b)}] The function $v_{{\nu }}$ is starlike of order $\alpha $
in $\mathbb{D}$ if and only if ${\nu \geq \nu }_{\alpha }^{\star }(v_{{\nu }%
}) $, where ${\nu }_{\alpha }^{\star }(v_{{\nu }})$ is the unique root of the
equation%
\begin{equation}\label{starv}
J_{{\nu }}(1)I_{{\nu +1}}(1)-J_{{\nu +1}}(1)I_{{\nu }}(1)+{(1-\alpha )}J_{{%
\nu }}(1)I_{{\nu }}(1)=0.
\end{equation}
In particular, $v_{\nu}$ is starlike in $\mathbb{D}$ if and only if ${\nu \geq \nu }_{0}^{\star }(v_{\nu })\simeq-0.53{\dots}$,
where ${\nu}_{0}^{\star }(v_{{\nu }})$ is the unique root of the equation \eqref{starv} when $\alpha=0.$

\item[\textbf{c)}] The function $w_{{\nu }}$ is starlike of order $\alpha $
in $\mathbb{D}$ if and only if ${\nu \geq \nu }_{\alpha }^{\star }(w_{{\nu }%
}) $, where ${\nu }_{\alpha }^{\star }(w_{{\nu }})$ is the unique root of the
equation%
\begin{equation}\label{starw}
J_{{\nu }}(1)I_{{\nu +1}}(1)-J_{{\nu +1}}(1)I_{{\nu }}(1)+4{(1-\alpha )}J_{{%
\nu }}(1)I_{{\nu }}(1)=0.
\end{equation}
In particular, $w_{\nu}$ is starlike in $\mathbb{D}$ if and only if ${\nu \geq \nu }_{0}^{\star }(w_{\nu })\simeq-0.69{\dots}$,
where ${\nu}_{0}^{\star }(w_{{\nu }})$ is the unique root of the equation \eqref{starw} when $\alpha=0.$
\end{enumerate}
\end{theorem}

\begin{center}
\begin{figure}[!ht]
   \centering
       \includegraphics[width=16cm]{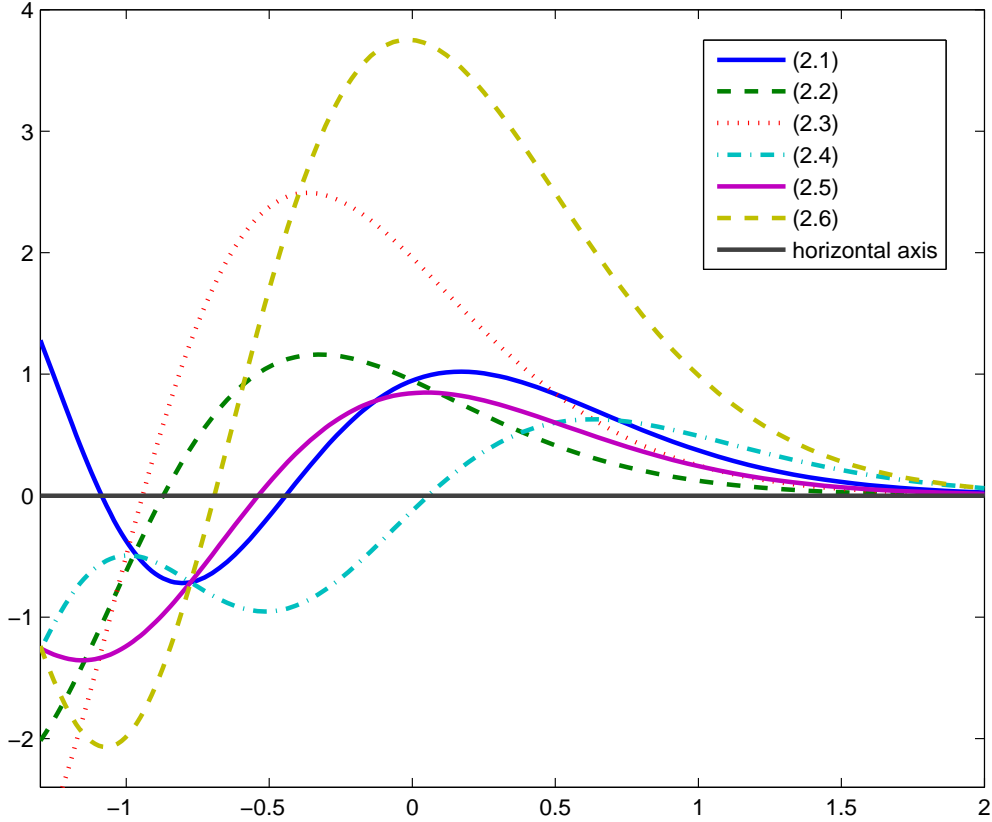}
       \caption{The graph of the left-hand sides of the equations \eqref{starf}, \eqref{starg}, \eqref{starh}, \eqref{staru},
\eqref{starv} and \eqref{starw} for $\alpha=0.$}
       \label{fig}
\end{figure}
\end{center}

\subsection{Convexity of the normalized forms of $\mathbf{\Phi}_{\nu}$ and $\mathbf{\Pi}_{\nu}$}
Recently, the Watson formulas for the zeros of $J_{\nu}$ and $J_{\nu}'$ were applied by Baricz and Sz\'asz \cite{BAS}
in order to deduce necessary and sufficient conditions for the parameter $\nu$ such that some normalized
forms of $J_{\nu}$ map the open unit disk into a convex domain. As we can see below it is possible to deduce the order of convexity of the functions
$g_{\nu},$ $h_{\nu},$ $v_{\nu}$ and $w_{\nu}$ without appealing to Watson-type formulas for the expression like $d\gamma_{\nu,n}/d\nu,$ $d\gamma_{\nu,n}'/d\nu$ and $d t_{\nu,n}/d\nu$ (here $\gamma_{\nu,n},$ $\gamma _{\nu ,n}^{\prime }$ and $t_{\nu ,n}$ are the $n$th positive
zeros of the functions $\mathbf{\Phi }_{\nu},$ $\mathbf{\Phi }_{{\nu }}^{\prime }$ and $\mathbf{\Pi}_{{\nu }}^{\prime },$ respectively), however the determination of the order of convexity for the functions $f_{\nu}$ and $u_{\nu}$ remains an open problem, and we believe that one of the best ways to solve this problem is to deduce for the above mentioned derivatives with respect to the order some Watson-type results, like what we have for $d j_{\nu,n}/d\nu$ and $d j_{\nu,n}'/d\nu,$ see formulas 10.21.17 and 10.21.18 in \cite{OL} (here $j_{\nu,n}'$ denotes the $n$th positive zero of $J_{\nu}'$).

\begin{theorem}\label{thconvexity1}
\begin{enumerate}
\item[\bf a)] The function $g_\nu$ is convex of order $\alpha$ in the unit disk $\mathbb{D}$ if and only if $\nu\geq\nu_{\alpha}^{c}(g_{\nu}),$ where
$\nu_{\alpha}^{c}(g_{\nu})$ is the unique  positive  root  of the equation $g_\nu''(1)+(1-\alpha)g_\nu'(1)=0.$

\item[\bf b)] The function $h_{\nu}$  is convex of order $\alpha$ in $\mathbb{D}$ if and only if  $\nu\geq\nu_{\alpha}^{c}(h_{\nu}),$ where
  $\nu_{\alpha}^{c}(h_{\nu})$ is the unique root of the equation $h_\nu''(1)+(1-\alpha)h_\nu'(1)=0.$
\end{enumerate}
\end{theorem}

\begin{theorem}\label{thconvexity2}
\begin{enumerate}
\item[\bf a)] The function $v_\nu$ is convex of order $\alpha$ in the unit disk $\mathbb{D}$ if and only if $\nu\geq\nu_{\alpha}^{c}(v_{\nu}),$ where
$\nu_{\alpha}^{c}(v_{\nu})$ is the unique  positive  root  of the equation $v_\nu''(1)+(1-\alpha)v_\nu'(1)=0.$

\item[\bf b)] The function $w_{\nu}$  is convex of order $\alpha$ in $\mathbb{D}$ if and only if  $\nu\geq\nu_{\alpha}^{c}(w_{\nu}),$ where
  $\nu_{\alpha}^{c}(w_{\nu})$ is the unique root of the equation $w_\nu''(1)+(1-\alpha)w_\nu'(1)=0.$
\end{enumerate}
\end{theorem}

We end this section with the following open problems.

\begin{problem}
Find the order of convexity of the functions $f_{\nu}$ and $u_{\nu}.$
\end{problem}

\begin{problem}
Find Watson-type formulas for the derivatives $d\gamma_{\nu,n}/d\nu,$ $d\gamma_{\nu,n}'/d\nu$ and $d t_{\nu,n}/d\nu.$
\end{problem}

\section{Preliminary Results}
\setcounter{equation}{0}

\subsection{The Hadamard factorization of the functions $\mathbf{\Phi}_{\nu}$ and $\mathbf{\Pi}_{\nu}$}

\begin{lemma}
\cite{HAT}\label{Had1} If ${\nu >-1}$ and $z\in {\mathbb{C}}$ then the
Hadamard factorizations of $\mathbf{\Phi }_{{\nu }}$ and $\mathbf{\Pi }_{{%
\nu }}$ are%
\begin{equation}
\mathbf{\Phi }_{{\nu }}(z)=\frac{z^{2{\nu }+1}}{2^{2{\nu }}\Gamma \left( {%
\nu }+1\right) \Gamma \left( {\nu }+2\right) }\prod_{n\geq1}\left( 1-\frac{z^{4}}{\gamma _{{\nu },n}^{4}}\right)  \label{fi1}
\end{equation}%
and%
\begin{equation}
\mathbf{\Pi }_{{\nu }}(z)=\frac{z^{2{\nu }}}{2^{2{\nu }}\Gamma^{2}\left(\nu+1\right)}\prod_{n\geq1}\left( 1-\frac{z^{4}}{j_{{\nu }%
,n}^{4}}\right) ,  \label{pi1}
\end{equation}%
where $\gamma _{\nu ,n}$ and $j_{\nu,n}$ are the $n$th positive zeros of the functions $\mathbf{%
\Phi }_{{\nu }}$ and $J_{\nu}$. Moreover, the zeros $\gamma _{{\nu },n}$ satisfy the
interlacing inequalities $j_{\nu ,n}<\gamma _{{\nu },n}<j_{\nu ,n+1}$ and $%
j_{\nu ,n}<\gamma _{{\nu },n}<j_{\nu +1,n}$ for $n\in \mathbb{N}$ and ${\nu
>-1}.$
\end{lemma}

\subsection{Monotonicity of quotients of power series}

We will also need the following result (see \cite{BK,PV}), which was often used in problems concerning the monotonicity
of quotients of some special functions, like modified Bessel functions of the first kind, modified Struve and Lommel functions of the first kind, Gaussian hypergeometric functions and other special functions which have power series structure with positive coefficients.

\begin{lemma}
\label{Quotients} Consider the power series $$f(x)=\sum_{n\geq0}a_{n}x^{n}\ \ \ \mbox{and}\ \ \ g(x)=\displaystyle\sum_{n\geq0}b_{n}x^{n},$$ where $a_{n}\in \mathbb{R}$ and $b_{n}>0$ for all $n\geq 0$.
Suppose that both series converge on $(-r,r)$, for some $r>0$. If the
sequence $\{a_{n}/b_{n}\}_{n\geq 0}$ is increasing (decreasing), then the
function $x\mapsto {f(x)}/{g(x)}$ is increasing (decreasing) too on $(0,r)$.
The result remains true for the power series
\begin{equation*}
f(x)=\displaystyle\sum_{n\geq0}a_{n}x^{4n}\ \ \ \mbox{and}\ \ \ g(x)=%
\displaystyle\sum_{n\geq0}b_{n}x^{4n}.
\end{equation*}
\end{lemma}

\subsection{Zeros of hyperbolic polynomials and the Laguerre-P\'{o}lya class of entire functions}

In this subsection we recall some necessary information about polynomials
and entire functions with real zeros. An algebraic polynomial is called
hyperbolic if all its zeros are real. We formulate the following specific
statement that we shall need, see \cite{bdoy} for more details.

\begin{lemma}
\label{Zeros} Let $p(x)=1-a_{1}x+a_{2}x^{2}-a_{3}x^{3}+\cdots
+(-1)^{n}a_{n}x^{n}=(1-x/x_{1})\cdots (1-x/x_{n})$ be a hyperbolic
polynomial with positive zeros $0<x_{1}\leq x_{2}\leq \cdots \leq x_{n}$,
and normalized by $p(0)=1$. Then, for any constant $C$, the polynomial $%
q(x)=Cp(x)-x\,p^{\prime }(x)$ is hyperbolic. Moreover, the smallest zero $%
\eta _{1}$ belongs to the interval $(0,x_{1})$ if and only if $C<0$.
\end{lemma}

By definition a real entire function $\psi $ belongs to the Laguerre-P\'{o}lya class $%
\mathcal{LP}$ if it can be represented in the form
\begin{equation*}
\psi (x)=cx^{m}e^{-ax^{2}+\beta x}\prod_{k\geq1}\left( 1+\frac{x}{%
x_{k}}\right) e^{-\frac{x}{x_{k}}},
\end{equation*}%
with $c,$ $\beta ,$ $x_{k}\in \mathbb{R},$ $a\geq 0,$ $m\in \mathbb{N}\cup
\{0\},$ $\sum x_{k}^{-2}<\infty .$ Similarly, $\phi $ is said to be of type
$\mathcal{I}$ in the Laguerre-P\'{o}lya class, written $\varphi \in \mathcal{LPI}$%
, if $\phi (x)$ or $\phi (-x)$ can be represented as
\begin{equation*}
\phi (x)=cx^{m}e^{\sigma x}\prod_{k\geq1}\left( 1+\frac{x}{x_{k}}%
\right) ,
\end{equation*}%
with $c\in \mathbb{R},$ $\sigma \geq 0,$ $m\in \mathbb{N}\cup \{0\},$ $%
x_{k}>0,$ $\sum 1/x_{k}<\infty .$ The class $\mathcal{LP}$ is the complement
of the space of hyperbolic polynomials in the topology induced by the
uniform convergence on the compact sets of the complex plane, while $\mathcal{%
LPI}$ is the complement of the hyperbolic polynomials whose zeros
posses a preassigned constant sign. Given an entire function $\varphi $ with
the Maclaurin expansion
\begin{equation*}
\varphi (x)=\sum_{n\geq0}\tau_{n}\frac{x^{n}}{n!},
\end{equation*}%
its Jensen polynomials are defined by
\begin{equation*}
P_{n}(\varphi ;x)=P_{n}(x)=\sum_{j=0}^{n}{\binom{n}{j}}\tau _{j}x^{j}.
\end{equation*}

The next result of Jensen \cite{Jen12} is a known characterization of functions belonging to $\mathcal{LP}.$

\begin{lemma}
\label{Jensen} The function $\varphi $ belongs to $\mathcal{LP}$ ($\mathcal{%
LPI}$, respectively) if and only if all the polynomials $%
P_{n}(\varphi ;x)$, $n=1,2,\ldots $, are hyperbolic (hyperbolic with zeros
of equal sign). Moreover, the sequence $P_{n}(\varphi ;z/n)$ converges
locally uniformly to $\varphi (z)$.
\end{lemma}

The following result is one of the key tools in the proof of main results.

\begin{lemma}
\label{LP} If ${\nu }>-1,$ then for ${a<1}$ the function $z\mapsto (2{%
\nu +a})\mathbf{\Phi }_{{\nu }}(z)-z\mathbf{\Phi }_{{\nu }}^{\prime }(z),$
and for ${b<0}$ the function $z\mapsto (2{\nu +b})\mathbf{\Pi }_{{\nu }%
}(z)-z\mathbf{\Pi }_{{\nu }}^{\prime }(z)$ can be represented in the form
\begin{equation*}
\Gamma \left( {\nu }+1\right) \Gamma \left( {\nu }+2\right) \left[ (2{\nu +a}%
)\mathbf{\Phi }_{{\nu }}(z)-z\mathbf{\Phi }_{{\nu }}^{\prime }(z)\right]
\mathbf{=}2\left( \frac{z}{2}\right) ^{2{\nu +1}}\Psi _{{\nu }}(z),
\end{equation*}%
\begin{equation*}
\Gamma ^{2}\left( {\nu }+1\right) \left[ (2{\nu +b})\mathbf{\Pi }_{{\nu }%
}(z))-z\mathbf{\Pi }_{{\nu }}^{\prime }(z)\right] \mathbf{=}\left( \frac{z}{2%
}\right) ^{2{\nu }}\phi _{{\nu }}(z),
\end{equation*}%
where $\Psi _{{\nu }}$ and $\phi _{{\nu }}$ are entire functions belonging to the Laguerre-P\'{o}lya class $\mathcal{LP}$. Moreover, the
smallest positive zero of $\Psi _{{\nu }}$ does not exceed $\gamma _{{\nu },1},$ while the
smallest positive zero of $\phi _{{\nu }}$ is less than $j_{{\nu },1}.$
\end{lemma}

\begin{proof}[\bf Proof]
Suppose that ${\nu }>-1.$ It is clear from Theorem \ref{thzeros} and Lemma \ref{Had1} on the infinite product
representation of $z\mapsto \mathbf{\Upsilon }_{{\nu }}(z)=2^{2{\nu }%
}\Gamma \left( {\nu }+1\right) \Gamma \left( {\nu }+2\right) z^{-2{\nu -1}}%
\mathbf{\Phi }_{{\nu }}(z)$ that this function belongs to $\mathcal{LP}$.
This implies that the function $z\mapsto \mathbf{\Upsilon }_{{\nu }%
}(2z^{\frac{1}{4}})=\mathbf{\tilde{\Upsilon}}_{{\nu }}(z)$ belongs to $\mathcal{LPI}$. Then it follows from Lemma \ref{Jensen} that its Jensen
polynomials%
\begin{equation*}
P_{n}(\mathbf{\tilde{\Upsilon}}_{{\nu }};\zeta )=\sum_{k=0}^{n}{\binom{n}{k}}%
\frac{1}{\left( {\nu +1}\right) _{k}\left( {\nu +2}\right) _{2k}}\left(
-\zeta \right) ^{k}
\end{equation*}%
are all hyperbolic. However, observe that the Jensen polynomials of $z\mapsto\tilde{%
\Psi}_{{\nu }}(z)=\Psi _{{\nu }}(2z^{\frac{1}{4}})$ are simply
\begin{equation*}
\frac{1}{4}P_{n}(\tilde{\Psi}_{{\nu }};\zeta )=\frac{1}{4}\left( {a}%
-1\right) \,P_{n}(\mathbf{\tilde{\Upsilon}}_{{\nu }};\zeta )-\zeta
\,P_{n}^{\prime }(\mathbf{\tilde{\Upsilon}}_{{\nu }};\zeta ).
\end{equation*}%
Lemma \ref{Zeros} implies that all zeros of $P_{n}(\tilde{\Psi}_{{\nu }%
};\zeta )$ are real and positive and that the smallest one precedes the
first zero of $P_{n}(\mathbf{\tilde{\Upsilon}}_{{\nu }};\zeta )$. In view of
Lemma \ref{Jensen}, the latter conclusion immediately yields that $\tilde{%
\Psi}_{{\nu }}\in \mathcal{LPI}$ and that its first zero is less than $%
\gamma _{{\nu },1}$. Finally, the first part of the statement of the lemma
follows after we go back from $\tilde{\Psi}_{{\nu }}$ to $\Psi _{{\nu }}$ by
setting $\zeta =\frac{z^{4}}{16}$.

Similarly, the function $z\mapsto \mathbf{\Omega }_{{\nu }}(z)=2^{2{\nu }%
}\Gamma ^{2}\left( {\nu }+1\right) z^{-2{\nu }}\mathbf{\Pi }_{{\nu }}(z)$
belongs to the Laguerre-P\'{o}lya class of entire functions, which implies
that the function $z\mapsto \mathbf{\Omega }_{{\nu }}(2z^{\frac{1}{4}})=%
\overline{\mathbf{\Omega }}_{{\nu }}(z)$ belongs to $\mathcal{LPI}$.
Then it follows from Lemma \ref{Jensen} that its Jensen polynomials%
\begin{equation*}
P_{n}(\overline{\mathbf{\Omega }}_{{\nu }};\zeta )=\sum_{k=0}^{n}{\binom{n}{k%
}}\frac{1}{\left( {\nu +1}\right) _{k}\left( {\nu +1}\right) _{2k}}\left(
-\zeta \right) ^{k}
\end{equation*}%
are all hyperbolic. However, observe that the Jensen polynomials of $z\mapsto\tilde{%
\phi}_{{\nu }}(z)=\phi _{{\nu }}(2z^{\frac{1}{4}})$ are simply
\begin{equation*}
\frac{1}{4}P_{n}(\tilde{\phi}_{{\nu }};\zeta )=\frac{b}{4}\,P_{n}(\overline{%
\mathbf{\Omega }}_{{\nu }};\zeta )-\zeta \,P_{n}^{\prime }(\overline{\mathbf{%
\Omega }}_{{\nu }};\zeta ).
\end{equation*}%
Lemma \ref{Zeros} implies that all zeros of $P_{n}(\tilde{\phi}_{{\nu }%
};\zeta )$ are real and positive and that the smallest one precedes the
first zero of $P_{n}(\overline{\mathbf{\Omega }}_{{\nu }};\zeta )$. In view
of Lemma \ref{Jensen}, the latter conclusion immediately yields that $\tilde{%
\phi}_{{\nu }}\in \mathcal{LPI}$ and that its first zero precedes $j_{%
{\nu },1}$. Thus, the second part of the statement of this lemma follows
after we go back from $\tilde{\phi}_{{\nu }}$ to $\phi _{{\nu }}$ by setting
$\zeta =\frac{z^{4}}{16}$.
\end{proof}

\subsection{The Hadamard factorization of the derivatives of $\mathbf{\Phi }%
_{{\protect\nu }}$ and $\mathbf{\Pi }_{{\protect\nu }}$}

The following infinite product representations are the Hadamard
factorizations for the derivatives of $\mathbf{\Phi }_{{\nu }}$ and $\mathbf{%
\Pi }_{{\nu }}.$

\begin{lemma}
\label{Had2} For $z\in {\mathbb{C}}$ the Hadamard factorizations of $%
\mathbf{\Phi }_{{\nu }}^{\prime }$ and $\mathbf{\Pi }_{{\nu }}^{\prime }$
are as follows: if $\nu >-\frac{1}{2}$ then%
\begin{equation}
\mathbf{\Phi }_{{\nu }}^{\prime }(z)=\frac{(2{\nu }+1)\left( \frac{z}{2}%
\right) ^{2{\nu }}}{\Gamma \left( {\nu }+1\right) \Gamma \left( {\nu }%
+2\right) }\prod_{n\geq1}\left( 1-\frac{z^{4}}{\gamma _{{\nu }%
,n}^{\prime 4}}\right) ,  \label{Ha1}
\end{equation}%
and if ${\nu >0}$ then%
\begin{equation}
\mathbf{\Pi }_{{\nu }}^{\prime }(z)=\frac{{\nu }\left( \frac{z}{2}\right) ^{2%
{\nu -1}}}{\Gamma ^{2}\left( {\nu }+1\right) }\prod_{n\geq1}\left( 1-\frac{z^{4}}{t_{{\nu },n}^{2}}\right) ,  \label{Ha2}
\end{equation}%
where $\gamma _{\nu ,n}^{\prime },$and $t_{\nu ,n}$ are the $n$th positive
zeros of the functions $\mathbf{\Phi }_{{\nu }}^{\prime }$ and $\mathbf{\Pi }%
_{{\nu }}^{\prime },$ respectively. Moreover, if $\nu >-\frac{1}{2}$ then the zeros $\gamma _{\nu ,n}$ and $\gamma
_{\nu ,n}^{\prime }$ interlace; and if ${\nu >0}$ then the
zeros $j_{\nu ,n}$ and $t_{\nu ,n}$ interlace.
\end{lemma}

\begin{proof}[\bf Proof]
From (\ref{fi})\ we have that%
\begin{equation*}
\frac{1}{2{\nu }+1}\Gamma \left( {\nu }+1\right) \Gamma \left( {\nu }%
+2\right) 2^{2{\nu }}z^{-2{\nu }}\mathbf{\Phi }_{{\nu }}^{\prime }(z)=\frac{1%
}{2{\nu }+1}\sum_{n\geq0}\frac{\left( -1\right) ^{n}\left( 2{\nu }%
+4n+1\right) }{2^{4n}n!\left( {\nu +1}\right) _{n}\left( {\nu +2}\right)
_{2n}}z^{4n}
\end{equation*}%
and%
\begin{equation*}
\lim_{n\rightarrow \infty }\frac{n\log n}{4n\log 2+\log \Gamma \left( {n}%
+1\right) +\log \left( {\nu +1}\right) _{n}+\log \left( {\nu +2}\right)
_{2n}-\log \left( 2{\nu }+4n+1\right) }=\frac{1}{4}.
\end{equation*}%
Here we used $n!=\Gamma \left( {n}+1\right) $ and
\begin{equation*}
\lim_{n\longrightarrow \infty }\frac{\log \Gamma \left( bn+c\right) }{n\log n%
}=b,
\end{equation*}%
where $b$ and $c$ are positive constants. In view of the formula of the growth order of entire functions \cite[p. 6]{LV} we infer that
the above entire function is of growth order $\rho =\frac{1}{4}$, and thus by
applying Hadamard's Theorem \cite[p. 26]{LV} the infinite product
representation of $\mathbf{\Phi }_{{\nu }}^{\prime }$ can be written indeed as in (%
\ref{Ha1}).

According to Theorem \ref{thzeros} the function $z\mapsto\mathbf{\Upsilon }_{{\nu }}(z)=2^{2{\nu }}\Gamma \left( {\nu }%
+1\right) \Gamma \left( {\nu }+2\right) z^{-2{\nu -1}}\mathbf{\Phi }_{{\nu }%
}(z)$ belongs to $\mathcal{LP}$ (since the exponential factors in the
infinite product are canceled because of the symmetry of the zeros $\pm
\gamma _{\nu ,n},$ $n\in \mathbb{N},$ with respect to the origin), it
follows that it satisfies the Laguerre inequality (see \cite{skov})
\begin{equation}
\left( \mathbf{\Upsilon }_{\nu }^{(n)}(x)\right) ^{2}-\mathbf{%
\Upsilon }_{\nu }^{(n-1)}(x) \mathbf{\Upsilon }_{\nu
}^{(n+1)}(x) >0,  \label{2.3}
\end{equation}%
where $\nu >-\frac{1}{2}$ and $z\in \mathbb{R}$. On the other hand, we have that%
\begin{equation*}
\mathbf{\Upsilon }_{\nu }^{\prime }(z)=2^{2{\nu }}\Gamma \left( {\nu }%
+1\right) \Gamma \left( {\nu }+2\right) z^{-2{\nu -2}}\left[ z\mathbf{\Phi }%
_{{\nu }}^{\prime }(z)-\left( 2{\nu }+1\right) \mathbf{\Phi }_{{\nu }}(z)%
\right] ,
\end{equation*}%
\begin{equation*}
\mathbf{\Upsilon }_{\nu }^{\prime \prime }(z)=2^{2{\nu }}\Gamma \left( {\nu }%
+1\right) \Gamma \left( {\nu }+2\right) z^{-2{\nu -3}}\left[ z^{2}\mathbf{%
\Phi }_{{\nu }}^{\prime \prime }(z)-2\left( 2{\nu }+1\right) z\mathbf{\Phi }%
_{{\nu }}^{\prime }(z)+\left( 2{\nu }+1\right) \left( 2{\nu }+2\right)
\mathbf{\Phi }_{{\nu }}(z)\right]
\end{equation*}%
and thus the Laguerre inequality (\ref{2.3}) for $n=1$ is equivalent to%
\begin{equation*}
2^{4{\nu }}\Gamma ^{2}\left( {\nu }+1\right) \Gamma ^{2}\left( {\nu }%
+2\right) z^{-4{\nu -4}}\left[ z^{2}\left( \mathbf{\Phi }_{{\nu }}^{\prime
}(z)\right) ^{2}-z^{2}\mathbf{\Phi }_{{\nu }}(z)\mathbf{\Phi }_{{\nu }%
}^{\prime \prime }(z)-\left( 2\nu +1\right) \left( \mathbf{\Phi }_{{\nu }%
}(z)\right) ^{2}\right] >0.
\end{equation*}%
This implies that%
\begin{equation}\label{contr}
\left( \mathbf{\Phi }_{{\nu }}^{\prime }(z)\right) ^{2}-\mathbf{\Phi }_{{\nu
}}(z)\mathbf{\Phi }_{{\nu }}^{\prime \prime }(z)>\frac{\left( 2\nu +1\right)
\left( \mathbf{\Phi }_{{\nu }}(z)\right) ^{2}}{z^{2}}>0
\end{equation}
for $\nu >-\frac{1}{2}$ and $z\in \mathbb{R}$, and thus $z\mapsto \mathbf{\Phi
}_{{\nu }}^{\prime }(z)/\mathbf{\Phi }_{{\nu }}(z)$ is decreasing on $%
(0,\infty )\backslash \left\{ \gamma _{\nu ,n}:n\in \mathbb{N}\right\} .$
Since the zeros $\gamma _{\nu ,n}$ of the function $\mathbf{\Phi }_{{\nu }}$
are real and simple\footnote{If the zeros $\gamma_{\nu,n}$ would be not simple,
let us suppose that are of multiplicity two for example, then we would have a
contradiction with the inequality \eqref{contr}.}, $\mathbf{\Phi }_{{\nu }}^{\prime }(z)$ does not vanish
in $\gamma _{\nu ,n},$ $n\in \mathbb{N}$. Thus, for a fixed $k\in \mathbb{N}$
the function $z\mapsto \mathbf{\Phi }_{{\nu }}^{\prime }(z)/\mathbf{\Phi
}_{{\nu }}(z)$ takes the limit $\infty $ when $z\searrow \gamma _{\nu ,k-1},$
and the limit $-\infty $ when $z\nearrow \gamma _{\nu ,k}.$ Moreover, since $%
z\mapsto \mathbf{\Phi }_{{\nu }}^{\prime }(z)/\mathbf{\Phi }_{{\nu }}(z)$
is decreasing on $(0,\infty )\backslash \left\{ \gamma _{\nu ,n}:n\in
\mathbb{N}\right\} $ it results that in each interval $(\gamma _{\nu
,k-1},\gamma _{\nu ,k})$ its restriction intersects the horizontal line only
once, and the abscissa of this intersection point is exactly $\gamma _{\nu
,k}^{\prime }$. Consequently, the zeros $\gamma _{\nu ,n}$ and $\gamma _{\nu
,n}^{\prime }$ interlace. Here we used the convention $\gamma _{\nu
,0}=0$.

By means of (\ref{pi1}) we have%
\begin{equation*}
\frac{1}{2{\nu }}\Gamma ^{2}\left( {\nu }+1\right) 2^{2{\nu }}z^{-2{\nu +1}}%
\mathbf{\Pi }_{{\nu }}^{\prime }(z)=\frac{1}{2{\nu }}\sum_{n\geq0}%
\frac{\left( -1\right) ^{n}\left( 2{\nu }+4n\right) }{2^{4n}n!\left( {\nu +1}%
\right) _{n}\left( {\nu +1}\right) _{2n}}z^{4n}
\end{equation*}%
of which growth order can be written as
\begin{equation*}
\lim_{n\rightarrow \infty }\frac{n\log n}{4n\log 2+\log \Gamma \left( {n}%
+1\right) +\log \left( {\nu +1}\right) _{n}+\log \left( {\nu +1}\right)
_{2n}-\log \left( 2{\nu }+4n\right) }=\frac{1}{4}.
\end{equation*}%
By Hadamard's Theorem \cite[p. 26]{LV} the infinite product
representation of $\mathbf{\Pi }_{{\nu }}^{\prime }$ is exactly as in (\ref%
{Ha2}).

According to Theorem \ref{thzeros} the function $z\mapsto \mathbf{\Omega }_{{\nu }}(z)=2^{2{\nu }}\Gamma
^{2}\left( {\nu }+1\right) z^{-2{\nu }}\mathbf{\Pi }_{{\nu }}(z)$ belongs to
$\mathcal{LP}$ (since the exponential factors in the infinite product are
canceled because of the symmetry of the zeros $\pm j_{\nu ,n},$ $n\in
\mathbb{N},$ with respect to the origin). On the other hand, we have that%
\begin{equation*}
\mathbf{\Omega }_{\nu }^{\prime }(z)=2^{2{\nu }}\Gamma ^{2}\left( {\nu }%
+1\right) z^{-2{\nu -1}}\left[ z\mathbf{\Pi }_{{\nu }}^{\prime }(z)-2{\nu }%
\mathbf{\Pi }_{{\nu }}(z)\right] ,
\end{equation*}%
\begin{equation*}
\mathbf{\Omega }_{\nu }^{\prime \prime }(z)=2^{2{\nu }}\Gamma ^{2}\left( {%
\nu }+1\right) z^{-2{\nu -2}}\left[ z^{2}\mathbf{\Pi }_{{\nu }}^{\prime
\prime }(z)-4{\nu }z\mathbf{\Pi }_{{\nu }}^{\prime }(z)+2{\nu }\left( 2{\nu }%
+1\right) \mathbf{\Pi }_{{\nu }}(z)\right]
\end{equation*}%
and thus the Laguerre inequality (see \cite{skov})
$$\left( \mathbf{\Omega}_{\nu }^{(n)}(x)\right) ^{2}-\mathbf{%
\Omega}_{\nu }^{(n-1)}(x)\mathbf{\Omega}_{\nu
}^{(n+1)}(x)>0$$
for $n=1$ is equivalent to%
\begin{equation*}
2^{4{\nu }}\Gamma ^{4}\left( {\nu }+1\right) z^{-4{\nu -2}}\left[
z^{2}\left( \mathbf{\Pi }_{{\nu }}^{\prime }(z)\right) ^{2}-z^{2}\mathbf{\Pi
}_{{\nu }}(z)\mathbf{\Pi }_{{\nu }}^{\prime \prime }(z)-2\nu \left( \mathbf{%
\Pi }_{{\nu }}(z)\right) ^{2}\right] >0.
\end{equation*}%
This implies that%
\begin{equation*}
\left( \mathbf{\Pi }_{{\nu }}^{\prime }(z)\right) ^{2}-\mathbf{\Pi }_{{\nu }%
}(z)\mathbf{\Pi }_{{\nu }}^{\prime \prime }(z)>\frac{2\nu \left( \mathbf{\Pi
}_{{\nu }}(z)\right) ^{2}}{z^{2}}>0
\end{equation*}%
for ${\nu >0}$ and $z\in \mathbb{R}$, and thus $z\mapsto \mathbf{\Pi }_{{%
\nu }}^{\prime }(z)/\mathbf{\Pi }_{{\nu }}(z)$ is decreasing on $(0,\infty
)\backslash \left\{ j_{\nu ,n}:n\in \mathbb{N}\right\} .$ Since the zeros $%
j_{\nu ,n}$ of the function $\mathbf{\Pi }_{{\nu }}$ are real and simple, $%
\mathbf{\Pi }_{{\nu }}^{\prime }(z)$ does not vanish in $j_{\nu ,n},$ $n\in
\mathbb{N}$. Thus, for a fixed $k\in \mathbb{N}$ the function $z\mapsto
\mathbf{\Pi }_{{\nu }}^{\prime }(z)/\mathbf{\Pi }_{{\nu }}(z)$ takes the
limit $\infty $ when $z\searrow j_{\nu ,k-1},$ and the limit $-\infty $ when
$z\nearrow j_{\nu ,k}.$ Moreover, since $z\mapsto \mathbf{\Pi }_{{\nu }%
}^{\prime }(z)/\mathbf{\Pi }_{{\nu }}(z)$ is decreasing on $(0,\infty
)\backslash \left\{ j_{\nu ,n}:n\in \mathbb{N}\right\} $ it results that in
each interval $(j_{\nu ,k-1},j_{\nu ,k})$ its restriction intersects the
horizontal line only once, and the abscissa of this intersection point is
exactly $t_{\nu ,k}$. Consequently, the zeros $j_{\nu ,n}$ and $t_{\nu ,n}$
interlace. Here we used the convention $j_{\nu ,0}=0$.
\end{proof}

\subsection{The Hadamard factorization of the derivatives of $g_{{\protect%
\nu }}$, $h_{{\protect\nu }}$, $v_{{\protect\nu }}$ and $w_{{\protect\nu }}$}
The next result contains the Hadamard factorizations of the derivatives of
$g_{\nu},$ $h_{\nu},$ $v_{\nu}$ and $w_{\nu},$ which by means of Lemma \ref{LP} belong to the
Laguerre-P\'olya class $\mathcal{LP}$ and therefore when $\nu>-1$ they have only real zeros.

\begin{lemma}
\label{Had3} If ${\nu >-1}$ then the functions $g_{{\nu }}^{\prime }$, $h_{{%
\nu }}^{\prime }$, $v_{{\nu }}^{\prime }$ and $w_{{\nu }}^{\prime }$ are
entire functions of order $\rho =\frac{1}{4}$. Consequently, their Hadamard
factorizations for $z\in {\mathbb{C}}$ are of the form%
\begin{equation*}
g_{{\nu }}^{\prime }(z)=\prod\limits_{n\geq1}\left( 1-\frac{z^{4}}{%
\zeta _{{\nu },n}^{4}}\right) \text{, \ \ \ }h_{{\nu }}^{\prime
}(z)=\prod\limits_{n\geq1}\left( 1-\frac{z}{\xi _{{\nu },n}^{4}}%
\right)
\end{equation*}%
and%
\begin{equation*}
v_{{\nu }}^{\prime }(z)=\prod\limits_{n\geq1}\left( 1-\frac{z^{4}}{%
\vartheta _{{\nu },n}^{4}}\right) \text{, \ \ \ }w_{{\nu }}^{\prime
}(z)=\prod\limits_{n\geq1}\left( 1-\frac{z}{\omega _{{\nu },n}^{4}}%
\right) ,
\end{equation*}%
where $\zeta _{{\nu },n}$ and $\xi _{{\nu },n}$\ are the $n$th positive zeros
of $z\mapsto z\mathbf{\Phi }_{{\nu }}^{\prime }(z)-2{\nu }\mathbf{\Phi }%
_{{\nu }}(z)$ and $z\mapsto z\mathbf{\Phi }_{{\nu }}^{\prime }(z)-(2{\nu
-3)}\mathbf{\Phi }_{{\nu }}(z)$\ while $\vartheta _{{\nu },n}$ and $\omega _{%
{\nu },n}$ are the $n$th positive zeros of $z\mapsto z\mathbf{\Pi }_{{\nu }%
}^{\prime }(z)-(2{\nu -1})\mathbf{\Pi }_{{\nu }}(z)$ and $z\mapsto z%
\mathbf{\Pi }_{{\nu }}^{\prime }(z)-(2{\nu -4})\mathbf{\Pi }_{{\nu }}(z)$.
\end{lemma}

\begin{proof}[\bf Proof]
We have that%
\begin{equation*}
g_{{\nu }}^{\prime }(z)=2^{2{\nu }}\Gamma \left( {\nu }+1\right) \Gamma
\left( {\nu }+2\right) z^{-2{\nu -1}}\left[ z\mathbf{\Phi }_{{\nu }}^{\prime
}(z)-2{\nu }\mathbf{\Phi }_{{\nu }}(z)\right] =\sum_{n\geq0}\frac{%
\left( -1\right) ^{n}(4n+1)z^{4n}}{n!2^{4n}\left( {\nu +1}\right) _{n}\left(
{\nu +2}\right) _{2n}},
\end{equation*}%
\begin{equation*}
h_{{\nu }}^{\prime }(z)=2^{2{\nu -2}}\Gamma \left( {\nu }+1\right) \Gamma
\left( {\nu }+2\right) z^{-\frac{{\nu }}{2}{-}\frac{{1}}{4}}\left[ z^{\frac{1}{4}}%
\mathbf{\Phi }_{{\nu }}^{\prime }(z^{\frac{1}{4}})-(2{\nu -3)}\mathbf{\Phi }_{{\nu }%
}(z^{\frac{1}{4}})\right] =\sum_{n\geq0}\frac{\left( -1\right) ^{n}(n+1)z^{n}%
}{n!2^{4n}\left( {\nu +1}\right) _{n}\left( {\nu +2}\right) _{2n}},
\end{equation*}%
\begin{equation*}
v_{{\nu }}^{\prime }(z)=2^{2{\nu }}\Gamma ^{2}\left( {\nu }+1\right) z^{-2{%
\nu }}\left[ z\mathbf{\Pi }_{{\nu }}^{\prime }(z)-(2{\nu -1})\mathbf{\Pi }_{{%
\nu }}(z)\right] =\sum_{n\geq0}\frac{\left( -1\right) ^{n}(4n+1)z^{4n}%
}{n!2^{4n}\left( {\nu +1}\right) _{n}\left( {\nu +1}\right) _{2n}},
\end{equation*}%
\begin{equation*}
w_{{\nu }}^{\prime }(z)=2^{2{\nu -2}}\Gamma ^{2}\left( {\nu }+1\right) z^{-%
\frac{{\nu }}{2}}\left[ z^{\frac{1}{4}}\mathbf{\Pi }_{{\nu }}^{\prime }(z^{\frac{1}{4}})-(2{%
\nu -4})\mathbf{\Pi }_{{\nu }}(z^{\frac{1}{4}})\right] =\sum_{n\geq0}\frac{%
\left( -1\right) ^{n}(n+1)z^{n}}{n!2^{4n}\left( {\nu +1}\right) _{n}\left( {%
\nu +1}\right) _{2n}}
\end{equation*}%
and since
\begin{equation*}
\lim_{n\rightarrow \infty }\frac{n\log n}{4n\log 2+\log \Gamma \left( {n}%
+1\right) +\log \left( {\nu +1}\right) _{n}+\log \left( {\nu +a}\right)
_{2n}-\log \left( bn+1\right) }=\frac{1}{4},
\end{equation*}%
for each $a,b>0,$ it follows that each of the above entire functions have
growth order $\frac{1}{4}$ and hence their genus is zero.
Moreover, we know that the zeros $\zeta _{{\nu },n}$, $\xi _{{\nu },n},$ $%
\vartheta _{{\nu },n}$ and $\omega _{{\nu },n},$ $n\in \mathbb{N}$, are real
according to Lemma \ref{LP}, and with this the rest of the proof follows by
applying Hadamard's Theorem \cite[p. 26]{LV}.
\end{proof}

\subsection{A monotonicity result on the zeros of Bessel functions} We know that the equation $J_\nu(1)=0$ has a single root
$\nu^{\circ}=-0.77{\dots}.$ If $\nu>\nu^{\circ},$ then the smallest positive
root of the equation  $J_{\nu}(z)=0$ is $j_{\nu,1}$ and
$j_{\nu,1}>1$. Moreover, if $\nu<\nu^{\circ},$ then $j_{\nu,1}\in(0,
1).$ The mapping $\zeta:(\nu^{\circ},\infty)\rightarrow(0,\infty),$ defined
by
$$\zeta(\nu)=\sum_{n\geq1}\frac{1}{j_{\nu,n}^4-1}+\sum_{n\geq1}\frac{j_{\nu,n}^4}{(j_{\nu,n}^4-1)^2},$$
is strictly decreasing and
$\lim\limits_{\nu\rightarrow\nu^{\circ}}\zeta(\nu)=+\infty,$
$\lim\limits_{\nu\rightarrow\infty}\zeta(\nu)=0.$ Consequently the
equation
$$1=\sum_{n\geq1}\frac{1}{j_{\nu,n}^4-1}+\sum_{n\geq1}\frac{j_{\nu,n}^4}{(j_{\nu,n}^4-1)^2}$$
has a unique root $\underline{\nu}\in(\nu^{\circ},\infty).$ Moreover, it is possible to show that $\underline{\nu}\in\left(\nu^{\circ},-\frac{1}{2}\right).$
By using the second Rayleigh sum for the zeros of Bessel functions of the first kind
$$\sum_{n\geq1}\frac{1}{j^4_{\nu,n}}=\frac{1}{2^4(\nu+1)^2(\nu+2)}$$
we have for $\nu>-\frac{1}{2}$ that
$$\sum_{n\geq1}\frac{1}{j^4_{\nu,n}}\leq\frac{1}{2^4(1-\frac{1}{2})^2(2-\frac{1}{2})}=\frac{1}{6}.$$
This implies that $j^4_{\nu,n}>6$ for $\nu>-\frac{1}{2}$ and $n\in\mathbb{N},$ and this inequality is equivalent to $$\frac{1}{j^4_{\nu,n}-1}<\frac{6}{5}\cdot\frac{1}{j^4_{\nu,n}}.$$ Thus, we get
\begin{equation}\label{c2b3mm3d4f5g}\sum_{n\geq 1}\frac{1}{j^4_{\nu,n}-1}<
\frac{6}{5}\sum_{n\geq1}\frac{1}{j^4_{\nu,n}}<\frac{6}{5}\cdot\frac{1}{6}=\frac{1}{5},\end{equation}
which implies that $\underline{\nu}\in\left(\nu^{\circ},-\frac{1}{2}\right).$

\begin{lemma}\label{lemzerosBessel}
The function $\Theta:(\underline{\nu},\infty)\rightarrow\mathbb{R},$ defined by
\begin{equation}\label{3m3d4f5g}
\Theta(\nu)=1-\sum_{n\geq1}\frac{1}{j_{\nu,n}^4-1}-\frac{\sum\limits_{n\geq1}\frac{j_{\nu,n}^4}{(j_{\nu,n}^4-1)^2}}{1-\sum\limits_{n\geq1}\frac{1}{j_{\nu,n}^4-1}},
\end{equation}
is strictly increasing.
\end{lemma}

\begin{proof}[\bf Proof]
We have to show that the functions $\Theta_1, \Theta_2:(\underline{\nu},\infty)\rightarrow\mathbb{R},$ defined by
$$\Theta_1(\nu)=\sum_{n\geq1}\frac{1}{j_{\nu,n}^4-1},\ \ \ \Theta_2(\nu)=\frac{\sum\limits_{n\geq1}\frac{j_{\nu,n}^4}{(j_{\nu,n}^4-1)^2}}{1-\sum\limits_{n\geq1}\frac{1}{j_{\nu,n}^4-1}},$$
are strictly decreasing. Since $j_{\nu,n}$ is strictly increasing with respect to $\nu$ it follows that $\Theta_1$ is strictly decreasing.
On the other hand we have $$\Theta'_2(\nu)=\frac{-\Delta(\nu)}{\left(1-\sum\limits_{n\geq1}\frac{1}{j_{\nu,n}^4-1}\right)^2},$$
where
$$\Delta(\nu)=\sum_{n\geq1}\frac{4j_{\nu,n}^3\frac{d j_{\nu,n}}{d\nu}(j_{\nu,n}^4+1)}{(j_{\nu,n}^4-1)^3}\left(1-\sum_{n\geq1}\frac{1}{j_{\nu,n}^4-1}\right)+
\sum_{n\geq1}\frac{4j_{\nu,n}^3\frac{d j_{\nu,n}}{d\nu}}{(j_{\nu,n}^4-1)^2}\sum_{n\geq1}\frac{j_{\nu,n}^4}{(j_{\nu,n}^4-1)^2}.$$
By using the inequality \eqref{c2b3mm3d4f5g} and the fact that $\frac{dj_{\nu,n}}{d\nu}>0$ for each $n\in\mathbb{N}$ and $\nu>-1,$ it follows that
$\Delta(\nu)>0$ and $\Theta_2'(\nu)<0$ for $\nu\in(\underline{\nu},\infty).$
\end{proof}

\section{Proofs of the main results}
\setcounter{equation}{0}

\begin{proof}[\bf Proof of Theorem \ref{thzeros}]
We consider the entire function $f(z)=\sum_{n\geq 0}a_nz^n.$ According to Runckel \cite[Theorem 4]{runckel}
we know that if $f$ can be represented as $f(z)=e^{az^2}h(z),$ where $a\leq0$ and $h$ is of type
$$h(z)=ce^{bz}\prod_{n\geq1}\left(1-\frac{z}{c_n}\right)e^{\frac{z}{c_n}}$$ with $\sum_{n\geq1}|c_n|^{-2}<\infty,$
$f$ has real zeros only (or no zeros at all), and $G$ is of type
\begin{equation}\label{typeA}G(z)=e^{\beta z}\prod_{n\geq1}\left(1+\frac{z}{\alpha_n}\right)e^{-\frac{z}{\alpha_n}}\end{equation} with $\alpha_n>0,$ $\beta\in\mathbb{R},$ $\sum_{n\geq1}\alpha_n^{-2}<\infty,$ then the function
$F(z)=\sum_{n\geq0}a_nG(n)z^n$ has real zeros only. Now, consider the function
\begin{equation}\label{zerof1}2^{\nu-\frac{1}{2}}z^{-\nu-\frac{1}{2}}\mathbf{\Phi}_{\nu}(\sqrt{2z})=
\sum_{n\geq0}\frac{(-1)^n\left(\frac{z}{2}\right)^{2n}}{n!\Gamma(\nu+n+1)\Gamma(\nu+2n+2)}.\end{equation}
Since $$\frac{1}{\Gamma(z)}=ze^{\gamma z}\prod_{n\geq 1}\left(1+\frac{z}{n}\right)e^{-\frac{z}{n}},$$ the function
$$G(z)=\frac{1}{\Gamma\left(\frac{z}{2}+\nu+1\right)\Gamma(z+\nu+2)}$$
is of type \eqref{typeA} when $\nu>-1.$ Thus, if we choose $f(z)=e^{-\left(\frac{z}{2}\right)^2},$ then by using Runckel's above mentioned result we obtain that the function in \eqref{zerof1} has real zeros only if $\nu>-1.$ This implies that the function $\mathbf{\Phi}_{\nu}$ has also real zeros only if $\nu>-1.$ Now, by choosing the same exponential function $f$ and
taking $$G(z)=\frac{1}{\Gamma\left(\frac{z}{2}+\nu+1\right)\Gamma(z+\nu+1)},$$
which is of type \eqref{typeA} when $\nu>-1,$ a similar argument as before shows that the function
$$2^{\nu}z^{-\nu}\mathbf{\Pi}_{\nu}(\sqrt{2z})=
\sum_{n\geq0}\frac{(-1)^n\left(\frac{z}{2}\right)^{2n}}{n!\Gamma(\nu+n+1)\Gamma(\nu+2n+1)}$$
has only real zeros when $\nu>-1$ and consequently $\mathbf{\Pi}_{\nu}$ has also only real zeros if $\nu>-1.$
\end{proof}

\begin{proof}[\bf Proof of Theorem \ref{thhyper}]
We consider the entire functions
$$\frac{1}{2}z^{-\frac{\nu}{2}-\frac{1}{4}}\mathbf{\Phi}_{\nu}(2\sqrt[4]{z})=\sum_{n\geq 0}\frac{(-1)^nz^n}{n!\Gamma(\nu+n+1)\Gamma(\nu+2n+2)}$$
and
$$z^{-\frac{\nu}{2}}\mathbf{\Pi}_{\nu}(2\sqrt[4]{z})=\sum_{n\geq 0}\frac{(-1)^nz^n}{n!\Gamma(\nu+n+1)\Gamma(\nu+2n+1)},$$
which according to Theorem \ref{thzeros} belong to the class $\mathcal{LP}.$ In view of the relation $(a)_{2n}=2^{2n}\left(\frac{a}{2}\right)_n\left(\frac{a+1}{2}\right)_n$ the Jensen polynomials of the above entire functions can be written as
\begin{align*}\sum_{k=0}^nC_n^k\frac{(-1)^kz^k}{\Gamma(\nu+k+1)\Gamma(\nu+2k+2)}
&=\frac{\nu+1}{\Gamma^2(\nu+2)}\sum_{k=0}^n\frac{(-n)_kz^k}{(\nu+1)_k(\nu+2)_{2k}k!}\\
&=\frac{\nu+1}{\Gamma^2(\nu+2)}\sum_{k=0}^n\frac{(-n)_k}{(\nu+1)_k\left(\frac{\nu+2}{2}\right)_{k}
\left(\frac{\nu+3}{2}\right)_{k}}\frac{\left(\frac{z}{4}\right)^k}{k!}\\
&=\frac{\nu+1}{\Gamma^2(\nu+2)}\cdot {}_1F_3\left(-n;\nu+1,\frac{\nu+2}{2},\frac{\nu+3}{2};\frac{z}{4}\right)\end{align*}
and
\begin{align*}\sum_{k=0}^nC_n^k\frac{(-1)^kz^k}{\Gamma(\nu+k+1)\Gamma(\nu+2k+1)}
&=\frac{1}{\Gamma^2(\nu+1)}\sum_{k=0}^n\frac{(-n)_kz^k}{(\nu+1)_k(\nu+1)_{2k}k!}\\
&=\frac{1}{\Gamma^2(\nu+1)}\sum_{k=0}^n\frac{(-n)_k}{(\nu+1)_k\left(\frac{\nu+1}{2}\right)_{k}
\left(\frac{\nu+2}{2}\right)_{k}}\frac{\left(\frac{z}{4}\right)^k}{k!}\\
&=\frac{1}{\Gamma^2(\nu+1)}\cdot {}_1F_3\left(-n;\nu+1,\frac{\nu+1}{2},\frac{\nu+2}{2};\frac{z}{4}\right).\end{align*}
Thus, by using Lemma \ref{Jensen} it follows that the zeros of the above hypergeometric polynomials are all real.
\end{proof}

\begin{proof}[\bf Proof of Theorem \ref{T1}]
First we prove part \textbf{a} for ${\nu >-\frac{1}{2}}$ and parts \textbf{b} and
\textbf{c} for ${\nu >-1}$. We need to show that for the corresponding
values of ${\nu }$ and $\alpha $ the inequalities
\begin{equation}
\real \left( \frac{zf_{{\nu }}^{\prime }(z)}{f_{{\nu }}(z)}\right) >\alpha ,%
\text{ \ \ }\real \left( \frac{zg_{{\nu }}^{\prime }(z)}{g_{{\nu }}(z)}\right)
>\alpha \text{ \ and \ }\real \left( \frac{zh_{{\nu }}^{\prime }(z)}{h_{{\nu }%
}(z)}\right) >\alpha \text{ \ \ }  \label{2.0}
\end{equation}%
are valid for $z\in \mathbb{D}_{r_{\alpha }^{\star }(f_{{\nu }})}$, $z\in
\mathbb{D}_{r_{\alpha }^{\star }(g_{{\nu }})}$ and $z\in \mathbb{D}%
_{r_{\alpha }^{\star }(h_{{\nu }})}$ respectively, and each of the above
inequalities does not hold in larger disks. It follows from (\ref{fi1}) that%
\begin{equation*}
\frac{zf_{{\nu }}^{\prime }(z)}{f_{{\nu }}(z)}=\frac{1}{2{\nu +1}}\frac{z%
\mathbf{\Phi }_{{\nu }}^{\prime }(z)}{\mathbf{\Phi }_{{\nu }}(z)}=1-\frac{1}{%
2{\nu +1}}\sum_{n\geq1}\frac{4z^{4}}{\gamma _{{\nu },n}^{4}-z^{4}},
\end{equation*}%
\begin{equation*}
\frac{zg_{{\nu }}^{\prime }(z)}{g_{{\nu }}(z)}=-2{\nu }+\frac{z\mathbf{\Phi }%
_{{\nu }}^{\prime }(z)}{\mathbf{\Phi }_{{\nu }}(z)}=1-\sum_{n\geq1}%
\frac{4z^{4}}{\gamma _{{\nu },n}^{4}-z^{4}}
\end{equation*}%
and%
\begin{equation*}
\frac{zh_{{\nu }}^{\prime }(z)}{h_{{\nu }}(z)}=\frac{3}{4}-\frac{{\nu }}{2}+%
\frac{1}{4}\frac{z^{\frac{1}{4}}\mathbf{\Phi }_{{\nu }}^{\prime }(z^{\frac{1}{4}})}{\mathbf{%
\Phi }_{{\nu }}(z^{\frac{1}{4}})}=1-\sum_{n\geq1}\frac{z}{\gamma _{{\nu }%
,n}^{4}-z}.
\end{equation*}%
On the other hand, it is known that \cite{BAS} if ${z\in \mathbb{C}}$ and $%
\beta $ ${\in \mathbb{R}}$ are such that $\beta >{\left\vert z\right\vert }$%
, then%
\begin{equation}
\frac{{\left\vert z\right\vert }}{\beta -{\left\vert z\right\vert }}\geq
\real\left( \frac{z}{\beta -z}\right) .  \label{s7}
\end{equation}%
Then the inequality
\begin{equation*}
\frac{{\left\vert z\right\vert }^{4}}{\gamma _{{\nu },n}^{4}-{\left\vert
z\right\vert }^{4}}\geq \real \left( \frac{z^{4}}{\gamma _{{\nu },n}^{4}-z^{4}}%
\right) ,
\end{equation*}%
holds for every ${\nu >-1}$, $n\in \mathbb{N}$ and ${\left\vert
z\right\vert <}\gamma _{{\nu },1}$. Therefore,
\begin{equation}
\real \left( \frac{zf_{{\nu }}^{\prime }(z)}{f_{{\nu }}(z)}\right) =1-\frac{1}{%
2{\nu +1}}\real \left( \sum_{n\geq1}\frac{4z^{4}}{\gamma _{{\nu }%
,n}^{4}-z^{4}}\right) \geq 1-\frac{1}{2{\nu +1}}\sum_{n\geq1}\frac{%
4\left\vert z\right\vert ^{4}}{\gamma _{{\nu },n}^{4}-\left\vert
z\right\vert ^{4}}=\frac{\left\vert z\right\vert f_{{\nu }}^{\prime
}(\left\vert z\right\vert )}{f_{{\nu }}(\left\vert z\right\vert )},
\label{sf}
\end{equation}%
\begin{equation}
\real \left( \frac{zg_{{\nu }}^{\prime }(z)}{g_{{\nu }}(z)}\right) =1-\real
\left( \sum_{n\geq1}\frac{4z^{4}}{\gamma _{{\nu },n}^{4}-z^{4}}%
\right) \geq 1-\sum_{n\geq1}\frac{4\left\vert z\right\vert ^{4}}{%
\gamma _{{\nu },n}^{4}-\left\vert z\right\vert ^{4}}=\frac{\left\vert
z\right\vert g_{{\nu }}^{\prime }(\left\vert z\right\vert )}{g_{{\nu }%
}(\left\vert z\right\vert )}  \label{sg}
\end{equation}%
and%
\begin{equation}
\real \left( \frac{zh_{{\nu }}^{\prime }(z)}{h_{{\nu }}(z)}\right) =1-\real
\left( \sum_{n\geq1}\frac{z}{\gamma _{{\nu },n}^{4}-z}\right) \geq
1-\sum_{n\geq1}\frac{\left\vert z\right\vert }{\gamma _{{\nu }%
,n}^{4}-\left\vert z\right\vert }=\frac{\left\vert z\right\vert h_{{\nu }%
}^{\prime }(\left\vert z\right\vert )}{h_{{\nu }}(\left\vert z\right\vert )},
\label{sh}
\end{equation}%
where equalities are attained only when $z=\left\vert z\right\vert =r$. The
above inequalities and the minimum principle for harmonic functions imply
that the corresponding inequalities in (\ref{2.0}) hold if and only if $%
\left\vert z\right\vert <x_{{\nu ,\alpha ,1}},$ $\left\vert z\right\vert <y_{%
{\nu ,\alpha ,1}}$ and $\left\vert z\right\vert <z_{{\nu ,\alpha ,1}},$
respectively, where $x_{{\nu ,\alpha ,1}}$, $y_{{\nu ,\alpha ,1}}$ and $z_{{%
\nu ,\alpha ,1}}$ are the smallest positive roots of the equations%
\begin{equation*}
rf_{{\nu }}^{\prime }(r)/f_{{\nu }}(r)=\alpha ,\text{ \ }rg_{{\nu }}^{\prime
}(r)/g_{{\nu }}(r)=\alpha ,\ rh_{{\nu }}^{\prime }(r)/h_{{\nu }}(r)=\alpha .
\end{equation*}%
Since their solutions coincide with the zeros of the functions
\begin{equation*}
r\mapsto r\mathbf{\Phi }_{{\nu }}^{\prime }(r)-\alpha (2{\nu +1)}\mathbf{%
\Phi }_{{\nu }}(r),\ r\mapsto r\mathbf{\Phi }_{{\nu }}^{\prime }(r)-(\alpha
+2{\nu )}\mathbf{\Phi }_{{\nu }}(r),
\end{equation*}%
\begin{equation*}
r\mapsto r^{\frac{1}{4}}\mathbf{\Phi }_{{\nu }}^{\prime }(r^{\frac{1}{4}})-(4\alpha +2{\nu
-3)}\mathbf{\Phi }_{{\nu }}(r^{\frac{1}{4}}),
\end{equation*}%
the result we need follows from Lemma \ref{LP} by taking instead of $a$ the
values $2{\nu }(\alpha {-1)+}\alpha $, $\alpha $ and $4\alpha -3,$
respectively. In other words, Lemma \ref{LP} shows that all the zeros of the
above three functions are real and their first positive zeros do not exceed
the first positive zero $\gamma _{{\nu },1}$. This guarantees that the above
inequalities hold. This completes the proof of part \textbf{a} when ${\nu
>-\frac{1}{2}}$, and parts \textbf{b} and \textbf{c} when ${\nu >-1.}$

Now, to prove the statement for part \textbf{a} when ${\nu }\in \left( -1,-%
\frac{1}{2}\right) $ we observe that the counterpart of (\ref{s7}) is
\begin{equation}
\real\left( \frac{z}{\beta -z}\right) \geq \frac{-{\left\vert
z\right\vert }}{\beta +{\left\vert z\right\vert }}  \label{2.10}
\end{equation}%
and it holds for all ${z\in \mathbb{C}}$ and $\beta $ ${\in \mathbb{R}}$
such that $\beta >{\left\vert z\right\vert }$ (see \cite{sz}). From (\ref%
{2.10}), we obtain the inequality
\begin{equation*}
\real \left( \frac{z^{4}}{\gamma _{{\nu },n}^{4}-z^{4}}\right) \geq \frac{-{%
\left\vert z\right\vert }^{4}}{\gamma _{{\nu },n}^{4}+{\left\vert
z\right\vert }^{4}},
\end{equation*}%
which holds for all ${\nu >-1},$ $n\in \mathbb{N}$ and ${\left\vert
z\right\vert <}\gamma _{{\nu },1}$ and it implies that%
\begin{equation*}
\real \left( \frac{zf_{{\nu }}^{\prime }(z)}{f_{{\nu }}(z)}\right) =1-\frac{1}{%
2{\nu +1}}\real \left( \sum_{n\geq1}\frac{4z^{4}}{\gamma _{{\nu }%
,n}^{4}-z^{4}}\right) \geq 1+\frac{1}{2{\nu +1}}\sum_{n\geq1}\frac{%
4\left\vert z\right\vert ^{4}}{\gamma _{{\nu },n}^{4}+\left\vert
z\right\vert ^{4}}=\frac{\sqrt{i}\left\vert z\right\vert f_{{\nu }}^{\prime
}(\sqrt{i}\left\vert z\right\vert )}{f_{{\nu }}(\sqrt{i}\left\vert
z\right\vert )}.
\end{equation*}%
In this case equality is attained if $z=\sqrt{i}\left\vert z\right\vert
=\sqrt{i}r.$ Moreover, the above inequality implies that
\begin{equation*}
\real \left( \frac{zf_{{\nu }}^{\prime }(z)}{f_{{\nu }}(z)}\right) >\alpha
\end{equation*}%
if and only if $\left\vert z\right\vert <x_{{\nu ,\alpha }}$, where $x_{{\nu
,\alpha }}$ denotes the smallest positive root of $\sqrt{i}rf_{{%
\nu }}^{\prime }(\sqrt{i}r)/f_{{\nu }}(\sqrt{i}r)=\alpha ,$ which is
equivalent to
\begin{equation*}
\sqrt{i}r\mathbf{\Phi }_{{\nu }}^{\prime }(\sqrt{i}r)-\alpha (2{\nu +1)}%
\mathbf{\Phi }_{{\nu }}(\sqrt{i}r)=0,\text{ for }{\nu }\in \left( -1,-\frac{1%
}{2}\right) .
\end{equation*}%
All we need to prove
is that the above equation has actually only one root in $(0,\infty )$.
Observe that, according to Lemma \ref{Quotients}, the function
\begin{equation*}
\left.r\mapsto \frac{\sqrt{i}r\mathbf{\Phi }_{{\nu }}^{\prime }(\sqrt{i}r)}{\mathbf{%
\Phi }_{{\nu }}(\sqrt{i}r)}=\sum_{n\geq0}\frac{\left( 4n+2{\nu +1}%
\right) \left( \frac{r}{2}\right) ^{4n}}{n!\Gamma \left( {\nu }+n+1\right)
\Gamma \left( {\nu }+2n+2\right) }\right/\sum_{n\geq0}\frac{\left(
\frac{r}{2}\right) ^{4n}}{n!\Gamma \left( {\nu }+n+1\right) \Gamma \left( {%
\nu }+2n+2\right) }
\end{equation*}%
is increasing on $(0,\infty )$ as a quotient of two power series whose
positive coefficients form the increasing ``quotient
sequence'' $\left\{ 4n+2{\nu +1}\right\} _{n\geq 0}.$ On
the other hand, the above function tends to $2{\nu +1}$ when $r\rightarrow
0, $ so that its graph can intersect the horizontal line $y=\alpha \left( 2{%
\nu +1}\right) >2{\nu +1}$ only once. This completes the proof of part
\textbf{a} of the theorem when ${\nu }\in \left( -1,-\frac{1}{2}\right) $.
\end{proof}

\begin{proof}[\bf Proof of Theorem \ref{T2}]
The proof of Theorem \ref{T2} is analogous to the proof
of Theorem \ref{T1}. First we prove part \textbf{a} for ${\nu >0}$ and parts \textbf{b} and
\textbf{c} for ${\nu >-1}$. From (\ref{pi1}) we have%
\begin{equation*}
\frac{zu_{{\nu }}^{\prime }(z)}{u_{{\nu }}(z)}=\frac{1}{2{\nu }}\frac{z%
\mathbf{\Pi }_{{\nu }}^{\prime }(z)}{\mathbf{\Pi }_{{\nu }}(z)}=1-\frac{1}{2{%
\nu }}\sum_{n\geq1}\frac{4z^{4}}{j_{{\nu },n}^{4}-z^{4}},
\end{equation*}%
\begin{equation*}
\frac{zv_{{\nu }}^{\prime }(z)}{v_{{\nu }}(z)}={1}-2{\nu }+\frac{z\mathbf{%
\Pi }_{{\nu }}^{\prime }(z)}{\mathbf{\Pi }_{{\nu }}(z)}=1-\sum_{n\geq1}\frac{4z^{4}}{j_{{\nu },n}^{4}-z^{4}}
\end{equation*}%
and%
\begin{equation*}
\frac{zw_{{\nu }}^{\prime }(z)}{w_{{\nu }}(z)}=1-\frac{{\nu }}{2}+\frac{1}{4}%
\frac{z^{\frac{1}{4}}\mathbf{\Pi }_{{\nu }}^{\prime }(z^{\frac{1}{4}})}{\mathbf{\Pi }_{{\nu }%
}(z^{\frac{1}{4}})}=1-\sum_{n\geq1}\frac{z}{j_{{\nu },n}^{4}-z}.
\end{equation*}%
On the other hand, by means of (\ref{s7}) we have the inequality
\begin{equation*}
\frac{{\left\vert z\right\vert }^{4}}{j_{{\nu },n}^{4}-{\left\vert
z\right\vert }^{4}}\geq \real \left( \frac{z^{4}}{j_{{\nu },n}^{4}-z^{4}}%
\right) ,
\end{equation*}%
for every ${\nu >-1}$, $n\in \mathbb{N}$ and ${\left\vert z\right\vert <}j_{{%
\nu },1}$. Therefore,%
\begin{equation}
\real \left( \frac{zu_{{\nu }}^{\prime }(z)}{u_{{\nu }}(z)}\right) \geq 1-%
\frac{1}{2{\nu }}\sum_{n\geq1}\frac{4\left\vert z\right\vert ^{4}}{j_{%
{\nu },n}^{4}-\left\vert z\right\vert ^{4}}=\frac{\left\vert z\right\vert u_{%
{\nu }}^{\prime }(\left\vert z\right\vert )}{u_{{\nu }}(\left\vert
z\right\vert )},  \label{su}
\end{equation}%
\begin{equation}
\real \left( \frac{zv_{{\nu }}^{\prime }(z)}{v_{{\nu }}(z)}\right) \geq
1-\sum_{n\geq1}\frac{4\left\vert z\right\vert ^{4}}{j_{{\nu }%
,n}^{4}-\left\vert z\right\vert ^{4}}=\frac{\left\vert z\right\vert v_{{\nu }%
}^{\prime }(\left\vert z\right\vert )}{v_{{\nu }}(\left\vert z\right\vert )}
\label{sv}
\end{equation}%
and%
\begin{equation}
\real \left( \frac{zw_{{\nu }}^{\prime }(z)}{w_{{\nu }}(z)}\right) \geq
1-\sum_{n\geq1}\frac{\left\vert z\right\vert }{j_{{\nu }%
,n}^{4}-\left\vert z\right\vert }=\frac{\left\vert z\right\vert w_{{\nu }%
}^{\prime }(\left\vert z\right\vert )}{w_{{\nu }}(\left\vert z\right\vert )},
\label{sw}
\end{equation}%
where equalities are attained only when $z=\left\vert z\right\vert =r$. The
above inequalities and the minimum principle for harmonic functions imply
that the inequalities%
\begin{equation*}
\real \left( \frac{zu_{{\nu }}^{\prime }(z)}{u_{{\nu }}(z)}\right) >\alpha ,%
\text{ \ \ }\real \left( \frac{zv_{{\nu }}^{\prime }(z)}{v_{{\nu }}(z)}\right)
>\alpha \text{ \ and \ }\real \left( \frac{zw_{{\nu }}^{\prime }(z)}{w_{{\nu }%
}(z)}\right) >\alpha \text{ \ \ }
\end{equation*}%
hold if and only if $\left\vert z\right\vert <\delta _{{\nu ,\alpha ,1}},$ $%
\left\vert z\right\vert <\rho _{{\nu ,\alpha ,1}}$ and $\left\vert
z\right\vert <\sigma _{{\nu ,\alpha ,1}},$ respectively, where $\delta _{{%
\nu ,\alpha ,1}}$, $\rho _{{\nu ,\alpha ,1}}$ and $\sigma _{{\nu ,\alpha ,1}%
} $ are the smallest positive roots of the equations%
\begin{equation*}
ru_{{\nu }}^{\prime }(r)/u_{{\nu }}(r)=\alpha ,\text{ \ }rv_{{\nu }}^{\prime
}(r)/v_{{\nu }}(r)=\alpha ,\ rw_{{\nu }}^{\prime }(r)/w_{{\nu }}(r)=\alpha .
\end{equation*}%
Since their solutions coincide with the zeros of the functions
\begin{equation*}
r\mapsto r\mathbf{\Pi }_{{\nu }}^{\prime }(r)-2\alpha {\nu }\mathbf{\Pi }_{{%
\nu }}(r),\ r\mapsto r\mathbf{\Pi }_{{\nu }}^{\prime }(r)-(\alpha +2{\nu -1)}%
\mathbf{\Pi }_{{\nu }}(r),
\end{equation*}%
\begin{equation*}
r\mapsto r^{\frac{1}{4}}\mathbf{\Pi }_{{\nu }}^{\prime }(r^{\frac{1}{4}})-2(2\alpha +{\nu -2)%
}\mathbf{\Pi }_{{\nu }}(r^{\frac{1}{4}}),
\end{equation*}%
the result we need follows from Lemma \ref{LP} by taking instead of $b$ the
values $2{\nu }(\alpha {-1)}$, $\alpha -1$ and $4(\alpha -1),$ respectively.
In other words, Lemma \ref{LP} shows that all the zeros of the above three
functions are real and their first positive zeros do not exceed the first
positive zero $j_{{\nu },1}$. This guarantees that the above inequalities
hold. This completes the proof of part \textbf{a} when ${\nu >0}$, and parts
\textbf{b} and \textbf{c} when ${\nu >-1.}$

Now, to prove the statement for part \textbf{a} when ${\nu }\in \left(
-1,0\right) $. From (\ref{2.10}), we obtain the inequality
\begin{equation*}
\real \left( \frac{z^{4}}{\gamma _{{\nu },n}^{4}-z^{4}}\right) \geq \frac{-{%
\left\vert z\right\vert }^{4}}{j_{{\nu },n}^{4}+{\left\vert z\right\vert }%
^{4}},
\end{equation*}%
which holds for all ${\nu >-1},$ $n\in \mathbb{N}$ and ${\left\vert
z\right\vert <}j_{{\nu },1}$ and it implies that%
\begin{equation*}
\real \left( \frac{zu_{{\nu }}^{\prime }(z)}{u_{{\nu }}(z)}\right) \geq 1+%
\frac{1}{2{\nu }}\sum_{n\geq1}\frac{4\left\vert z\right\vert ^{4}}{j_{%
{\nu },n}^{4}+\left\vert z\right\vert ^{4}}=\frac{\sqrt{i}\left\vert
z\right\vert u_{{\nu }}^{\prime }(\sqrt{i}\left\vert z\right\vert )}{u_{{\nu }%
}(\sqrt{i}\left\vert z\right\vert )}.
\end{equation*}%
In this case equality is attained if $z=\sqrt{i}\left\vert z\right\vert
=\sqrt{i}r.$ Moreover, the above inequality implies that
\begin{equation*}
\real \left( \frac{zu_{{\nu }}^{\prime }(z)}{u_{{\nu }}(z)}\right) >\alpha
\end{equation*}%
if and only if $\left\vert z\right\vert <\delta _{{\nu ,\alpha }}$, where $%
\delta _{{\nu ,\alpha }}$ denotes the smallest positive root of
$\sqrt{i}ru_{{\nu }}^{\prime }(\sqrt{i}r)/u_{{\nu }}(\sqrt{i}r)=\alpha ,$ which
is equivalent to
\begin{equation*}
\sqrt{i}r\mathbf{\Pi }_{{\nu }}^{\prime }(\sqrt{i}r)-2\alpha {\nu }\mathbf{\Pi
}_{{\nu }}(\sqrt{i}r)=0,\text{ for }{\nu }\in \left( -1,0\right) .
\end{equation*}%
All we need to prove is that the above
equation has actually only one root in $(0,\infty )$. Observe that,
according to Lemma \ref{Quotients}, the function
\begin{equation*}
\left.r\mapsto \frac{\sqrt{i}r\mathbf{\Pi }_{{\nu }}^{\prime }(\sqrt{i}r)}{\mathbf{%
\Pi }_{{\nu }}(\sqrt{i}r)}=\sum_{n\geq0}\frac{\left( 4n+2{\nu }\right)
\left( \frac{r}{2}\right) ^{4n}}{n!\Gamma \left( {\nu }+n+1\right) \Gamma
\left( {\nu }+2n+1\right) }\right/\sum_{n\geq0}\frac{\left( \frac{r%
}{2}\right) ^{4n}}{n!\Gamma \left( {\nu }+n+1\right) \Gamma \left( {\nu }%
+2n+1\right) }
\end{equation*}%
is increasing on $(0,\infty )$ as a quotient of two power series whose
positive coefficients form the increasing \textquotedblleft quotient
sequence\textquotedblright\ $\left\{ 4n+2{\nu }\right\} _{n\geq 0}.$ On the
other hand, the above function tends to $2{\nu }$ when $r\rightarrow 0,$ so
that its graph can intersect the horizontal line $y=2\alpha {\nu }>2{\nu }$
only once. This completes the proof of part \textbf{a} of the theorem when ${%
\nu }\in \left( -1,0\right) $.
\end{proof}

\begin{proof}[\bf Proof of Theorem \ref{THbound3}]
The radius of starlikeness of $f_{{\nu }}$ is the first positive root of
the equation $f_{{\nu }}^{\prime }(z)=0,$ which is equivalent to $\mathbf{\Phi}_{\nu}'(z)=0.$ On the other hand, by using \eqref{fi1}
we have that
$$\frac{z\mathbf{\Phi}_{\nu}'(z)}{\mathbf{\Phi}_{\nu}(z)}=2\nu+1-\sum_{n\geq 1}\frac{4z^4}{\gamma_{\nu,n}^4-z^4}.$$ Since the above expression vanishes at
$r^{\star}(f_{\nu})$ we obtain that
\begin{equation*}
\frac{1}{(r^{\star }(f_{\nu }))^{4}}=\frac{1}{2\nu +1}\sum_{n\geq 1}\frac{4}{%
\gamma _{{\nu },n}^{4}-(r^{\star }(f_{\nu }))^{4}}>\frac{1}{2\nu +1}%
\sum_{n\geq 1}\frac{4}{\gamma _{{\nu },n}^{4}}=\frac{1}{4(\nu +1)_{3}(2\nu
+1)}.
\end{equation*}%
Here we used the first Rayleigh sum of the zeros $\gamma_{\nu,n}$ computed in \cite{cross}. Now, we consider the infinite sum and product representation of the
function
\begin{equation*}
z\mapsto \mathbf{\varphi }_{\nu }(z)=\frac{1}{2{\nu +1}}2^{2{\nu }}z^{\frac{-%
{\nu }}{2}}\Gamma \left( {\nu }+1\right) \Gamma \left( {\nu }+2\right)
\mathbf{\Phi }_{{\nu }}^{\prime }(\sqrt[4]{z}).
\end{equation*}
We have%
\begin{equation*}
\mathbf{\varphi }_{\nu }(z)=\sum_{n\geq 0}\frac{\left( -1\right) ^{n}\left( 2%
{\nu +4n+1}\right) }{2^{4n}n!\left( 2{\nu +1}\right) \left( {\nu +1}\right)
_{n}\left( {\nu +2}\right) _{2n}}z^{n}
\end{equation*}%
and%
\begin{equation*}
\mathbf{\varphi }_{\nu }(z)=\prod\limits_{n\geq 1}\left( 1-\frac{z}{\gamma _{%
{\nu },n}^{\prime 4}}\right),
\end{equation*}
according to Lemma \ref{Had2}. By using the Euler-Rayleigh inequalities $\tau _{k}^{-\frac{1}{k}}<\gamma _{{%
\nu },1}'^4<\frac{\tau _{k}}{\tau _{k+1}},$ where $\tau_k=\sum_{n\geq1}\gamma_{\nu,n}'^{-4k},$ we obtain the
following inequalities for $\nu >-\frac{1}{2}$ and $k\in \mathbb{N}$
\begin{equation*}
\sqrt[4]{\tau _{k}^{-\frac{1}{k}}}<r^{\star }(f_{\nu })<\sqrt[4]{\frac{\tau
_{k}}{\tau _{k+1}}}.
\end{equation*}%
Since%
\begin{equation*}
\tau _{1}=\frac{2{\nu +5}}{16\left( 2{\nu +1}\right) (\nu +1)(\nu +2)(\nu +3)%
},\text{ \ }\tau _{2}=\frac{20\nu ^{3}+184\nu ^{2}+529\nu +473}{2^{8}\left(
2\nu +1\right) ^{2}(\nu +1)_{3}(\nu +1)_{5}},
\end{equation*}%
\begin{equation*}
\tau _{3}=\frac{\nu _{1}}{\allowbreak 2^{11}\left( 2\nu +1\right)
^{3}\left((\nu +1)_{3}\right) ^{2}(\nu +1)_{7}}
\end{equation*}%
and%
\begin{equation*}
\tau _{4}=\frac{\nu _{2}}{2^{16}\left( 2\nu +1\right) ^{4}\left(
(\nu +1)_{3}\right) ^{2}(\nu +1)_{5}(\nu +1)_{9}},
\end{equation*}%
where%
\begin{equation*}
\nu _{1}=168\nu ^{5}+2876\nu ^{4}+18\,590\nu ^{3}+57\,349\nu
^{2}+84\,874\nu +48\,267
\end{equation*}%
and%
\begin{eqnarray*}
\nu _{2}&=&6864\nu ^{9}+245\,792\nu ^{8}+3802\,808\nu
^{7}+33\,438\,984\nu ^{6}+184\,372\,941\nu ^{5}+661\,304\,856\nu ^{4}\\
&&+1542\,867\,228\nu ^{3}+2256\,870\,262\nu ^{2}+1877\,042\,671\nu
+675\,828\,138,
\end{eqnarray*}%
in particular, when $k\in \{1,2,3\}$ we have the bounds for the smallest
positive zero of derivative of cross-product of Bessel functions. It is
possible to have more tight bounds for other values of $k\in \mathbb{N}.$
\end{proof}

\begin{proof}[\bf Proof of Theorem \ref{THbound4}]
By using the first Rayleigh sum and the implicit relation for $r^{\star
}(g_{\nu })$ we get for all $\nu >-1$ that
\begin{equation*}
\frac{1}{(r^{\star }(g_{\nu }))^{4}}=\sum_{n\geq 1}\frac{4}{\gamma _{{\nu }%
,n}^{4}-(r^{\star }(g_{\nu }))^{4}}>\sum_{n\geq 1}\frac{4}{\gamma _{{\nu }%
,n}^{4}}=\frac{1}{4(\nu +1)_{3}}.
\end{equation*}
Now, by using the Euler-Rayleigh inequalities it is possible to have more
tight bounds for the radius of starlikeness $r^{\star
}(g_{\nu }).$ For this observe that the zeros of
\begin{equation*}
g_{\nu }(z)=\sum_{n\geq 0}\frac{\left( -1\right) ^{n}}{2^{4n}n!\left( {\nu +1%
}\right) _{n}\left( {\nu +2}\right) _{2n}}z^{4n+1}
\end{equation*}%
all are real when $\nu >-1,$ according to Theorem \ref{thzeros}. Consequently, this function belongs to the
Laguerre-P\'{o}lya class $\mathcal{LP}$ of real entire functions (see subsection 3.3 for more details), which are uniform limits of real polynomials whose
all zeros are real. Now, since the Laguerre-P\'{o}lya class $\mathcal{LP}$
is closed under differentiation, it follows that $g_{\nu }^{\prime }$
belongs also to the Laguerre-P\'{o}lya class and hence all of its zeros are
real. Thus, the function $z\mapsto \widetilde{g}_{\nu }(z)=g_{\nu }^{\prime
}(\sqrt[4]{z})$ has only positive real zeros and having growth order $\frac{1%
}{4}$ it can be written as the product (see Lemma \ref{Had3})
\begin{equation*}
\widetilde{g}_{\nu }(z)=\prod\limits_{n\geq 1}\left( 1-\frac{z}{\zeta _{{\nu
},n}^{4}}\right) ,
\end{equation*}%
where $\zeta _{\nu ,n}>0$ for each $n\in \mathbb{N}.$ Now, by using the
Euler-Rayleigh sum $\sigma _{k}=\sum_{n\geq 1}\zeta _{{\nu },n}^{-4k}$ and
the infinite sum representation of $\mathbf{\Phi }_{{\nu }}$ we have
\begin{equation*}
\frac{\widetilde{g}_{\nu }^{\prime }(z)}{\widetilde{g}_{\nu }(z)}%
=\sum_{n\geq 1}\frac{-1}{\zeta _{{\nu },n}^{4}-z}=\sum_{n\geq 1}\sum_{k\geq
0}\frac{-1}{\zeta _{{\nu },n}^{4k+4}}z^{k}=\sum_{k\geq 0}-\sigma
_{k+1}z^{k},\ \ |z|<\zeta _{{\nu },1},
\end{equation*}%
\begin{equation*}
\left. \frac{\widetilde{g}_{\nu }^{\prime }(z)}{\widetilde{g}_{\nu }(z)}%
=\sum_{n\geq 0}\frac{(-1)^{n+1}(4n+5)}{2^{4n+4}n!\left( {\nu +1}\right)
_{n+1}\left( {\nu +2}\right) _{2n+2}}z^{n}\right/ \sum_{n\geq 0}\frac{\left(
-1\right) ^{n}(4n+1)}{2^{4n}n!\left( {\nu +1}\right) _{n}\left( {\nu +2}%
\right) _{2n}}z^{n}.
\end{equation*}%
From these relations it is possible to express the Euler-Rayleigh sums in
terms of $\nu $ and by using the Euler-Rayleigh inequalities $\sigma _{k}^{-%
\frac{1}{k}}<\zeta _{{\nu },1}^{4}<\frac{\sigma _{k}}{\sigma _{k+1}}$ we
obtain the following inequalities for $\nu >-1$ and $k\in \mathbb{N}$
\begin{equation*}
\sqrt[4]{\sigma _{k}^{-\frac{1}{k}}}<r^{\star }(g_{\nu })<\sqrt[4]{\frac{%
\sigma _{k}}{\sigma _{k+1}}}.
\end{equation*}%
Since
\begin{equation*}
\sigma _{1}=\frac{5}{16(\nu +1)_{3}},\ \sigma _{2}=\allowbreak \frac{16\nu
^{2}+189\nu +473}{2^{8}(\nu +1)_{3}(\nu +1)_{5}},\ \sigma _{3}=\frac{\nu _{3}}{\allowbreak 2^{11}\left( (\nu +1)_{3}%
\right) ^{2}(\nu +1)_{7}}\allowbreak ,
\end{equation*}%
and%
\begin{equation*}
\sigma _{4}=\frac{\nu_{4}}{\allowbreak \allowbreak 2^{16}\left(
(\nu +1)_{3}\right)^{2}(\nu +1)_{5}(\nu +1)_{9}},
\end{equation*}%
in particular, when $k\in \{1,2,3\}$ from the above Euler-Rayleigh
inequalities we have the following inequalities
\begin{equation*}
\sqrt[4]{\frac{16(\nu +1)_{3}}{5}}<r^{\star }(g_{\nu })<\sqrt[4]{\frac{%
80(\nu +1)_{5}}{16\nu ^{2}+189\nu +473}},
\end{equation*}%
\begin{equation*}
\sqrt[8]{\frac{2^{8}(\nu +1)_{3}(\nu +1)_{5}}{16\nu ^{2}+189\nu +473}}%
<r^{\star }(g_{\nu })<\sqrt[4]{\frac{8(\nu +1)_{3}(\nu +6)(\nu +7)\left(
16\nu ^{2}+189\nu +473\right) }{\nu_{3}}},
\end{equation*}%
\begin{equation*}
\sqrt[12]{\frac{\allowbreak 2^{11}\left((\nu +1)_{3}\right)^{2}(\nu +1)_{7}%
}{\nu _{3}}}<r^{\star }(g_{\nu })<\sqrt[4]{\frac{\allowbreak
\allowbreak 32(\nu +1)_{5}(\nu +8)(\nu +9)\nu _{3}^{\star }}{\nu_{4}}%
},
\end{equation*}%
where%
\begin{equation*}
\nu _{3}=32\nu ^{4}+824\nu ^{3}+7969\nu ^{2}+32\,944\nu +48\,267,
\end{equation*}%
\begin{eqnarray*}
\nu _{4}&=&256\nu ^{8}+13\,568\nu ^{7}+312\,736\nu
^{6}+4085\,373\nu ^{5}+32\,951\,080\nu ^{4}\\
&&+167\,370\,756\nu ^{3}+521\,177\,838\nu ^{2}+907\,600\,351\nu +675\,828\,138
\end{eqnarray*}
and it is possible to have more tight bounds for other values of $k\in
\mathbb{N}.$
\end{proof}

\begin{proof}[\bf Proof of Theorem \ref{THbound5}]
We have that%
\begin{equation*}
\frac{zh_{{\nu }}^{\prime }(z)}{h_{{\nu }}(z)}=\frac{1}{z}-\sum_{n\geq 1}%
\frac{1}{\gamma _{{\nu },n}^{4}-z},
\end{equation*}%
which vanishes at $r^{\star }(h_{\nu }).$ In view of the first Rayleigh sum
for the zeros of the cross-product of Bessel functions of the first kind we get
\begin{equation*}
\frac{1}{r^{\star }(h_{\nu })}=\sum_{n\geq 1}\frac{1}{\gamma _{{\nu }%
,n}^{4}-r^{\star }(h_{\nu })}>\sum_{n\geq 1}\frac{1}{\gamma _{{\nu },n}^{4}}=%
\frac{1}{16(\nu +1)_{3}}.
\end{equation*}
Now, we consider the infinite sum representation of $h_{\nu }$ and its
derivative%
\begin{equation*}
h_{\nu }(z)=\sum_{n\geq 0}\frac{\left( -1\right) ^{n}}{2^{4n}n!\left( {\nu +1%
}\right) _{n}\left( {\nu +2}\right) _{2n}}z^{n+1},
\end{equation*}%
\begin{equation*}
h_{\nu }^{\prime }(z)=\sum_{n\geq0}\frac{\left( -1\right) ^{n}\left(
n+1\right) }{2^{4n}n!\left( {\nu +1}\right) _{n}\left( {\nu +2}\right) _{2n}}%
z^{n}.
\end{equation*}%
The function $h_{\nu }$ has only real zeros for $\nu >-1$ and belongs to the
Laguerre-P\'{o}lya class $\mathcal{LP}$ of real entire functions, according to Theorem \ref{thzeros}. Therefore $h_{\nu }^{\prime }$ belongs also to the Laguerre-P\'{o}lya
class $\mathcal{LP}$ and has also only real zeros. Moreover, it can be seen that the function $h_{\nu
}^{\prime }$ has only positive real zeros and having growth order $\frac{1}{4%
}$ it can be written as the product (see Lemma \ref{Had3})
\begin{equation*}
h_{\nu }^{\prime }(z)=\prod\limits_{n\geq 1}\left( 1-\frac{z}{\xi _{{\nu }%
,n}^{4}}\right) ,
\end{equation*}%
where $\xi _{\nu ,n}>0$ for each $n\in \mathbb{N}$. Now, by using the
Euler-Rayleigh sum $\rho _{k}=\sum_{n\geq 1}\xi _{{\nu },n}^{-4k}$ and the
infinite sum representation of the function $\mathbf{\Phi }_{{\nu }}$ we have%
\begin{equation*}
\frac{h_{\nu }^{\prime \prime }(z)}{h_{\nu }^{\prime }(z)}=\sum_{n\geq 1}%
\frac{-1}{\xi _{{\nu },n}^{4}-z}=\sum_{n\geq 1}\sum_{k\geq 0}\frac{-1}{\xi _{%
{\nu },n}^{4k+4}}z^{k}=\sum_{k\geq 0}-\rho _{k+1}z^{k},\ \ |z|<\xi _{{\nu }%
,1},
\end{equation*}%
\begin{equation*}
\left. \frac{h_{\nu }^{\prime \prime }(z)}{h_{\nu }^{\prime }(z)}%
=\sum_{n\geq 0}\frac{(-1)^{n+1}(n+2)}{2^{4n+4}n!\left( {\nu +1}\right)
_{n+1}\left( {\nu +2}\right) _{2n+2}}z^{n}\right/ \sum_{n\geq 0}\frac{\left(
-1\right) ^{n}(n+1)}{2^{4n}n!\left( {\nu +1}\right) _{n}\left( {\nu +2}%
\right) _{2n}}z^{n}.
\end{equation*}%
We can express the Euler-Rayleigh sums in terms of $\nu $ and by using the
Euler-Rayleigh inequalities $\rho _{k}^{-\frac{1}{k}}<\xi _{{\nu },1}^{4}<%
\frac{\rho _{k}}{\rho _{k+1}}$ we get the following inequalities for $\xi _{{\nu }%
,1}^{4}$ when $\nu >-1$ and $k\in \mathbb{N}$%
\begin{equation*}
\rho _{k}^{-\frac{1}{k}}<r^{\star }(h_{\nu })<\frac{\rho _{k}}{\rho _{k+1}}.
\end{equation*}%
Since
\begin{equation*}
\rho _{1}=\frac{1}{8(\nu +1)_{3}},\text{ \ }\rho _{2}=\frac{\nu ^{2}+24\nu
+71}{2^{8}(\nu +1)_{3}(\nu +1)_{5}},\ \rho _{3}=\frac{\nu_{5}}{\allowbreak 2^{12}\left( (\nu +1)_{3}%
\right) ^{2}(\nu +1)_{7}}
\end{equation*}%
and%
\begin{equation*}
\rho _{4}=\frac{\nu_{6}}{2^{16}\left((\nu +1)_{3}\right)^{2}(\nu
+1)_{5}(\nu +1)_{9}},
\end{equation*}%
where%
\begin{equation*}
\nu _{5}=\nu ^{4}+37\nu ^{3}+593\nu ^{2}+3275\nu +5598,
\end{equation*}%
\begin{eqnarray*}
\nu _{6}&=&\nu ^{8}+68\nu ^{7}+2062\nu ^{6}+36\,519\nu
^{5}+388\,627\nu ^{4}\\
&&+2477\,862\nu ^{3}+9218\,508\nu ^{2}+18\,391\,471\nu +15\,167\,442,
\end{eqnarray*}%
in particular, when $k\in \{1,2,3\}$ we have the inequalities of this
theorem.
\end{proof}

\begin{proof}[\bf Proof of Theorem \ref{THSbound1}]
By using the first Rayleigh sum and the implicit relation for $r^{\star
}(u_{\nu }),$ we obtain for all $\nu >{0}$ that
\begin{equation*}
\frac{1}{(r^{\star }(u_{\nu }))^{4}}=\frac{1}{2\nu }\sum_{n\geq 1}\frac{4}{%
j_{{\nu },n}^{4}-(r^{\star }(u_{\nu }))^{4}}>\frac{1}{2\nu }\sum_{n\geq 1}%
\frac{4}{j_{{\nu },n}^{4}}=\frac{1}{8\nu (\nu +1)^{2}(\nu +2)}.
\end{equation*}%
Now, we consider the infinite sum and product representation of the
function
\begin{equation*}
z\mapsto \mathbf{\chi }_{\nu }(z)=\frac{1}{2{\nu }}2^{2{\nu }}\Gamma
^{2}\left( {\nu }+1\right) z^{\frac{-{\nu }}{2}+\frac{1}{4}}\mathbf{\Pi }_{{%
\nu }}^{\prime }(\sqrt[4]{z}).
\end{equation*}%
We have%
\begin{equation*}
\mathbf{\chi }_{\nu }(z)=\sum_{n\geq 0}\frac{\left( -1\right) ^{n}\left( {%
\nu }+2n\right) }{2^{4n}n!{\nu }\left( {\nu +1}\right) _{n}\left( {\nu +1}%
\right) _{2n}}z^{n}
\end{equation*}%
and%
\begin{equation*}
\mathbf{\chi }_{\nu }(z)=\prod\limits_{n\geq 1}\left( 1-\frac{z}{t_{{\nu }%
,n}^{4}}\right),
\end{equation*}
according to Lemma \ref{Had2}. By using the Euler-Rayleigh inequalities $\eta_{k}^{-\frac{1}{k}}<t_{{\nu }%
,1}^{4}<\frac{\eta_{k}}{\eta_{k+1}},$ where $\eta_k=\sum_{n\geq 1}t_{\nu,n}^{-4k},$ we obtain the inequalities for $\nu >0
$ and $k\in \mathbb{N}$
\begin{equation*}
\sqrt[4]{\eta_{k}^{-\frac{1}{k}}}<r^{\star }(u_{\nu })<\sqrt[4]{\frac{\eta
_{k}}{\eta_{k+1}}}.
\end{equation*}%
Since%
\begin{equation*}
\eta_{1}=\frac{1}{16{\nu }(\nu +1)^{2}},\text{ \ }\eta_{2}=\frac{5\nu
^{2}+15\nu +12}{2^{8}\left( \nu \right) _{3}(\nu )_{4}\left( \nu +1\right)
^{2}}, \ \eta_{3}=\frac{\nu_{7}}{2^{11}\nu (\nu )_{4}\left( \nu \right)
_{6}\left( \nu +1\right) ^{4}}
\end{equation*}%
and%
\begin{equation*}
\eta_{4}=\frac{\nu_{8}}{2^{16}\left(\left(\nu\right)_{3}\right)
^{2}\left( \nu \right) _{5}\left( \nu \right) _{8}\left( \nu +1\right) ^{4}},
\end{equation*}%
where%
\begin{equation*}
\nu_{7}=21\nu ^{4}+173\nu ^{3}+533\nu ^{2}+717\nu +360
\end{equation*}%
and%
\begin{eqnarray*}
\nu_{8} &=&429\nu ^{8}+8688\nu ^{7}+76\,280\nu ^{6}+377\,494\nu
^{5}+1148\,139\nu ^{4}\\
&&+2194\,202\nu ^{3}+2574\,064\nu ^{2}+1698\,048\nu +483\,840,
\end{eqnarray*}%
in particular, when $k\in \{1,2,3\}$ we have the bounds for the smallest
positive zero of derivative of product of Bessel functions. It is possible
to have more tight bounds for other values of $k\in \mathbb{N}.$
\end{proof}

\begin{proof}[\bf Proof of Theorem \ref{THSbound2}]
By using the first Rayleigh sum and the implicit relation for $r^{\star
}(v_{\nu }),$ we get for all $\nu >-1$ that
\begin{equation*}
\frac{1}{(r^{\star }(v_{\nu }))^{4}}=\sum_{n\geq 1}\frac{4}{j_{{\nu }%
,n}^{4}-(r^{\star }(v_{\nu }))^{4}}>\sum_{n\geq 1}\frac{4}{j_{{\nu },n}^{4}}=%
\frac{1}{4(\nu +1)^{2}(\nu +2)}.
\end{equation*}
Now, by using the Euler-Rayleigh inequalities it is possible to have more
tight bounds for the radius of starlikeness $r^{\star
}(v_{\nu }).$ For this we recall that according to Theorem \ref{thzeros} the zeros of
\begin{equation*}
v_{\nu }(z)=\sum_{n\geq 0}\frac{\left( -1\right) ^{n}}{2^{4n}n!\left( {\nu +1%
}\right) _{n}\left( {\nu +1}\right) _{2n}}z^{4n+1}
\end{equation*}%
all are real when $\nu >-1.$ Consequently, this function belongs to the
Laguerre-P\'{o}lya class $\mathcal{LP}$ of real entire functions and since the
Laguerre-P\'{o}lya class $\mathcal{LP}$ is closed under differentiation, it
follows that $v_{\nu }^{\prime }$ belongs also to the Laguerre-P\'{o}lya
class. Consequently all of its zeros of $v_{\nu}'$ are real when $\nu>-1$. Thus, the function $z\mapsto
\widetilde{v}_{\nu }(z)=v_{\nu }^{\prime }(\sqrt[4]{z})$ has only positive
real zeros and having growth order $\frac{1}{4}$ it can be written as the
product (see Lemma \ref{Had3})
\begin{equation*}
\widetilde{v}_{\nu }(z)=\prod\limits_{n\geq 1}\left( 1-\frac{z}{\vartheta _{{%
\nu },n}^{4}}\right) ,
\end{equation*}%
where $\vartheta _{\nu ,n}>0$ for each $n\in \mathbb{N}.$ Now, by using the
Euler-Rayleigh sum $\varrho_{k}=\sum_{n\geq 1}\vartheta _{{\nu },n}^{-4k}$ and the
infinite sum representation of $\mathbf{\Pi }_{{\nu }}$ we have
\begin{equation*}
\frac{\widetilde{v}_{\nu }^{\prime }(z)}{\widetilde{v}_{\nu }(z)}%
=\sum_{n\geq 1}\frac{-1}{\vartheta _{{\nu },n}^{4}-z}=\sum_{n\geq
1}\sum_{k\geq 0}\frac{-1}{\vartheta _{{\nu },n}^{4k+4}}z^{k}=\sum_{k\geq
0}-\varrho_{k+1}z^{k},\ \ |z|<\vartheta _{{\nu },1},
\end{equation*}%
\begin{equation*}
\left. \frac{\widetilde{v}_{\nu }^{\prime }(z)}{\widetilde{v}_{\nu }(z)}%
=\sum_{n\geq 0}\frac{(-1)^{n+1}(4n+5)}{2^{4n+4}n!\left( {\nu +1}\right)
_{n+1}\left( {\nu +1}\right) _{2n+2}}z^{n}\right/ \sum_{n\geq 0}\frac{\left(
-1\right) ^{n}(4n+1)}{2^{4n}n!\left( {\nu +1}\right) _{n}\left( {\nu +1}%
\right) _{2n}}z^{n}.
\end{equation*}%
From these relations it is possible to express the Euler-Rayleigh sums in
terms of $\nu $ and by using the Euler-Rayleigh inequalities $\varrho_{k}^{-\frac{1%
}{k}}<\vartheta _{{\nu },1}^{4}<\frac{\varrho_{k}}{\varrho_{k+1}}$ we obtain the
inequalities for $\nu >-1$ and $k\in \mathbb{N}$
\begin{equation*}
\sqrt[4]{\varrho_{k}^{-\frac{1}{k}}}<r^{\star }(v_{\nu })<\sqrt[4]{\frac{\varrho_{k}}{%
\varrho_{k+1}}}.
\end{equation*}%
Since
\begin{equation*}
\varrho_{1}=\frac{5}{16(\nu +1)^{2}(\nu +2)},\ \varrho_{2}=\frac{16\nu ^{2}+157\nu +291}{%
2^{8}(\nu +1)_{4}\left( \nu +1\right) ^{3}\left( \nu +2\right) }\allowbreak
,\ \varrho_{3}=\frac{\nu _{9}}{\allowbreak 2^{11}(\nu +1)_{3}(\nu
+1)_{6}\left( \nu +1\right) ^{4}\left( \nu +2\right) }\allowbreak ,
\end{equation*}%
and%
\begin{equation*}
\varrho_{4}=\frac{\nu _{10}}{\allowbreak \allowbreak \allowbreak
\allowbreak 2^{16}(\nu +1)_{4}(\nu +1)_{8}\left( \nu +1\right) ^{6}\left(
\nu +2\right) ^{2}},
\end{equation*}%
where%
\begin{equation*}
\nu_{9}=32\nu ^{5}+792\nu ^{4}+7753\nu ^{3}+35\,977\nu
^{2}+78\,453\nu +64\,469
\end{equation*}%
and%
\begin{eqnarray*}
\nu_{10} &=&256\nu ^{8}+11\,520\nu ^{7}+224\,672\nu
^{6}+2469\,757\nu ^{5}+16\,606\,040\nu ^{4}\\
&&+69\,429\,816\nu ^{3}+175\,324\,950\nu ^{2}+243\,560\,267\nu +142\,215\,442,
\end{eqnarray*}%
in particular, when $k\in \{1,2,3\}$ from the above Euler-Rayleigh
inequalities we have the inequalities in this theorem.
\end{proof}

\begin{proof}[\bf Proof of Theorem \ref{THSbound3}]
We have that%
\begin{equation*}
\frac{zw_{{\nu }}^{\prime }(z)}{w_{{\nu }}(z)}=\frac{1}{z}-\sum_{n\geq 1}%
\frac{1}{j_{{\nu },n}^{4}-z},
\end{equation*}%
which vanishes at $r^{\star }(w_{\nu }).$ In view of the first Rayleigh sum
for the zeros of the Bessel functions of the first kind we get
\begin{equation*}
\frac{1}{r^{\star }(w_{\nu })}=\sum_{n\geq 1}\frac{1}{j_{{\nu }%
,n}^{4}-r^{\star }(w_{\nu })}>\sum_{n\geq 1}\frac{1}{j_{{\nu },n}^{4}}=\frac{%
1}{16(\nu +1)^{2}(\nu +2)}.
\end{equation*}
Now, we consider the infinite sum representations of $w_{\nu }$ and its
derivative%
\begin{equation*}
w_{\nu }(z)=\sum_{n\geq 0}\frac{\left( -1\right) ^{n}}{2^{4n}n!\left( {\nu +1%
}\right) _{n}\left( {\nu +1}\right) _{2n}}z^{n+1},
\end{equation*}%
\begin{equation*}
w_{\nu }^{\prime }(z)=\sum_{n\geq 0}\frac{\left( -1\right) ^{n}\left(
n+1\right) }{2^{4n}n!\left( {\nu +1}\right) _{n}\left( {\nu +1}\right) _{2n}}%
z^{n}.
\end{equation*}%
According to Theorem \ref{thzeros} the function $w_{\nu }$ has only real zeros for $\nu >-1$ and belongs to the
Laguerre-P\'{o}lya class $\mathcal{LP}$ of real entire functions. Therefore $%
w_{\nu }^{\prime }$ belongs also to the Laguerre-P\'{o}lya class $\mathcal{LP%
}$ and consequently has also only real zeros. Moreover, the function $w_{\nu }^{\prime }$ has only
positive real zeros and having growth order $\frac{1}{4}$ it can be written
as the product (see Lemma \ref{Had3})
\begin{equation*}
w_{\nu }^{\prime }(z)=\prod\limits_{n\geq 1}\left( 1-\frac{z}{\omega _{{\nu }%
,n}^{4}}\right) ,
\end{equation*}%
where $\omega _{\nu ,n}>0$ for each $n\in \mathbb{N}$. Now, by using the
Euler-Rayleigh sum $q_{k}=\sum_{n\geq 1}\omega _{{\nu },n}^{-4k}$ and the
infinite sum representation of the function $\mathbf{\Pi }_{{\nu }}$ we have%
\begin{equation*}
\frac{w_{\nu }^{\prime \prime }(z)}{w_{\nu }^{\prime }(z)}=\sum_{n\geq 1}%
\frac{-1}{\omega _{{\nu },n}^{4}-z}=\sum_{n\geq 1}\sum_{k\geq 0}\frac{-1}{%
\omega _{{\nu },n}^{4k+4}}z^{k}=\sum_{k\geq 0}-q_{k+1}z^{k},\ \ |z|<\omega _{%
{\nu },1},
\end{equation*}%
\begin{equation*}
\left. \frac{w_{\nu }^{\prime \prime }(z)}{w_{\nu }^{\prime }(z)}%
=\sum_{n\geq 0}\frac{(-1)^{n+1}(n+2)}{2^{4n+4}n!\left( {\nu +1}\right)
_{n+1}\left( {\nu +1}\right) _{2n+2}}z^{n}\right/ \sum_{n\geq 0}\frac{\left(
-1\right) ^{n}(n+1)}{2^{4n}n!\left( {\nu +1}\right) _{n}\left( {\nu +1}%
\right) _{2n}}z^{n}.
\end{equation*}%
We can express the Euler-Rayleigh sums in terms of $\nu $ and by using the
Euler-Rayleigh inequalities $q_{k}^{-\frac{1}{k}}<\omega _{{\nu },1}^{4}<%
\frac{q_{k}}{q_{k+1}}$ we get the inequalities for $\omega _{{\nu },1}^{4}$
when $\nu >-1$ and $k\in \mathbb{N}$%
\begin{equation*}
q_{k}^{-\frac{1}{k}}<r^{\star }(w_{\nu })<\frac{q_{k}}{q_{k+1}}.
\end{equation*}%
Since
\begin{equation*}
q_{1}=\frac{1}{8(\nu +1)^{2}(\nu +2)},\text{ \ }q_{2}=\frac{\nu ^{2}+22\nu
+45}{2^{8}(\nu +1)_{4}\left( \nu +1\right)^{3}\left(\nu+2\right) },\ q_{3}=\frac{\nu_{11}}{2^{12}(\nu +1)_{3}(\nu +1)_{6}\left( \nu
+1\right) ^{4}\left( \nu +2\right)}
\end{equation*}%
and%
\begin{equation*}
q_{4}=\frac{\nu_{12}}{\allowbreak \allowbreak 2^{16}(\nu
+1)_{4}(\nu +1)_{8}\left( \nu +1\right) ^{6}\left( \nu +2\right) ^{2}},
\end{equation*}%
where%
\begin{equation*}
\nu _{11}=\nu ^{5}+36\nu ^{4}+584\nu ^{3}+3554\nu ^{2}+8919\nu +7834
\end{equation*}%
and%
\begin{eqnarray*}
\nu _{12} &=&\nu ^{8}+60\nu ^{7}+1610\nu ^{6}+25\,303\nu
^{5}+229\,535\nu ^{4}\\
&&+1199\,202\nu ^{3}+3542\,412\nu ^{2}+5461\,979\nu +3396\,826,
\end{eqnarray*}%
in particular, when $k\in \{1,2,3\}$ we have the inequalities of this
theorem.
\end{proof}

\begin{proof}[\bf Proof of Theorem \ref{T3}]
\textbf{a)} Observe that
\begin{equation*}
1+\frac{zf_{{\nu }}^{\prime \prime }(z)}{f_{{\nu }}^{\prime }(z)}=1+\frac{z%
\mathbf{\Phi }_{{\nu }}^{\prime \prime }(z)}{\mathbf{\Phi }_{{\nu }}^{\prime
}(z)}+\left( \frac{1}{2{\nu +1}}-1\right) \frac{z\mathbf{\Phi }_{{\nu }%
}^{\prime }(z)}{\mathbf{\Phi }_{{\nu }}(z)},\text{ \ }{\nu \neq -}\frac{1}{2}%
\text{\ .\ }
\end{equation*}%
By means of (\ref{fi1}) and (\ref{Ha1}) we have%
\begin{equation}
\frac{z\mathbf{\Phi }_{{\nu }}^{\prime }(z)}{\mathbf{\Phi }_{{\nu }}(z)}=2{%
\nu +1}-\sum_{n\geq1}\frac{4z^{4}}{\gamma _{{\nu },n}^{4}-z^{4}},%
\text{ \ \ \ }1+\frac{z\mathbf{\Phi }_{{\nu }}^{\prime \prime }(z)}{\mathbf{%
\Phi }_{{\nu }}^{\prime }(z)}=2{\nu +1}-\sum_{n\geq1}\frac{4z^{4}}{%
\gamma _{{\nu },n}^{\prime 4}-z^{4}},  \label{2.1}
\end{equation}%
and it follows that%
\begin{equation*}
1+\frac{zf_{{\nu }}^{\prime \prime }(z)}{f_{{\nu }}^{\prime }(z)}=1-\left(
\frac{1}{2{\nu +1}}-1\right) \sum_{n\geq1}\frac{4z^{4}}{\gamma _{{\nu
},n}^{4}-z^{4}}-\sum_{n\geq1}\frac{4z^{4}}{\gamma _{{\nu },n}^{\prime
4}-z^{4}},\text{ \ }{\nu \neq -}\frac{1}{2}.
\end{equation*}%
Now, suppose that ${\nu }\in ({-}\frac{1}{2},0].$ By using (\ref{s7}), we
obtain for all $z\in \mathbb{D}_{\gamma _{{\nu },1}^{\prime }}$ the
inequality%
\begin{equation*}
\real\left(1+\frac{zf_{{\nu }}^{\prime \prime }(z)}{f_{{\nu }}^{\prime
}(z)}\right) \geqslant 1-\left( \frac{1}{2{\nu +1}}-1\right)
\sum_{n\geq1}\frac{4r^{4}}{\gamma _{{\nu },n}^{4}-r^{4}}%
-\sum_{n\geq1}\frac{4r^{4}}{\gamma _{{\nu },n}^{\prime 4}-r^{4}},
\end{equation*}%
where ${\left\vert z\right\vert =r.}$ Moreover, observe that if we use the
inequality \cite[Lemma 2.1]{BAS}%
\begin{equation}
\lambda \real\left( \frac{z}{a-z}\right) -\real\left( \frac{z}{b-z}%
\right) \geqslant \lambda \frac{{\left\vert z\right\vert }}{a-{\left\vert
z\right\vert }}-\frac{{\left\vert z\right\vert }}{b-{\left\vert z\right\vert
}},  \label{2.2}
\end{equation}%
where $a>b>0,$ $\lambda \in \left[ 0,1\right] $ and $z\in \mathbb{C}$ such
that ${\left\vert z\right\vert <b}$, then we get that the above inequality
is also valid when ${\nu }>0$. Here we used that the zeros $\gamma _{{\nu }%
,n}$ and $\gamma _{{\nu },n}^{\prime }$ interlace, according to Lemma \ref{Had2}. Thus, for ${r}\in \left(
0,\gamma _{{\nu },1}^{\prime }\right) $ we have
\begin{equation*}
\inf_{z\in \mathbb{D}_{r}}\left\{ \real\left( 1+\frac{zf_{{\nu }%
}^{\prime \prime }(z)}{f_{{\nu }}^{\prime }(z)}\right) \right\} =1+\frac{rf_{%
{\nu }}^{\prime \prime }(r)}{f_{{\nu }}^{\prime }(r)}.
\end{equation*}%
On the other hand, the function $F_{{\nu }}:\left( 0,\gamma _{{\nu }%
,1}^{\prime }\right) \longrightarrow \mathbb{R}$, defined by%
\begin{equation*}
F_{{\nu }}(r)=1+\frac{rf_{{\nu }}^{\prime \prime }(r)}{f_{{\nu }}^{\prime
}(r)},
\end{equation*}%
is strictly decreasing for all ${\nu }>{-}\frac{1}{2}.$ Namely, we have
\begin{eqnarray*}
F_{{\nu }}^{\prime }(r) &=&-\left( \frac{1}{2{\nu +1}}-1\right)
\sum_{n\geq1}\frac{16r^{3}\gamma _{{\nu },n}^{4}}{\left( \gamma _{{%
\nu },n}^{4}-r^{4}\right) ^{2}}-\sum_{n\geq1}\frac{16r^{3}\gamma _{{%
\nu },n}^{\prime 4}}{\left( \gamma _{{\nu },n}^{\prime 4}-r^{4}\right) ^{2}}
\\
&<&\sum_{n\geq1}\frac{16r^{3}\gamma _{{\nu },n}^{4}}{\left( \gamma _{{%
\nu },n}^{4}-r^{4}\right) ^{2}}-\sum_{n\geq1}\frac{16r^{3}\gamma _{{%
\nu },n}^{\prime 4}}{\left( \gamma _{{\nu },n}^{\prime 4}-r^{4}\right) ^{2}}%
<0
\end{eqnarray*}%
for ${\nu }>{-}\frac{1}{2}$ and ${r}\in \left( 0,\gamma _{{\nu },1}^{\prime
}\right) .$ Here we used again that the zeros $\gamma _{{\nu },n}$ and $%
\gamma _{{\nu },n}^{\prime }$ interlace, and for all $n\in \mathbb{N}$, ${%
\nu }>{-}\frac{1}{2}$ and $r<\sqrt{\gamma _{{\nu },n}\gamma _{{\nu }%
,n}^{\prime }}$ we have that%
\begin{equation*}
\gamma _{{\nu },n}^{4}\left( \gamma _{{\nu },n}^{\prime 4}-r^{4}\right)
^{2}<\gamma _{{\nu },n}^{\prime 4}\left( \gamma _{{\nu },n}^{4}-r^{4}\right)
^{2}.
\end{equation*}%
Now, since $\lim_{r\searrow 0}F_{{\nu }}(r)=1>\alpha $ and $\lim_{r\nearrow
\gamma _{{\nu },1}^{\prime }}F_{{\nu }}(r)=-\infty ,$ in view of the minimum
principle for harmonic functions it follows that for ${\nu }>{-}\frac{1}{2}$
and $z\in \mathbb{D}_{r_{1}}$ we have%
\begin{equation}
\real\left( 1+\frac{zf_{{\nu }}^{\prime \prime }(z)}{f_{{\nu }}^{\prime
}(z)}\right) >\alpha  \label{f1}
\end{equation}%
if and only if $r_{1}$ is the unique root of
\begin{equation*}
1+\frac{rf_{{\nu }}^{\prime \prime }(r)}{f_{{\nu }}^{\prime }(r)}=\alpha
\end{equation*}%
situated in $\left( 0,\gamma _{{\nu },1}^{\prime }\right) .$ This completes
the proof of part \textbf{a} of our theorem when ${\nu }>{-}\frac{1}{2}$.

\textbf{b) }In view of Lemma \ref{Had3} we have that
\begin{equation*}
1+\frac{zg_{{\nu }}^{\prime \prime }(z)}{g_{{\nu }}^{\prime }(z)}%
=1-\sum_{n\geq1}\frac{4z^{4}}{\zeta _{{\nu },n}^{4}-z^{4}}\text{.}
\end{equation*}%
By using the inequality (\ref{s7}), for all $z\in \mathbb{D}_{\zeta _{{\nu }%
,1}}$ we obtain the inequality%
\begin{equation*}
\real\left( 1+\frac{zg_{{\nu }}^{\prime \prime }(z)}{g_{{\nu }}^{\prime
}(z)}\right) \geqslant 1-\sum_{n\geq1}\frac{4r^{4}}{\zeta _{{\nu }%
,n}^{4}-r^{4}},
\end{equation*}%
where ${\left\vert z\right\vert =r.}$ Thus, for ${r}\in \left( 0,\zeta _{{%
\nu },1}\right) $ we get%
\begin{equation*}
\inf_{z\in \mathbb{D}_{r}}\left\{ \real\left( 1+\frac{zg_{{\nu }%
}^{\prime \prime }(z)}{g_{{\nu }}^{\prime }(z)}\right) \right\} =1+\frac{rg_{%
{\nu }}^{\prime \prime }(r)}{g_{{\nu }}^{\prime }(r)}.
\end{equation*}%
The function $G_{{\nu }}:\left( 0,\zeta _{{\nu },1}\right) \longrightarrow
\mathbb{R}$, defined by%
\begin{equation*}
G_{{\nu }}(r)=1+\frac{rg_{{\nu }}^{\prime \prime }(r)}{g_{{\nu }}^{\prime
}(r)},
\end{equation*}%
is strictly decreasing and $\lim_{r\searrow 0}G_{{\nu }}(r)=1>\alpha $, $%
\lim_{r\nearrow \zeta _{{\nu },1}}G_{{\nu }}(r)=-\infty $. Consequently, in
view of the minimum principle for harmonic functions for $z\in \mathbb{D}%
_{r_{2}}$ we have that
\begin{equation*}
\real\left( 1+\frac{zg_{{\nu }}^{\prime \prime }(z)}{g_{{\nu }}^{\prime
}(z)}\right) >\alpha
\end{equation*}%
if and only if $r_{2}$ is the unique root of
\begin{equation*}
1+\frac{rg_{{\nu }}^{\prime \prime }(r)}{g_{{\nu }}^{\prime }(r)}=\alpha
\end{equation*}%
situated in $\left( 0,\zeta _{{\nu },1}\right) .$ Finally, the inequality $%
\zeta _{{\nu },1}<\gamma _{{\nu },1}$ follows from Lemma \ref{LP}.

\textbf{c)} In view of Lemma \ref{Had3} we have that
\begin{equation*}
1+\frac{zh_{{\nu }}^{\prime \prime }(z)}{h_{{\nu }}^{\prime }(z)}%
=1-\sum_{n\geq1}\frac{z}{\xi _{{\nu },n}^{4}-z}\text{.}
\end{equation*}%
By using the inequality (\ref{s7}), for all $z\in \mathbb{D}_{\xi _{{\nu }%
,1}}$ we obtain the inequality%
\begin{equation*}
\real\left( 1+\frac{zh_{{\nu }}^{\prime \prime }(z)}{h_{{\nu }}^{\prime
}(z)}\right) \geqslant 1-\sum_{n\geq1}\frac{r}{\xi _{{\nu },n}^{4}-r},
\end{equation*}%
where ${\left\vert z\right\vert =r.}$ Thus, for ${r}\in \left( 0,\xi _{{\nu }%
,1}\right) $ we get%
\begin{equation*}
\inf_{z\in \mathbb{D}_{r}}\left\{ \real\left( 1+\frac{zh_{{\nu }%
}^{\prime \prime }(z)}{h_{{\nu }}^{\prime }(z)}\right) \right\} =1+\frac{%
r\,h_{{\nu }}^{\prime \prime }(r)}{h_{{\nu }}^{\prime }(r)}.
\end{equation*}%
The function $H_{{\nu }}:\left( 0,\xi _{{\nu },1}\right) \longrightarrow
\mathbb{R}$, defined by%
\begin{equation*}
H_{{\nu }}(r)=1+\frac{r\,h_{{\nu }}^{\prime \prime }(r)}{h_{{\nu }}^{\prime
}(r)},
\end{equation*}%
is strictly decreasing and $\lim_{r\searrow 0}H_{{\nu }}(r)=1>\alpha $, $%
\lim_{r\nearrow \xi _{{\nu },1}}H_{{\nu }}(r)=-\infty $. Consequently, in
view of the minimum principle for harmonic functions for $z\in \mathbb{D}%
_{r_{3}}$ we have that
\begin{equation*}
\real\left( 1+\frac{zh_{{\nu }}^{\prime \prime }(z)}{h_{{\nu }}^{\prime
}(z)}\right) >\alpha
\end{equation*}%
if and only if $r_{3}$ is the unique root of
\begin{equation*}
1+\frac{r\,h_{{\nu }}^{\prime \prime }(r)}{h_{{\nu }}^{\prime }(r)}=\alpha
\end{equation*}%
situated in $\left( 0,\xi _{{\nu },1}\right) .$ Finally, the inequality $\xi
_{{\nu },1}<\gamma _{{\nu },1}$ follows from Lemma \ref{LP}.
\end{proof}

\begin{proof}[\bf Proof of Theorem \ref{T4}]
\textbf{a)} By means of (\ref{pi1}) and (\ref{Ha2}) we have%
\begin{align*}
1+\frac{zu_{\nu }^{\prime \prime }(z)}{u_{\nu }^{\prime }(z)}& =1+\frac{z%
\mathbf{\ \mathbf{\Pi }}_{\nu }^{\prime \prime }(z)}{\mathbf{\mathbf{\Pi }}%
_{\nu }^{\prime }(z)}+\left( \frac{1}{2\nu }-1\right) \frac{z\mathbf{\mathbf{%
\Pi }}_{\nu }^{\prime }(z)}{\mathbf{\mathbf{\Pi }}_{\nu }(z)},\text{ }\nu
\neq 0 \\
& =1-\left( \frac{1}{2\nu }-1\right) \sum_{n\geq1}\frac{4z^{4}}{j_{{%
\nu },n}^{4}-z^{4}}-\sum_{n\geq1}\frac{4z^{4}}{t_{{\nu },n}^{4}-z^{4}}%
.
\end{align*}%
Now, suppose that $\nu \in (0,\frac{1}{2}].$ By using the inequality (\ref%
{s7}), for all $z\in \mathbb{D}_{t_{\nu ,1}}$ we obtain the inequality
\begin{equation*}
\real\left( 1+\frac{zu_{\nu }^{\prime \prime }(z)}{u_{\nu }^{\prime }(z)}%
\right) \geqslant 1-\left( \frac{1}{2\nu }-1\right) \sum_{n\geq1}%
\frac{4r^{4}}{j_{{\nu },n}^{4}-r^{4}}-\sum_{n\geq1}\frac{4r^{4}}{t_{{%
\nu },n}^{4}-r^{4}},
\end{equation*}%
where ${\left\vert z\right\vert =r.}$ Moreover, observe that if we use the
inequality (\ref{2.2}) then we get that the above inequality is also valid
when $\nu >\frac{1}{2}$. Here we used that the zeros $j_{\nu ,n}$ and $%
t_{\nu ,n}$ interlace, according to Lemma \ref{Had2}. The above inequality implies for ${r}\in \left(
0,t_{\nu ,1}\right) $%
\begin{equation*}
\inf_{z\in \mathbb{D}_{r}}\left\{ \real\left( 1+\frac{zu_{\nu }^{\prime
\prime }(z)}{u_{\nu }^{\prime }(z)}\right) \right\} =1+\frac{ru_{\nu
}^{\prime \prime }(r)}{u_{\nu }^{\prime }(r)}.
\end{equation*}%
On the other hand, the function $U_{\nu }:\left( 0,t_{\nu ,1}\right)
\longrightarrow \mathbb{R}$, defined by%
\begin{equation*}
U_{\nu }(r)=1+\frac{ru_{\nu }^{\prime \prime }(r)}{u_{\nu }^{\prime }(r)},
\end{equation*}%
is strictly decreasing since%
\begin{eqnarray*}
U_{\nu }^{\prime }(r) &=&-\left( \frac{1}{2\nu }-1\right) \sum_{n\geq1}
\frac{16r^{3}j_{{\nu },n}^{4}}{\left( j_{{\nu },n}^{4}-r^{4}\right) ^{2}}%
-\sum_{n\geq1}\frac{16r^{3}t_{{\nu },n}^{4}}{\left( t_{{\nu }%
,n}^{4}-r^{4}\right) ^{2}} \\
&<&\sum_{n\geq1}\frac{16r^{3}j_{{\nu },n}^{4}}{\left( j_{{\nu }%
,n}^{4}-r^{4}\right) ^{2}}-\sum_{n\geq1}\frac{16r^{3}t_{{\nu },n}^{4}%
}{\left( t_{{\nu },n}^{4}-r^{4}\right) ^{2}}<0
\end{eqnarray*}%
for $\nu >0$ and ${r}\in \left( 0,t_{\nu ,1}\right) .$ Here we used again
that the zeros $j_{\nu ,n}$ and $t_{\nu ,n}$ interlace for all $n\in \mathbb{%
N}$, $\nu >0$ and $r<\sqrt{j_{\nu ,n}t_{\nu ,n}}$ we have that
\begin{equation*}
j_{{\nu },n}^{4}\left( t_{{\nu },n}^{4}-r^{4}\right) ^{2}<t_{{\nu }%
,n}^{4}\left( j_{{\nu },n}^{4}-r^{4}\right) ^{2}.
\end{equation*}%
Since $\lim_{r\searrow 0}U_{\nu }(r)=1>\alpha $ and $\lim_{r\nearrow t_{\nu
,1}}U_{\nu }(r)=-\infty ,$ in view of the minimum principle for harmonic
functions it follows that for $z\in \mathbb{D}_{r_{4}}$ we have%
\begin{equation*}
\real\left( 1+\frac{zu_{\nu }^{\prime \prime }(z)}{u_{\nu }^{\prime }(z)}%
\right) >\alpha
\end{equation*}%
if and only if $r_{4}$ is the unique root of
\begin{equation*}
1+\frac{ru_{\nu }^{\prime \prime }(r)}{u_{\nu }^{\prime }(r)}=\alpha
\end{equation*}%
situated in $\left( 0,t_{\nu ,1}\right) .$

\textbf{b)} In view of Lemma \ref{Had3} we have that
\begin{equation*}
1+\frac{zv_{{\nu }}^{\prime \prime }(z)}{v_{{\nu }}^{\prime }(z)}%
=1-\sum_{n\geq1}\frac{4z^{4}}{\vartheta _{{\nu },n}^{4}-z^{4}}\text{.}
\end{equation*}%
By using the inequality (\ref{s7}), for all $z\in \mathbb{D}_{\vartheta _{{%
\nu },1}}$ we obtain the inequality%
\begin{equation*}
\real\left( 1+\frac{zv_{{\nu }}^{\prime \prime }(z)}{v_{{\nu }}^{\prime
}(z)}\right) \geqslant 1-\sum_{n\geq1}\frac{4r^{4}}{\vartheta _{{\nu }%
,n}^{4}-r^{4}},
\end{equation*}%
where ${\left\vert z\right\vert =r.}$ Thus, for ${r}\in \left( 0,\vartheta _{%
{\nu },1}\right) $ we get%
\begin{equation*}
\inf_{z\in \mathbb{D}_{r}}\left\{ \real\left( 1+\frac{zv_{{\nu }%
}^{\prime \prime }(z)}{v_{{\nu }}^{\prime }(z)}\right) \right\} =1+\frac{rv_{%
{\nu }}^{\prime \prime }(r)}{v_{{\nu }}^{\prime }(r)}.
\end{equation*}
On the other hand, the function $V_{\nu }:\left( 0,\vartheta _{{\nu }%
,1}\right) \longrightarrow \mathbb{R}$, defined by%
\begin{equation*}
V_{\nu }(r)=1+\frac{rv_{\nu }^{\prime \prime }(r)}{v_{\nu }^{\prime }(r)},
\end{equation*}%
is strictly decreasing \ and $\lim_{r\searrow 0}V_{\nu }(r)=1>\alpha $ and $%
\lim_{r\nearrow \vartheta _{{\nu },1}}V_{\nu }(r)=-\infty $. Consequently,
in view of the minimum principle for harmonic functions for $z\in \mathbb{D}%
_{r_{5}}$ we have that
\begin{equation*}
\real\left( 1+\frac{zv_{\nu }^{\prime \prime }(z)}{v_{\nu }^{\prime }(z)}%
\right) >\alpha
\end{equation*}%
if and only if $r_{5}$ is the unique root of
\begin{equation*}
1+\frac{rv_{\nu }^{\prime \prime }(r)}{v_{\nu }^{\prime }(r)}=\alpha
\end{equation*}%
situated in $\left( 0,\vartheta _{{\nu },1}\right) .$ Finally, the
inequality $\vartheta _{{\nu },1}<j_{{\nu },1}$ follows from Lemma \ref{LP}.

\textbf{c)} Similarly, from Lemma \ref{Had3} by using
\begin{equation*}
1+\frac{zw_{{\nu }}^{\prime \prime }(z)}{w_{{\nu }}^{\prime }(z)}%
=1-\sum_{n\geq1}\frac{z}{\omega _{{\nu },n}^{4}-z}.
\end{equation*}%
and the inequality (\ref{s7}), we get for all $z\in \mathbb{D}_{\omega _{{%
\nu },1}}$
\begin{equation*}
\real\left( 1+\frac{zw_{\nu }^{\prime \prime }(z)}{w_{\nu }^{\prime }(z)}%
\right) \geqslant 1-\sum_{n\geq1}\frac{r}{\omega _{{\nu },n}^{4}-r},
\end{equation*}%
where ${\left\vert z\right\vert =r.}$ Hence, for ${r}\in \left( 0,\omega _{{%
\nu },1}\right) $ we obtain
\begin{equation*}
\inf_{z\in \mathbb{D}_{r}}\left\{ \real\left( 1+\frac{zw_{\nu }^{\prime
\prime }(z)}{w_{\nu }^{\prime }(z)}\right) \right\} =1+\frac{rw_{\nu
}^{\prime \prime }(r)}{w_{\nu }^{\prime }(r)}.
\end{equation*}%
On the other hand, the function $W_{\nu }:\left( 0,\omega _{{\nu },1}\right)
\longrightarrow \mathbb{R}$, defined by%
\begin{equation*}
W_{\nu }(r)=1+\frac{rw_{\nu }^{\prime \prime }(r)}{w_{\nu }^{\prime }(r)},
\end{equation*}%
is strictly decreasing and $\lim_{r\searrow 0}W_{\nu }(r)=1>\alpha $ and $%
\lim_{r\nearrow \omega _{{\nu },1}}W_{\nu }(r)=-\infty .$ Consequently, in
view of the minimum principle for harmonic functions for $z\in \mathbb{D}%
_{r_{6}}$ we have that
\begin{equation*}
\real\left( 1+\frac{zw_{\nu }^{\prime \prime }(z)}{w_{\nu }^{\prime }(z)}%
\right) >\alpha
\end{equation*}%
if and only if $r_{6}$ is the unique root of
\begin{equation*}
1+\frac{rw_{\nu }^{\prime \prime }(r)}{w_{\nu }^{\prime }(r)}=\alpha
\end{equation*}%
situated in $\left( 0,\omega _{{\nu },1}\right) .$ Finally, the inequality $%
\omega _{{\nu },1}<j_{{\nu },1}$ follows from Lemma \ref{LP}.
\end{proof}

\begin{proof}[\bf Proof of Theorem \ref{THC1}]
We have that%
\begin{equation*}
1+\frac{zg_{{\nu }}^{\prime \prime }(z)}{g_{{\nu }}^{\prime }(z)}%
=1-\sum_{n\geq1}\frac{4z^{4}}{\zeta _{{\nu },n}^{4}-z^{4}},
\end{equation*}%
which vanishes at $r^{c}(g_{\nu })$. For all $\nu >-1$ we get
\begin{equation*}
\frac{1}{(r^{c}(g_{\nu }))^{4}}=\sum_{n\geq 1}\frac{4}{\zeta _{{\nu }%
,n}^{4}-(r^{c}(g_{\nu }))^{4}}>\sum_{n\geq 1}\frac{4}{\zeta _{{\nu },n}^{4}}%
=4\sigma _{1}=\frac{5}{4(\nu +1)_{3}},
\end{equation*}%
where $\zeta _{{\nu },n}$ is the $n$th positive zeros of $g_{\nu }^{\prime }.$ By using the Euler-Rayleigh inequalities it is possible to have more
tight bounds for the radius of convexity $r^{c}(g_{\nu }).$ For this observe that the zeros
of
\begin{equation*}
g_{\nu }^{\prime }(z)=\sum_{n\geq 0}\frac{\left( -1\right) ^{n}\left(
4n+1\right) }{2^{4n}n!\left( {\nu +1}\right) _{n}\left( {\nu +2}\right) _{2n}%
}z^{4n}
\end{equation*}%
are all real when $\nu >-1,$ since $g_{\nu}'$ belongs to the
Laguerre-P\'{o}lya class $\mathcal{LP},$ according to the proof of Theorem \ref{THbound4}. Now, since the
Laguerre-P\'{o}lya class $\mathcal{LP}$ is closed under differentiation, it
follows that the function $z\mapsto G_{\nu }(z)=\left( zg_{\nu }^{\prime
}(z)\right) ^{\prime }$ belongs also to the Laguerre-P\'{o}lya class and
hence all of its zeros are real. Thus, the function $z\mapsto \widetilde{G}%
_{\nu }(z)=G_{\nu }(\sqrt[4]{z})$ has only positive real zeros and having
growth order $\frac{1}{4}$ it can be written as the product
\begin{equation*}
\widetilde{G}_{\nu }(z)=\prod\limits_{n\geq 1}\left( 1-\frac{z}{a_{{\nu }%
,n}^{4}}\right) ,
\end{equation*}%
where $a_{\nu ,n}>0$ for each $n\in \mathbb{N}.$ Now, by using the
Euler-Rayleigh sum $\kappa_{k}=\sum_{n\geq 1}a_{{\nu },n}^{-4k}$ and the
infinite sum representation of $\widetilde{G}_{\nu }$ we have
\begin{equation*}
\frac{\widetilde{G}_{\nu }^{\prime }(z)}{\widetilde{G}_{\nu }(z)}%
=\sum_{n\geq 1}\frac{-1}{a_{{\nu },n}^{4}-z}=\sum_{n\geq 1}\sum_{k\geq 0}%
\frac{-1}{a_{{\nu },n}^{4k+4}}z^{k}=\sum_{k\geq 0}-\kappa_{k+1}z^{k},\ \
|z|<a_{{\nu },1},
\end{equation*}%
\begin{equation*}
\left. \frac{\widetilde{G}_{\nu }^{\prime }(z)}{\widetilde{G}_{\nu }(z)}%
=\sum_{n\geq 0}\frac{(-1)^{n+1}(4n+5)^{2}}{2^{4n+4}n!\left( {\nu +1}\right)
_{n+1}\left( {\nu +2}\right) _{2n+2}}z^{n}\right/ \sum_{n\geq 0}\frac{\left(
-1\right) ^{n}(4n+1)^{2}}{2^{4n}n!\left( {\nu +1}\right) _{n}\left( {\nu +2}%
\right) _{2n}}z^{n}.
\end{equation*}%
From these relations it is possible to express the Euler-Rayleigh sums in
terms of $\nu $ and by using the Euler-Rayleigh inequalities $\kappa_{k}^{-%
\frac{1}{k}}<a_{{\nu },1}^{4}<\frac{\kappa_{k}}{\kappa_{k+1}}$ we obtain
the following inequalities for $\nu >-1$ and $k\in \mathbb{N}$
\begin{equation*}
\sqrt[4]{\kappa_{k}^{-\frac{1}{k}}}<r^{c}(g_{\nu })<\sqrt[4]{\frac{\kappa_{k}}{\kappa_{k+1}}}.
\end{equation*}%
Since
\begin{equation*}
\kappa_{1}=\frac{25}{16(\nu +1)_{3}},\ \kappa_{2}=\allowbreak \frac{%
544\nu ^{2}+5301\nu +12\,257}{2^{8}(\nu +1)_{3}(\nu +1)_{5}}
\end{equation*}%
and%
\begin{equation*}
\kappa_{3}=\frac{3\left( 2112\nu ^{4}+48\,784\nu ^{3}+417\,279\nu
^{2}+1556\,904\nu +2123\,797\right) }{\allowbreak 2^{11}\left( (\nu +1)_{3}%
\right) ^{2}(\nu +1)_{7}}\allowbreak,
\end{equation*}%
in particular, when $k\in \{1,2\}$ from the above Euler-Rayleigh
inequalities we have the inequalities in this theorem and it is possible to
have more tight bounds for other values of $k\in \mathbb{N}.$
\end{proof}

\begin{proof}[\bf Proof of Theorem \ref{THC2}]
The expression
\begin{equation*}
1+\frac{zh_{{\nu }}^{\prime \prime }(z)}{h_{{\nu }}^{\prime }(z)}%
=1-\sum_{n\geq1}\frac{z}{\xi _{{\nu },n}^{4}-z}
\end{equation*}%
vanishes at $r^{c}(h_{\nu }).$ In view of the first Rayleigh sum for
the zeros of $h_{\nu }^{\prime }$ we get
\begin{equation*}
\frac{1}{r^{c}(h_{\nu })}=\sum_{n\geq 1}\frac{1}{\xi _{{\nu }%
,n}^{4}-r^{c}(h_{\nu })}>\sum_{n\geq 1}\frac{1}{\xi _{{\nu },n}^{4}}=\rho
_{1}=\frac{1}{8(\nu +1)_{3}}.
\end{equation*}
Now, we consider the infinite sum representation of $h_{\nu}^{\prime},$ that is,
\begin{equation*}
h_{\nu}^{\prime}(z)=\sum_{n\geq0}\frac{\left( -1\right) ^{n}\left(
n+1\right)}{2^{4n}n!\left({\nu+1}\right)_{n}\left({\nu +2}\right)_{2n}}%
z^{n}.
\end{equation*}%
The function $h_{\nu}'$ has only real zeros for $\nu >-1$ and
belongs to the Laguerre-P\'{o}lya class $\mathcal{LP}$ of real entire
functions, according to the proof of Theorem \ref{THbound5}. Therefore, the function
$z\mapsto \widetilde{H}_{\nu }(z)=\left(zh_{\nu}'(z)\right)'$ belongs also to the Laguerre-P\'{o}%
lya class $\mathcal{LP}$ and has also only real zeros. The function $%
\widetilde{H}_{\nu }(z)$ has only positive real zeros and having growth
order $\frac{1}{4}$ it can be written as the product
\begin{equation*}
\widetilde{H}_{\nu }(z)=\prod\limits_{n\geq 1}\left( 1-\frac{z}{b_{{\nu }%
,n}^{4}}\right) ,
\end{equation*}%
where $b_{\nu ,n}>0$ for each $n\in \mathbb{N}$. Now, by using the
Euler-Rayleigh sum $\alpha_{k}=\sum_{n\geq 1}b_{{\nu },n}^{-4k}$ and the
infinite sum representation of the function $\widetilde{H}_{\nu }$ we have%
\begin{equation*}
\frac{\widetilde{H}_{\nu }^{\prime }(z)}{\widetilde{H}_{\nu }(z)}%
=\sum_{n\geq 1}\frac{-1}{b_{{\nu },n}^{4}-z}=\sum_{n\geq 1}\sum_{k\geq 0}%
\frac{-1}{b_{{\nu },n}^{4k+4}}z^{k}=\sum_{k\geq 0}-\alpha_{k+1}z^{k},\ \
|z|<b_{{\nu },1},
\end{equation*}%
\begin{equation*}
\left. \frac{\widetilde{H}_{\nu }^{\prime }(z)}{\widetilde{H}_{\nu }(z)}%
=\sum_{n\geq 0}\frac{(-1)^{n+1}(n+2)^{2}}{2^{4n+4}n!\left( {\nu +1}\right)
_{n+1}\left( {\nu +2}\right) _{2n+2}}z^{n}\right/ \sum_{n\geq 0}\frac{\left(
-1\right) ^{n}(n+1)^{2}}{2^{4n}n!\left( {\nu +1}\right) _{n}\left( {\nu +2}%
\right) _{2n}}z^{n}.
\end{equation*}%
We can express the Euler-Rayleigh sums in terms of $\nu $ and by using the
Euler-Rayleigh inequalities $\alpha_{k}^{-\frac{1}{k}}<b_{{\nu },1}^{4}<\frac{%
\alpha_{k}}{\alpha_{k+1}}$ we get the inequalities for $b_{{\nu },1}^{4}$ when
$\nu >-1$ and $k\in \mathbb{N}$%
\begin{equation*}
\alpha_{k}^{-\frac{1}{k}}<r^{c}(h_{\nu })<\frac{\alpha_{k}}{\alpha_{k+1}}.
\end{equation*}%
Since
\begin{equation*}
\alpha_{1}=\frac{1}{4(\nu +1)_{3}},\text{ \ }\alpha_{2}=\frac{7\nu ^{2}+108\nu
+293}{2^{8}(\nu +1)_{3}(\nu +1)_{5}}
\end{equation*}%
and%
\begin{equation*}
\alpha_{3}=\frac{3\left( 3\nu ^{4}+91\nu ^{3}+1059\nu ^{2}+4965\nu
+7834\right) }{\allowbreak 2^{11}\left((\nu +1)_{3}\right)^{2}(\nu +1)_{7}}
\end{equation*}
in particular, when $k\in \{1,2\}$ we have the inequalities of this theorem.
\end{proof}

\begin{proof}[\bf Proof of Theorem \ref{THC3}]
We know that
\begin{equation*}
1+\frac{zv_{{\nu }}^{\prime \prime }(z)}{v_{{\nu }}^{\prime }(z)}%
=1-\sum_{n\geq1}\frac{4z^{4}}{\vartheta _{{\nu },n}^{4}-z^{4}},
\end{equation*}%
which vanishes at $r^{c}(v_{\nu })$. For all $\nu >-1$ we get
\begin{equation*}
\frac{1}{(r^{c}(v_{\nu }))^{4}}=\sum_{n\geq 1}\frac{4}{\vartheta _{{\nu }%
,n}^{4}-(r^{c}(v_{\nu }))^{4}}>\sum_{n\geq 1}\frac{4}{\vartheta _{{\nu }%
,n}^{4}}=4\varrho_{1}=\frac{5}{4(\nu +1)^{2}(\nu +2)},
\end{equation*}%
where $\vartheta_{{\nu },n}$ is the $n$th positive zeros of $v_{\nu }^{\prime }.$
By using the Euler-Rayleigh inequalities it is possible to have more
tight bounds for the radius convexity $r^{c}(v_{\nu }).$ For this we recall that the zeros
of
\begin{equation*}
v_{\nu }^{\prime }(z)=\sum_{n\geq 0}\frac{\left( -1\right) ^{n}\left(
4n+1\right) }{2^{4n}n!\left( {\nu +1}\right) _{n}\left( {\nu +1}\right) _{2n}%
}z^{4n}
\end{equation*}%
are all real when $\nu >-1$ since this function belongs to the
Laguerre-P\'{o}lya class $\mathcal{LP}$, according to the proof of Theorem \ref{THSbound2}. It follows that the function $%
z\mapsto \overline{V}_{\nu}(z)=\left( zv_{\nu }^{\prime }(z)\right) ^{\prime }$
belongs also to the Laguerre-P\'{o}lya class and hence all of its zeros are
real. Thus, the function $z\mapsto \widetilde{V}_{\nu }(z)=\overline{V}_{\nu }(\sqrt[4]{%
z})$ has only positive real zeros and having growth order $\frac{1}{4}$ it
can be written as the product
\begin{equation*}
\widetilde{V}_{\nu }(z)=\prod\limits_{n\geq 1}\left( 1-\frac{z}{d_{{\nu }%
,n}^{4}}\right) ,
\end{equation*}%
where $d_{\nu,n}>0$ for each $n\in \mathbb{N}.$ Now, by using the
Euler-Rayleigh sum $\epsilon _{k}=\sum_{n\geq 1}d_{{\nu },n}^{-4k}$ and the
infinite sum representation of $\widetilde{V}_{\nu }$ we have
\begin{equation*}
\frac{\widetilde{V}_{\nu }^{\prime }(z)}{\widetilde{V}_{\nu }(z)}%
=\sum_{n\geq 1}\frac{-1}{d_{{\nu },n}^{4}-z}=\sum_{n\geq 1}\sum_{k\geq 0}%
\frac{-1}{d_{{\nu },n}^{4k+4}}z^{k}=\sum_{k\geq 0}-\epsilon _{k+1}z^{k},\ \
|z|<a_{{\nu },1},
\end{equation*}%
\begin{equation*}
\left. \frac{\widetilde{V}_{\nu }^{\prime }(z)}{\widetilde{V}_{\nu }(z)}%
=\sum_{n\geq 0}\frac{(-1)^{n+1}(4n+5)^{2}}{2^{4n+4}n!\left( {\nu +1}\right)
_{n+1}\left( {\nu +1}\right) _{2n+2}}z^{n}\right/ \sum_{n\geq 0}\frac{\left(
-1\right) ^{n}(4n+1)^{2}}{2^{4n}n!\left( {\nu +1}\right) _{n}\left( {\nu +1}%
\right) _{2n}}z^{n}.
\end{equation*}%
From these relations it is possible to express the Euler-Rayleigh sums in
terms of $\nu $ and by using the Euler-Rayleigh inequalities $\epsilon
_{k}^{-\frac{1}{k}}<d_{{\nu },1}^{4}<\frac{\epsilon _{k}}{\epsilon _{k+1}}$
we obtain the following inequalities for $\nu >-1$ and $k\in \mathbb{N}$
\begin{equation*}
\sqrt[4]{\epsilon _{k}^{-\frac{1}{k}}}<r^{c}(v_{\nu })<\sqrt[4]{\frac{%
\epsilon _{k}}{\epsilon _{k+1}}}.
\end{equation*}%
Since
\begin{equation*}
\epsilon _{1}=\frac{25}{16(\nu +1)^{2}(\nu +2)},\ \epsilon _{2}=\frac{544\nu
^{2}+4213\nu +7419}{2^{8}(\nu +1)_{4}\left( \nu +1\right) ^{3}\left( \nu
+2\right) }\allowbreak
\end{equation*}%
and%
\begin{equation*}
\epsilon _{3}=\frac{6336\nu ^{5}+140\,016\nu ^{4}+1212\,429\nu
^{3}+5112\,301\nu ^{2}+10\,459\,449\nu +8300\,897}{2^{11}(\nu +1)_{3}(\nu
+1)_{6}\left( \nu +1\right) ^{4}\left( \nu +2\right) }\allowbreak ,
\end{equation*}%
in particular, when $k\in \{1,2\}$ from the above Euler-Rayleigh
inequalities we have the inequalities in this theorem and it is possible to
have more tight bounds for other values of $k\in \mathbb{N}.$
\end{proof}

\begin{proof}[\bf Proof of Theorem \protect\ref{THC4}]
We have that%
\begin{equation*}
1+\frac{zw_{{\nu }}^{\prime \prime }(z)}{w_{{\nu }}^{\prime }(z)}%
=1-\sum_{n\geq1}\frac{z}{\omega _{{\nu },n}^{4}-z},
\end{equation*}%
which vanishes at $r^{c}(w_{\nu })$. For all $\nu >-1$ we obtain
\begin{equation*}
\frac{1}{r^{c}(v_{\nu })}=\sum_{n\geq 1}\frac{1}{\omega _{{\nu }%
,n}^{4}-r^{c}(v_{\nu })}>\sum_{n\geq 1}\frac{1}{\omega _{{\nu },n}^{4}}%
=q_{1}=\frac{1}{8(\nu +1)^{2}(\nu +2)},
\end{equation*}%
where $\omega _{{\nu },n}$ is the $n$th positive zeros of $w_{\nu }^{\prime }.$
Now, we consider the infinite sum representation of $w_{\nu }^{\prime }$%
\begin{equation*}
w_{\nu }^{\prime }(z)=\sum_{n\geq0}\frac{\left( -1\right) ^{n}\left(
n+1\right) }{2^{4n}n!\left( {\nu +1}\right) _{n}\left( {\nu +1}\right) _{2n}}%
z^{n}.
\end{equation*}%
The function $w_{\nu }^{\prime }$ has only real zeros for $\nu >-1$ and
belongs to the Laguerre-P\'{o}lya class $\mathcal{LP}$ of real entire
functions, according to the proof of Theorem \ref{THSbound3}. Therefore, the function $z\mapsto \widetilde{W}_{\nu }(z)=\left(
zw_{\nu }^{\prime }(z)\right) ^{\prime }$ belongs also to the Laguerre-P\'{o}%
lya class $\mathcal{LP}$ and has also only real zeros. The function $%
\widetilde{W}_{\nu }(z)$ has only positive real zeros and having growth
order $\frac{1}{4}$ it can be written as the product
\begin{equation*}
\widetilde{W}_{\nu }(z)=\prod\limits_{n\geq 1}\left( 1-\frac{z}{\widetilde{%
\omega }_{{\nu },n}^{4}}\right) ,
\end{equation*}%
where $\widetilde{\omega }_{\nu ,n}>0$ for each $n\in \mathbb{N}$. Now, by
using the Euler-Rayleigh sum $\iota _{k}=\sum_{n\geq 1}\widetilde{\omega }_{{%
\nu },n}^{-4k}$ and the infinite sum representation of the function $%
\widetilde{W}_{\nu }$ we have%
\begin{equation*}
\frac{\widetilde{W}_{\nu }^{\prime }(z)}{\widetilde{W}_{\nu }(z)}%
=\sum_{n\geq 1}\frac{-1}{\widetilde{\omega }_{{\nu },n}^{4}-z}=\sum_{n\geq
1}\sum_{k\geq 0}\frac{-1}{\widetilde{\omega }_{{\nu },n}^{4k+4}}%
z^{k}=\sum_{k\geq 0}-\iota _{k+1}z^{k},\ \ |z|<\widetilde{\omega }_{{\nu }%
,1},
\end{equation*}%
\begin{equation*}
\left. \frac{\widetilde{W}_{\nu }^{\prime }(z)}{\widetilde{W}_{\nu }(z)}%
=\sum_{n\geq 0}\frac{(-1)^{n+1}(n+2)^{2}}{2^{4n+4}n!\left( {\nu +1}\right)
_{n+1}\left( {\nu +1}\right) _{2n+2}}z^{n}\right/ \sum_{n\geq 0}\frac{\left(
-1\right) ^{n}(n+1)^{2}}{2^{4n}n!\left( {\nu +1}\right) _{n}\left( {\nu +1}%
\right) _{2n}}z^{n}.
\end{equation*}%
We can express the Euler-Rayleigh sums in terms of $\nu $ and by using the
Euler-Rayleigh inequalities $\iota _{k}^{-\frac{1}{k}}<\widetilde{\omega }_{{%
\nu },1}^{4}<\frac{\iota _{k}}{\iota _{k+1}}$ we get the following inequalities for $%
\widetilde{\omega }_{{\nu },1}^{4}$ when $\nu >-1$ and $k\in \mathbb{N}$%
\begin{equation*}
\iota _{k}^{-\frac{1}{k}}<r^{c}(w_{\nu })<\frac{\iota _{k}}{\iota _{k+1}}.
\end{equation*}%
Since%
\begin{equation*}
\imath _{1}=\frac{1}{4(\nu +1)^{2}(\nu +2)},\ \imath _{2}=\frac{7\nu
^{2}+94\nu +183}{2^{8}(\nu +1)_{4}\left( \nu +1\right) ^{3}\left( \nu
+2\right) }
\end{equation*}%
and%
\begin{equation*}
\imath _{3}=\frac{9\nu ^{5}+264\nu ^{4}+3108\nu ^{3}+16\,222\nu
^{2}+37\,827\nu +32\,138}{2^{11}(\nu +1)_{3}(\nu +1)_{6}\left( \nu +1\right)
^{4}\left( \nu +2\right) }\allowbreak ,
\end{equation*}%
in particular, when $k\in \{1,2\}$ from the above Euler-Rayleigh
inequalities we have the inequalities in this theorem and it is possible to
have more tight bounds for other values of $k\in \mathbb{N}.$
\end{proof}

\begin{proof}[\bf Proof of Theorem \ref{T5}]
Recall, that $\nu\mapsto\gamma_{\nu,n}$ is increasing on $(-1,\infty)$ for $n\in\mathbb{N}$ fixed, see \cite[Lemma 4]{HAT}. Thus, we have that
$\gamma_{\nu,1}>1$ if and only if $\nu>\nu_{\star},$ where $\nu_{\star}\simeq-0.97{\dots}$ is the unique root of $\gamma_{\nu,1}=1,$ or
equivalently $\mathbf{\Phi}_{\nu}(1)=0,$ or equivalently $J_{\nu+1}(1)I_{\nu}(1)+J_{\nu}(1)I_{\nu+1}(1)=0.$ Consequently, we have that for $\nu>\nu_{\star}$ and $n\in\mathbb{N}$ all zeros $\gamma_{\nu,n}$ are outside of the unit disk $\mathbb{D}.$ According to (\ref{sf}), (\ref{sg}) and (\ref{sh}) for ${z\in \mathbb{D}}$, $%
{\left\vert z\right\vert =r}$ we get%
\begin{equation*}
\real \left( \frac{zf_{{\nu }}^{\prime }(z)}{f_{{\nu }}(z)}\right) \geq 1-%
\frac{1}{2{\nu +1}}\sum_{n\geq1}\frac{4r^{4}}{\gamma _{{\nu }%
,n}^{4}-r^{4}}=\frac{1}{2{\nu +1}}\frac{r\mathbf{\Phi }_{{\nu }}^{\prime }(r)%
}{\mathbf{\Phi }_{{\nu }}(r)}=\frac{rf_{{\nu }}^{\prime }(r)}{f_{{\nu }}(r)},\ \ \nu>-\frac{1}{2},
\end{equation*}%
\begin{equation*}
\real \left( \frac{zg_{{\nu }}^{\prime }(z)}{g_{{\nu }}(z)}\right) \geq
1-\sum_{n\geq1}\frac{4r^{4}}{\gamma _{{\nu },n}^{4}-r^{4}}=-2{\nu }+%
\frac{r\mathbf{\Phi }_{{\nu }}^{\prime }(r)}{\mathbf{\Phi }_{{\nu }}(r)}=%
\frac{rg_{{\nu }}^{\prime }(r)}{g_{{\nu }}(r)}, \ \ \nu>-1
\end{equation*}%
and%
\begin{equation*}
\real \left( \frac{zh_{{\nu }}^{\prime }(z)}{h_{{\nu }}(z)}\right) \geq
1-\sum_{n\geq1}\frac{r}{\gamma _{{\nu },n}^{4}-r}=\frac{3}{4}-\frac{{%
\nu }}{2}+\frac{1}{4}\frac{r^{\frac{1}{4}}\mathbf{\Phi }_{{\nu }}^{\prime }(r^{\frac{1}{4}})%
}{\mathbf{\Phi }_{{\nu }}(r^{\frac{1}{4}})}=\frac{rh_{{\nu }}^{\prime }(r)}{h_{{\nu }%
}(r)},\ \ \nu>-1,
\end{equation*}%
since $\gamma _{{\nu },n}>\gamma_{\nu ,1}>1$ for ${\nu }%
>\nu_{\star}$ and $n\in\mathbb{N}.$ Now, we have
\begin{equation*}
\frac{\partial }{\partial r}\left( \frac{rf_{{\nu }}^{\prime }(r)}{f_{{\nu }%
}(r)}\right) =-\frac{16}{2{\nu +1}}\sum_{n\geq1}\frac{\gamma _{{\nu }%
,n}^{4}r^{3}}{\left( \gamma _{{\nu },n}^{4}-r^{4}\right) ^{2}}<0,\ \ \nu>-\frac{1}{2},
\end{equation*}%
\begin{equation*}
\frac{\partial }{\partial r}\left( \frac{rg_{{\nu }}^{\prime }(r)}{g_{{\nu }%
}(r)}\right) =-16\sum_{n\geq1}\frac{\gamma _{{\nu },n}^{4}r^{3}}{%
\left( \gamma _{{\nu },n}^{4}-r^{4}\right) ^{2}}<0, \ \ \nu>-1
\end{equation*}%
and%
\begin{equation*}
\frac{\partial }{\partial r}\left( \frac{rh_{{\nu }}^{\prime }(r)}{h_{{\nu }%
}(r)}\right) =-\sum_{n\geq1}\frac{\gamma _{{\nu },n}^{4}}{\left(
\gamma _{{\nu },n}^{4}-r\right) ^{2}}<0, \ \ \nu>-1.
\end{equation*}%
These imply that the functions $r\mapsto {rf_{{\nu }}^{\prime }(r)}/{f_{%
{\nu }}(r)},$ $r\mapsto {rg_{{\nu }}^{\prime }(r)}/{g_{{\nu }}(r)}\ $%
and $r\mapsto {rh_{{\nu }}^{\prime }(r)}/{h_{{\nu }}(r)}$ are
decreasing on $(0,1)\subset (0,\gamma _{\nu ,1}).$ So, for $z\in\mathbb{D}$ and for the corresponding $\nu$ we have
\begin{equation*}
\real \left( \frac{zf_{{\nu }}^{\prime }(z)}{f_{{\nu }}(z)}\right) \geq \frac{%
rf_{{\nu }}^{\prime }(r)}{f_{{\nu }}(r)}\geq \frac{f_{{\nu }}^{\prime }(1)}{%
f_{{\nu }}(1)}=1-\frac{1}{2{\nu +1}}\sum_{n\geq1}\frac{4}{\gamma _{{%
\nu },n}^{4}-1},\ \ \nu>-\frac{1}{2},
\end{equation*}%
\begin{equation*}
\real \left( \frac{zg_{{\nu }}^{\prime }(z)}{g_{{\nu }}(z)}\right) \geq \frac{%
rg_{{\nu }}^{\prime }(r)}{g_{{\nu }}(r)}\geq \frac{g_{{\nu }}^{\prime }(1)}{%
g_{{\nu }}(1)}=1-\sum_{n\geq1}\frac{4}{\gamma _{{\nu },n}^{4}-1},\ \ \nu>-1
\end{equation*}%
and%
\begin{equation*}
\real \left( \frac{zh_{{\nu }}^{\prime }(z)}{h_{{\nu }}(z)}\right) \geq \frac{%
rh_{{\nu }}^{\prime }(r)}{h_{{\nu }}(r)}\geq \frac{h_{{\nu }}^{\prime }(1)}{%
h_{{\nu }}(1)}=1-\sum_{n\geq1}\frac{1}{\gamma _{{\nu },n}^{4}-1}, \ \ \nu>-1.
\end{equation*}%
By using again the fact that the function ${\nu }\mapsto \gamma _{{\nu },n}$ is
increasing on $\left( -1,\infty \right) $ for all fixed $n\in\mathbb{N},$ we get that
the functions ${\nu }\mapsto {f_{{\nu }}^{\prime }(1)}/{f_{{\nu
}}(1)},$ ${\nu }\mapsto {g_{{\nu }}^{\prime }(1)}/{g_{{\nu }}(1)}$
and ${\nu }\mapsto {h_{{\nu }}^{\prime }(1)}/{h_{{\nu }}(1)}$ are
also increasing on $\left(-\frac{1}{2},\infty\right)$ and $(-1,\infty),$ respectively. Therefore the following statements are valid for $0\leq
\alpha <1:$

\vskip2mm

$\bullet$ ${f_{{\nu }}^{\prime }(1)}/{f_{{\nu }}(1)}\geq\alpha $ if and only if ${\nu
\geq \nu }_{\alpha }^{\star }(f_{{\nu }})$, where ${\nu }_{\alpha }^{\star
}(f_{{\nu }})$ is the unique root of $f_{{\nu }}^{\prime }(1)=\alpha f_{{\nu }}(1).$

\vskip2mm

$\bullet$ ${g_{{\nu }}^{\prime }(1)}/{g_{{\nu }}(1)}\geq\alpha $ if and only if ${%
\nu \geq \nu }_{\alpha }^{\star }(g_{{\nu }})$, where ${\nu }_{\alpha }^{\star
}(g_{{\nu }})$ is the unique root of $g_{{\nu }}^{\prime }(1)=\alpha g_{{\nu }}(1).$

\vskip2mm

$\bullet$ ${h_{{\nu }}^{\prime }(1)}/{h_{{\nu }}(1)}\geq\alpha $ if and only if ${%
\nu \geq \nu }_{\alpha }^{\star }(h_{{\nu }})$, where ${\nu }_{\alpha }^{\star
}(h_{{\nu }})$ is the unique root of $h_{{\nu }}^{\prime }(1)=\alpha h_{{\nu }}(1).$

\vskip2mm

The above equations are equivalent to
\begin{equation*}
2\mathbf{\Pi }_{{\nu }}(1)-(\alpha (2{\nu +1)+1)}\mathbf{\Phi }_{{\nu }%
}(1)=0,
\end{equation*}
\begin{equation*}
2\mathbf{\Pi }_{{\nu }}(1)-(\alpha +2{\nu +1)}\mathbf{\Phi }_{{\nu }}(1)=0
\end{equation*}%
and%
\begin{equation*}
\mathbf{\Pi }_{{\nu }}(1)-(2\alpha +{\nu -1)}\mathbf{\Phi }_{{\nu }}(1)=0
\end{equation*}%
which are equivalent to the equations in the statements of the theorem. This completes the proof.
\end{proof}

\begin{proof}[\bf Proof of Theorem \ref{T6}]
According to (\ref{su}), (\ref{sv}) and (\ref{sw}) for ${z\in \mathbb{D}}$
and ${\left\vert z\right\vert =r,}$ we obtain%
\begin{equation*}
\real \left( \frac{zu_{{\nu }}^{\prime }(z)}{u_{{\nu }}(z)}\right) \geq 1-%
\frac{1}{2{\nu }}\sum_{n\geq1}\frac{4r^{4}}{j_{{\nu },n}^{4}-r^{4}}=%
\frac{1}{2{\nu }}\frac{r\mathbf{\Pi }_{{\nu }}^{\prime }(r)}{\mathbf{\Pi }_{{%
\nu }}(r)}=\frac{ru_{{\nu }}^{\prime }(r)}{u_{{\nu }}(r)}, \ \ \nu>0,
\end{equation*}%
\begin{equation*}
\real \left( \frac{zv_{{\nu }}^{\prime }(z)}{v_{{\nu }}(z)}\right) \geq
1-\sum_{n\geq1}\frac{4r^{4}}{j_{{\nu },n}^{4}-r^{4}}={1}-2{\nu }+%
\frac{r\mathbf{\Pi }_{{\nu }}^{\prime }(r)}{\mathbf{\Pi }_{{\nu }}(r)}=\frac{%
rv_{{\nu }}^{\prime }(r)}{v_{{\nu }}(r)}, \ \ \nu>-1
\end{equation*}%
and%
\begin{equation*}
\real \left( \frac{zw_{{\nu }}^{\prime }(z)}{w_{{\nu }}(z)}\right) \geq
1-\sum_{n\geq1}\frac{r}{j_{{\nu },n}^{4}-r}=1-\frac{{\nu }}{2}+\frac{1%
}{4}\frac{r^{\frac{1}{4}}\mathbf{\Pi }_{{\nu }}^{\prime }(r^{\frac{1}{4}})}{\mathbf{\Pi }_{{%
\nu }}(r^{\frac{1}{4}})}=\frac{rw_{{\nu }}^{\prime }(r)}{w_{{\nu }}(r)},\ \ \nu>-1,
\end{equation*}%
since $j_{{\nu },n}>j_{\nu,1}>1$ for $n\in\mathbb{N}$ and ${\nu }>\nu^{\circ}\simeq-0.77{\dots},$
where $\nu^{\circ}$ is unique root of the equation $j_{\nu,1}=1,$ see \cite{bszasz} for more details. Now, we get
\begin{equation*}
\frac{\partial }{\partial r}\left( \frac{ru_{{\nu }}^{\prime }(r)}{u_{{\nu }%
}(r)}\right) =-\frac{8}{{\nu }}\sum_{n\geq1}\frac{j_{{\nu }%
,n}^{4}r^{3}}{\left( j_{{\nu },n}^{4}-r^{4}\right) ^{2}}<0,\ \ \nu>0,
\end{equation*}%
\begin{equation*}
\frac{\partial }{\partial r}\left( \frac{rv_{{\nu }}^{\prime }(r)}{v_{{\nu }%
}(r)}\right) =-16\sum_{n\geq1}\frac{j_{{\nu },n}^{4}r^{3}}{\left( j_{{%
\nu },n}^{4}-r^{4}\right) ^{2}}<0, \ \ \nu>-1
\end{equation*}%
and%
\begin{equation*}
\frac{\partial }{\partial r}\left( \frac{rw_{{\nu }}^{\prime }(r)}{w_{{\nu }%
}(r)}\right) =-\sum_{n\geq1}\frac{j_{{\nu },n}^{4}}{\left( j_{{\nu }%
,n}^{4}-r\right) ^{2}}<0,\ \ \nu>-1.
\end{equation*}%
So, the functions $r\mapsto {ru_{{\nu }}^{\prime }(r)}/{u_{{\nu }}(r)%
},$ $r\mapsto {rv_{{\nu }}^{\prime }(r)}/{v_{{\nu }}(r)}$ and $%
r\mapsto {rw_{{\nu }}^{\prime }(r)}/{w_{{\nu }}(r)}$ are decreasing
on $(0,1)\subset (0,j_{\nu ,1}).$ Hence,%
\begin{equation*}
\real \left( \frac{zu_{{\nu }}^{\prime }(z)}{u_{{\nu }}(z)}\right) \geq \frac{%
ru_{{\nu }}^{\prime }(r)}{u_{{\nu }}(r)}\geq \frac{u_{{\nu }}^{\prime }(1)}{%
u_{{\nu }}(1)}=1-\frac{1}{2{\nu }}\sum_{n\geq1}\frac{4}{j_{{\nu }%
,n}^{4}-1}, \ \ \nu>0,
\end{equation*}%
\begin{equation*}
\real \left( \frac{zv_{{\nu }}^{\prime }(z)}{v_{{\nu }}(z)}\right) \geq \frac{%
rv_{{\nu }}^{\prime }(r)}{v_{{\nu }}(r)}\geq \frac{v_{{\nu }}^{\prime }(1)}{%
v_{{\nu }}(1)}=1-\sum_{n\geq1}\frac{4}{j_{{\nu },n}^{4}-1}, \ \ \nu>-1
\end{equation*}%
and%
\begin{equation*}
\real \left( \frac{zw_{{\nu }}^{\prime }(z)}{w_{{\nu }}(z)}\right) \geq \frac{%
rw_{{\nu }}^{\prime }(r)}{w_{{\nu }}(r)}\geq \frac{w_{{\nu }}^{\prime }(1)}{%
w_{{\nu }}(1)}=1-\sum_{n\geq1}\frac{1}{j_{{\nu },n}^{4}-1}, \ \ \nu>-1.
\end{equation*}%
Since ${\nu }\mapsto j_{{\nu },n}$ is increasing on $\left(
-1,\infty \right)$ for all fixed $n\in\mathbb{N},$ the functions ${%
\nu }\mapsto {u_{{\nu }}^{\prime }(1)}/{u_{{\nu }}(1)},$ ${\nu }%
\mapsto {v_{{\nu }}^{\prime }(1)}/{v_{{\nu }}(1)}$ and ${\nu }%
\mapsto {w_{{\nu }}^{\prime }(1)}/{w_{{\nu }}(1)}$ are also
increasing on $(0,\infty)$ and $(-1,\infty),$ respectively. Therefore the following statements are true for $0\leq \alpha
<1: $

\vskip2mm

$\bullet$ ${u_{{\nu }}^{\prime }(1)}/{u_{{\nu }}(1)}\geq\alpha $ if and only if ${\nu
\geq \nu }_{\alpha }^{\star }(u_{{\nu }})$, where ${\nu }_{\alpha }^{\star
}(u_{{\nu }})$ is the unique root of $u_{{\nu }}^{\prime }(1)=\alpha u_{{\nu }}(1).$

\vskip2mm

$\bullet$ ${v_{{\nu }}^{\prime }(1)}/{v_{{\nu }}(1)}\geq\alpha $ if and only if ${%
\nu \geq \nu }_{\alpha }^{\star }(v_{{\nu }})$, where ${\nu }_{\alpha }^{\star
}(v_{{\nu }})$ is the unique root of $v_{{\nu }}^{\prime }(1)=\alpha v_{{\nu }}(1).$

\vskip2mm

$\bullet$ ${w_{{\nu }}^{\prime }(1)}/{w_{{\nu }}(1)}\geq\alpha $ if and only if ${%
\nu \geq \nu }_{\alpha }^{\star }(w_{{\nu }})$, where ${\nu }_{\alpha }^{\star
}(w_{{\nu }})$ is the unique root of $w_{{\nu }}^{\prime }(1)=\alpha w_{{\nu }}(1).$

\vskip2mm

The above equations are equivalent to
\begin{equation*}
J_{{\nu }}(1)I_{{\nu +1}}(1)-J_{{\nu +1}}(1)I_{{\nu }}(1)+2{\nu (1-\alpha )}%
J_{{\nu }}(1)I_{{\nu }}(1)=0,
\end{equation*}
\begin{equation*}
J_{{\nu }}(1)I_{{\nu +1}}(1)-J_{{\nu +1}}(1)I_{{\nu }}(1)+{(1-\alpha )}J_{{%
\nu }}(1)I_{{\nu }}(1)=0\
\end{equation*}%
and%
\begin{equation*}
J_{{\nu }}(1)I_{{\nu +1}}(1)-J_{{\nu +1}}(1)I_{{\nu }}(1)+4{(1-\alpha )}J_{{%
\nu }}(1)I_{{\nu }}(1)=0.
\end{equation*}
\end{proof}

\begin{proof}[\bf Proof of Theorem \ref{thconvexity1}]
Since the proofs of Theorems \ref{thconvexity1} and \ref{thconvexity2} are very similar not only by nature, but also from the point
of view of computations, we will present in details only the proof of Theorem \ref{thconvexity2}. The proof of Theorem \ref{thconvexity1} goes along
the same lines as the proof of Theorem \ref{thconvexity2}, we just need to change everywhere the zeros $j_{\nu,n}$ to the zeros $\gamma_{\nu,n}$ of the cross-product. The only different thing is the inequality \eqref{c2b3mm3d4f5g}, which in the case of the zeros $\gamma_{\nu,n}$ will be the following: if $\nu>-\frac{1}{2},$  then
$$
\sum_{n\geq1}\frac{1}{\gamma^4_{\nu,n}-1}<\frac{1}{29}.
$$
This follows from the fact that if $\nu>-\frac{1}{2}$ then
$$\sum_{n\geq1}\frac{1}{\gamma^4_{\nu,n}}=\frac{1}{2^4(\nu+1)(\nu+2)(\nu+3)}\leq\frac{1}{2^4(1-\frac{1}{2})(2-\frac{1}{2})(3-\frac{1}{2})}=\frac{1}{30},$$
which implies $\gamma^4_{\nu,n}>30$ for every $\nu>-\frac{1}{2}$ and $n\in\mathbb{N},$ or equivalently $\frac{1}{\gamma^4_{\nu,n}-1}<\frac{30}{29}\frac{1}{\gamma^4_{\nu,n}},$ resulting that
$$\sum_{n\geq1}\frac{1}{\gamma^4_{\nu,n}-1}<\frac{30}{29}\sum_{n\geq1}\frac{1}{\gamma^4_{\nu,n}}<\frac{30}{29}\frac{1}{30}=\frac{1}{29}.$$
Moreover, similarly as we did before and in Lemma \ref{lemzerosBessel}, it can be shown that the
equation
$$1=\sum_{n\geq1}\frac{1}{\gamma_{\nu,n}^4-1}+\sum_{n\geq1}\frac{\gamma_{\nu,n}^4}{(\gamma_{\nu,n}^4-1)^2}$$
has a unique root and the function $$\nu\mapsto 1-\sum_{n\geq1}\frac{1}{\gamma_{\nu,n}^4-1}-\frac{\sum\limits_{n\geq1}\frac{\gamma_{\nu,n}^4}{(\gamma_{\nu,n}^4-1)^2}}{1-\sum\limits_{n\geq1}\frac{1}{\gamma_{\nu,n}^4-1}}$$
is strictly increasing on the corresponding interval.
\end{proof}

\begin{proof}[\bf Proof of Theorem \ref{thconvexity2}]
{\bf a)} We define the mapping $\omega:(0,j_{\nu,1})\rightarrow\mathbb{R},$ by $$\omega(r)=\frac{rv'_{\nu}(r)}{v_{\nu}(r)}=1-\sum_{n\geq1}\frac{4r^4}{j_{\nu,n}^4-r^4}.$$
It can be seen that the inequality (\ref{c2b3mm3d4f5g})
implies $\omega(1)>0.$  This means that $r^{\star}({v_\nu})>1.$ Thus, from the proof of part {\bf b} of Theorem \ref{T4} we infer that
$$\real\left(1+\frac{z{v_\nu''}(z)}{v'_\nu(z)}\right)\geq1+\frac{r{v_\nu''}(r)}{v'_\nu(r)}$$ holds for every $1>r\geq|z|.$
This inequality and the monotonicity of the function $V_{\nu}$ imply
$$\real\left(1+\frac{z{v_\nu''}(z)}{v'_\nu(z)}\right)\geq1+\frac{r{v_\nu''}(r)}{v'_\nu(r)}\geq1+\frac{{v_\nu''}(1)}{v'_\nu(1)}$$
for each $1>r\geq|z|.$ Let $\overline{\nu}$ be the unique root of the equation
$$1=\sum_{n\geq1}\frac{4}{j_{\nu,n}^4-1}+\sum_{n\geq1}\frac{16j_{\nu,n}^4}{(j_{\nu,n}^4-1)^2}.$$
In the proof of Lemma \ref{lemzerosBessel} we proved that the functions
$\Theta_1,\Theta_2:(\underline{\nu},\infty)\rightarrow(0,\infty),$ defined by
$$\Theta_1(\nu)=\sum_{n\geq1}\frac{1}{j_{\nu,n}^4-1},\ \Theta_2(\nu)=\frac{\sum\limits_{n\geq1}\frac{j_{\nu,n}^4}{(j_{\nu,n}^4-1)^2}}{1-\sum\limits_{n\geq1}\frac{1}{j_{\nu,n}^4-1}},$$
are strictly decreasing.  Since  $\underline{\nu}<\overline{\nu}$ it follows that $\Theta_1$ and $\Theta_2$ are increasing on the interval $(\overline{\nu},\infty)$ too. It can be shown that $\Theta_3:(\overline{\nu},\infty)\rightarrow(0,\infty),$ defined by
$$\Theta_3(\nu)=\frac{1-\Theta_1(\nu)}{1-4\Theta_1(\nu)},$$ is strictly decreasing. Consequently the mapping $\Theta_4:(\overline{\nu},\infty)\rightarrow(0,\infty),$ defined by $$\Theta_4(\nu)=\Theta_2(\nu)\Theta_3(\nu)
=\frac{\sum\limits_{n\geq1}\frac{j_{\nu,n}^4}{(j_{\nu,n}^4-1)^2}}{1-4\sum\limits_{n\geq1}\frac{1}{j_{\nu,n}^4-1}},$$ is strictly decreasing because
is a product of two strictly decreasing positive functions. Thus, we get that the function $\chi:(\overline{\nu},\infty)\rightarrow\mathbb{R},$ defined by   $$\chi(\nu)=1-4\Theta_1(\nu)-16\Theta_4(\nu)=1+\frac{{v_\nu''}(1)}{v'_\nu(1)},$$ is strictly increasing.
Since $$\chi(\overline{\nu})=1-4\Theta_1(\overline{\nu})-16\Theta_4(\overline{\nu})
<1-4\Theta_1(\overline{\nu})-\sum_{n\geq1}\frac{j_{\overline{\nu},n}^4}{(j_{\overline{\nu},n}^4-1)^2}=0$$
and $\lim\limits_{\nu\rightarrow\infty}\chi(\nu)=1$ it follows that the equation $\chi(\nu)=\alpha$ has an unique root $\nu_{\alpha}^{c}(v_{\nu})$  for every  $\alpha\in[0,1),$ and we have
$$\real\left(1+\frac{z{v_\nu''}(z)}{v'_\nu(z)}\right)\geq1+\frac{{v_{\nu_{\alpha}^{c}(v_{\nu})}''}(1)}{v'_{\nu_{\alpha}^{c}(v_{\nu})}(1)}=\alpha$$
for every $|z|<1$ and $\nu\geq\nu_{\alpha}^{c}(v_{\nu}).$

{\bf b)} First we show the following affirmation: the radius of convexity of order $\alpha$ of the functions $w_{\nu}$
is $r^{c}_{\alpha}(w_\nu)=r_6,$ where $r_6$  is the unique  positive  root  of the following equation
$$\left.\left(1
+\frac{z{w_\nu''}(z)}{w_\nu'(z)}\right)\right|_{z=r}=\alpha,$$
situated in the interval $(0,\omega_{\nu,1}).$ This affirmation is actually equivalent to part {\bf c} of Theorem \ref{T4}, however, we give
here an alternative proof. The logarithmic differentation of the equality
$$w_{\nu}(z)=z\prod_{n\geq 1}\left(1-\frac{z}{j_{\nu,n}^4}\right)$$
gives
$$\frac{z{w_\nu'}(z)}{w_\nu(z)}=1-\sum_{n\geq1}\frac{z}{j_{\nu,n}^4-z}.$$ A second logarithmic differentiation leads to
$$1+\frac{z{w_\nu''}(z)}{w'_\nu(z)}=1-\sum_{n\geq1}\frac{z}{j_{\nu,n}^4-z}-\frac{\sum\limits_{n\geq1}\frac{zj_{\nu,n}^4}{(j_{\nu,n}^4-z)^2}}
{1-\sum\limits_{n\geq1}\frac{z}{j_{\nu,n}^4-z}}.$$
Since $r^{\star}({w_\nu})$ is the smallest root of the equation $1-\sum\limits_{n\geq1}\frac{r}{j_{\nu,n}^4-r}=0,$
for $r^{\star}({w_\nu})>r\geq|z|$ we have that
\begin{align*}
\real\left(1+\frac{z{w_\nu''}(z)}{w'_\nu(z)}\right)&=1-\sum_{n\geq1}\real\frac{z}{j_{\nu,n}^4-z}-
\real\frac{\sum\limits_{n\geq1}\frac{zj_{\nu,n}^4}{(j_{\nu,n}^4-z)^2}}{1-\sum\limits_{n\geq1}\frac{z}{j_{\nu,n}^4-z}}\\
&\geq1-\sum_{n\geq1}\frac{r}{j_{\nu,n}^4-r}
-\frac{\sum\limits_{n\geq1}\left|\frac{zj_{\nu,n}^4}{(j_{\nu,n}^4-z)^2}\right|}{1-\sum\limits_{n\geq1}\left|\frac{z}{j_{\nu,n}^4-z}\right|}\\
&\geq1-\sum\limits_{n\geq1}\frac{r}{j_{\nu,n}^4-r}-\frac{\sum\limits_{n\geq1}\frac{rj_{\nu,n}^4}{(j_{\nu,n}^4-r)^2}}{1-\sum\limits_{n\geq1}\frac{r}{j_{\nu,n}^4-r}}=
1+\frac{r{w_\nu''}(r)}{w'_\nu(r)}.
\end{align*}
The rest of the proof is just the same as in the proof of part {\bf c} of Theorem \ref{T4}, so we omit the details. From the proof of part {\bf c} of Theorem \ref{T4} we also know that the function $W_{\nu}$ is strictly decreasing and together with the above inequality this implies that for $|z|<1$ we have
\begin{eqnarray}\label{d4cp5b1mkj}\real\left(1+\frac{z{w_\nu''}(z)}{w'_\nu(z)}\right)\geq1+\frac{{w_\nu''}(1)}{w'_\nu(1)}.\end{eqnarray}
On the other hand, from Lemma \ref{lemzerosBessel} we know that the function $\Theta:(\underline{\nu},\infty)\rightarrow\mathbb{R},$ defined by
$$\Theta(\nu)=1+\frac{{w_\nu''}(1)}{w'_\nu(1)}=1-\sum_{n\geq1}\frac{1}{j_{\nu,n}^4-1}-\frac{\sum\limits_{n\geq1}\frac{j_{\nu,n}^4}{(j_{\nu,n}^4-1)^2}}
{1-\sum\limits_{n\geq1}\frac{1}{j_{\nu,n}^4-1}},$$
is strictly increasing. Moreover, we have that
$$\Theta(\underline{\nu})=1-\sum_{n\geq1}\frac{1}{j_{\overline{\nu},n}^4-1}-
\frac{\sum\limits_{n\geq1}\frac{j_{\overline{\nu},n}^4}{(j_{\overline{\nu},n}^4-1)^2}}{1-\sum\limits_{n\geq1}\frac{1}{j_{\overline{\nu},n}^4-1}}
<1-\sum_{n\geq1}\frac{1}{j_{\overline{\nu},n}^4-1}-{\sum_{n\geq1}\frac{j_{\overline{\nu},n}^4}{(j_{\overline{\nu},n}^4-1)^2}}=0,$$
and $\lim\limits_{\nu\rightarrow\infty}\Theta(\nu)=1.$ Thus, the
equation $\Theta(\nu)=\alpha$ has an unique root
$\nu_{\alpha}^c(w_{\nu})$ for every fixed number $\alpha\in[0,1).$
Finally, if $\nu\geq\nu_{\alpha}^c(w_{\nu}),$  then the inequality
(\ref{d4cp5b1mkj}) and the monotonicity of the mapping $\Theta$
imply that
$$\real\left(1+\frac{z{w_\nu''}(z)}{w'_\nu(z)}\right)\geq1+\frac{{w_\nu''}(1)}{w'_\nu(1)}\geq
1+\frac{{w_{\nu_{\alpha}^c(w_{\nu})}''}(1)}{w'_{\nu_{\alpha}^c(w_{\nu})}(1)}=\alpha, \ \textrm{for all} \ \ z\in\mathbb{D}.$$
\end{proof}


\begin{thebibliography}{width}

\bibitem[\bf ABY]{aktas}
\textsc{I. Akta\c{s}, \'{A}. Baricz, N. Ya\u{g}mur},
Bounds for the radii of univalence of some special functions,
{\em Math. Inequal. Appl.} (submitted), \texttt{arXiv:1604.02649}.


\bibitem[\bf ABP16]{HAT}
\textsc{H.A. Al-Kharsani, \'{A}. Baricz, T.K. Pog\'{a}ny,}
Starlikeness of a cross-product of Bessel functions, {\em J. Math. Inequal.} 10(3) (2016) 819--827.

\bibitem[\bf AB95]{ben}
\textsc{M.S. Ashbaugh, R.D. Benguria}, On Rayleigh's conjecture for the clamped plate and its generalization to three
dimensions, {\em Duke Math. J.} 78(1) (1995) 1--17.

\bibitem[\bf BDOY16]{bdoy}
\textsc{\'{A}. Baricz, D.K. Dimitrov, H. Orhan, N. Ya\u{g}mur}%
, Radii of starlikeness of some special functions, {\em Proc. Amer. Math. Soc.} 144(8) (2016) 3355--3367.

\bibitem[\bf BDM16]{bdm}
\textsc{\'{A}. Baricz, D.K. Dimitrov, I. Mez\H{o}}, Radii of
starlikeness and convexity of some $q$-Bessel functions, {\em J. Math. Anal. Appl.} 435 (2016) 968--985.

\bibitem[\bf BKS14]{sz}
\textsc{\'{A}. Baricz, P.A. Kup\'{a}n, R. Sz\'{a}sz}, {The
radius of starlikeness of normalized Bessel functions of the first kind},
\emph{Proc. Amer. Math. Soc.} 142(6) (2014) 2019--2025.

\bibitem[\bf BPS]{cross}
\textsc{\'A. Baricz, S. Ponnusamy, S. Singh}, Cross-product
of Bessel functions: monotonicity patterns and functional inequalities, {\em Proc. Indian Sci. (Math. Sci.)} (in press), \texttt{arXiv:1507.01104}.

\bibitem[\bf BS14]{BAS}
\textsc{\'{A}. Baricz, R. Sz\'{a}sz,} The radius of convexity
of normalized Bessel functions of the first kind, \emph{Anal. Appl.} 12(5)
(2014) 485--509.

\bibitem[\bf BS16]{bszasz}
\textsc{\'A. Baricz, R. Sz\'asz}, Close-to-convexity of some special functions
and their derivatives, {\em Bull. Malays. Math. Sci. Soc.} 39(1) (2016) 427--437.

\bibitem[\bf BY]{BY}
\textsc{\'{A}. Baricz, N. Ya\u{g}mur}, Geometric properties of
some Lommel and Struve functions, {\em Ramanujan J.} (in press), doi.10.1007/s11139-015-9724-6.

\bibitem[\bf BK55]{BK}
\textsc{M. Biernacki, J. Krzy\.{z}}, On the monotonity of
certain functionals in the theory of analytic function, \emph{Ann. Univ.
Mariae Curie-Sk lodowska. Sect. A.} 9 (1955) 135--147.

\bibitem[\bf Br60]{brown}
\textsc{R.K. Brown}, Univalence of Bessel functions, {\em Proc. Amer. Math. Soc.} 11(2) (1960) 278--283.

\bibitem[\bf Br62]{brown2}
\textsc{R.K. Brown}, Univalent solutions of $W''+pW=0,$ {\em Canad. J. Math.} 14 (1962) 69--78.

\bibitem[\bf Br82]{brown3}
\textsc{R.K. Brown}, Univalence of normalized solutions of $W''(z)+p(z)W(z)=0,$
{\em Int. J. Math. Math. Sci.} 5 (1982) 459--483.

\bibitem[\bf DC09]{dimitrov}
\textsc{D.K. Dimitrov, Y. Ben Cheikh}, Laguerre polynomials as Jensen polynomials of Laguerre-P\'olya
entire functions, {\em J. Comput. Appl. Math.} 233 (2009) 703--707.

\bibitem[\bf IM95]{ismail}
\textsc{M.E.H. Ismail, M.E. Muldoon}, Bounds for the small real and purely imaginary zeros of Bessel and related functions,
{\em Methods Appl. Anal.} 2(1) (1995) 1--21.

\bibitem[\bf Je13]{Jen12}
\textsc{J.L.W.V. Jensen}, Recherches sur la th\'{e}orie
des \'{e}quations, \emph{Acta Math.} 36 (1913) 181--195.

\bibitem[\bf KK00]{kim}
\textsc{H. Ki, Y.O. Kim}, On the number of nonreal zeros of real entire functions and the Fourier-P\'olya
conjecture, {\em Duke Math. J.} 104(1) (2000) 45--73.

\bibitem[\bf KT60]{todd}
\textsc{E. Kreyszig, J. Todd},
The radius of univalence of Bessel functions, {\em Illinois J. Math.} 4 (1960) 143--149.

\bibitem[\bf Le96]{LV}
\textsc{B.Ya. Levin}, {\em Lectures on Entire Functions}, Amer. Math.
Soc., Transl. of Math. Monographs, vol. 150, 1996.

\bibitem[\bf Lo94]{Lorch}
\textsc{L. Lorch}, Monotonicity of the zeros of a
cross-product of Bessel functions, \emph{Methods Appl. Anal.} 1(1) (1994)
75--80.

\bibitem[\bf MRS62]{merkes}
\textsc{E.P. Merkes, M.S. Robertson, W.T. Scott},
On products of starlike functions, {\em Proc. Amer. Math. Soc.} 13 (1962) 960--964.

\bibitem[\bf Ne49]{nehari}
\textsc{Z. Nehari}, The Schwarzian derivative and schlicht functions, {\em Bull. Amer. Math. Soc.} 55 (1949) 545--551.

\bibitem[\bf OLBC10]{OL}
\textsc{F.W.J. Olver, D.W. Lozier, R.F. Boisvert, C.W.
Clark} (eds.), {\em NIST Handbook of Mathematical Functions}, Cambridge University
Press, Cambridge, 2010.

\bibitem[\bf {\"O\c{S}}16]{ozer}
\textsc{S. \"Ozer, T. \c{S}eng\"ul},
Stability and transitions of the second grade Poiseuille flow, {\em Physica D} 331 (2016) 71--80.

\bibitem[\bf PV97]{PV}
\textsc{S. Ponnusamy, M. Vuorinen,} Asymptotic expansions and
inequalities for hypergeometric functions, \emph{Mathematika} 44 (1997)
278--301.

\bibitem[\bf Ro54]{robertson}
\textsc{M.S. Robertson},
Schlicht solutions of $W''+pW=0,$ {\em Trans. Amer. Math. Soc.} 76 (1954) 254--274.

\bibitem[\bf Ru69]{runckel}
\textsc{H.J. Runckel}, Zeros of entire functions, {\em Trans. Amer. Math. Soc.} 143 (1969) 343--362.

\bibitem[\bf Sk54]{skov}
\textsc{H. Skovgaard}, On inequalities of the Tur\'{a}n type,
\emph{Math. Scand.} 2 (1954) 65--73.

\bibitem[\bf Wa44]{Wat}
\textsc{G.N. Watson}, \emph{A Treatise on the Theory of
Bessel Functions}, Cambridge Univ. Press, Cambridge, 1944.

\bibitem[\bf Wi62]{wilf}
\textsc{H.S. Wilf},
The radius of univalence of certain entire functions, {\em Illinois J. Math.} 6 (1962) 242--244.

\end{thebibliography}
\end{document}